\documentclass[10pt]{amsart}

\usepackage{amssymb,amsmath,amsthm}

\newtheorem{theorem}{Theorem}[section]
\newtheorem{lemma}[theorem]{Lemma}
\newtheorem{prop}[theorem]{Proposition}

\newtheorem*{remarknonum}{Remark}
\def \R{\mathbb{R}}
\def \mR{\mathbb{R}}

\def \Rn{\mathbb{R}^n}
\def \Rno{\mathbb{R}^n_0}

\def \RnIp{\mathbb{R}^{n+1}_{+}}
\def \Rnp{\mathbb{R}^{n}_{+}}

\def \a{\alpha}
\def \b{\beta}
\def \d{\delta}
\def \g{\gamma}

\def \e{\varepsilon}
\def \eps{\varepsilon}
\def \ph{\varphi}

\def \s{\sigma}

\def \L{\mathcal{L}}

\def \z{\zeta}

\def \Om{\Omega}

\def \Lap{\triangle}
\def \grad{\nabla}

\def \S{\mathcal{S}}
\def \Czinf{C^{\infty}_0}

\def \half{\frac{1}{2}}

\def \Lphe{\mathcal{L}_{\ph,\e}}

\def \Lphet{\tilde{\mathcal{L}}_{\ph, \e}}

\def \RA{N^{RA}} 
\def \AR{N^{AR}} 

\newcommand{\abs}[1]{\lvert #1 \rvert}          
\newcommand{\norm}[1]{\lVert #1 \rVert}         

\numberwithin{equation}{section}

\begin{document}

\title[Partial Data Problems for the Hodge Laplacian]{Partial Data Inverse Problems for the Hodge Laplacian}

\author[Chung]{Francis J. Chung}
\address{Department of Mathematics, University of Kentucky, Lexington, USA}

\author[Salo]{Mikko Salo}
\address{Department of Mathematics and Statistics, University of Jyv\"{a}skyl\"{a}, Jyv\"{a}skyl\"{a}, Finland}

\author[Tzou]{Leo Tzou}
\address{School of Mathematics and Statistics, University of Sydney, Sydney, Australia}



\begin{abstract}
We prove uniqueness results for a Calder\'{o}n type inverse problem for the Hodge Laplacian acting on graded forms on certain manifolds in three dimensions.  In particular, we show that partial measurements of the relative-to-absolute or absolute-to-relative boundary value maps uniquely determine a zeroth order potential.  The method is based on Carleman estimates for the Hodge Laplacian with relative or absolute boundary conditions, and on the construction of complex geometrical optics solutions which reduce the Calder\'{o}n type problem to a tomography problem for 2-tensors. The arguments in this paper allow to establish partial data results for elliptic systems that generalize the scalar results due to Kenig-Sj\"{o}strand-Uhlmann.
\end{abstract}


\maketitle

\tableofcontents

\section{Introduction} \label{sec:intro}

This article is concerned with inverse problems with partial data for elliptic systems. We first discuss the prototype for such problems, which comes from the scalar case: the inverse problem of Calder\'on asks to determine the electrical conductivity $\gamma$ of a medium $\Omega$ from electrical measurements made on its boundary. More precisely, let $\Omega \subset \mR^n$ be a bounded domain with smooth boundary and let $\gamma \in L^{\infty}(\Omega)$ satisfy $\gamma \geq c > 0$ a.e.~in $\Omega$. The full boundary measurements are given by the Dirichlet-to-Neumann map (DN map) 
$$
\Lambda_{\gamma}^{DN}: H^{1/2}(\partial \Omega) \to H^{-1/2}(\partial \Omega), \ \ f \mapsto \gamma \partial_{\nu} u|_{\partial \Omega}
$$
where $u \in H^1(\Omega)$ is the unique solution of $\mathrm{div}(\gamma \nabla u) = 0$ in $\Omega$ with $u|_{\partial \Omega} = f$, and the conormal derivative $\gamma \partial_{\nu} u|_{\partial \Omega}$ is defined in the weak sense. Equivalently, one can consider the Neumann-to-Dirichlet map (ND map) 
$$
\Lambda_{\gamma}^{ND}: H^{-1/2}_{\diamond}(\partial \Omega) \to H^{1/2}(\partial \Omega), \ \ g \mapsto v|_{\partial \Omega}
$$
where $\mathrm{div}(\gamma \nabla v) = 0$ in $\Omega$ with $\gamma \partial_{\nu} v|_{\partial \Omega} = g$, and $H^{-1/2}_{\diamond}(\partial \Omega)$ consists of those elements in $H^{-1/2}(\partial \Omega)$ that are orthogonal to constants. The inverse problem of Calder\'on asks to determine the conductivity $\gamma$ from the knowledge of the DN map or (equivalently) the ND map. There is a substantial literature on this problem, with pioneering works including \cite{F}, \cite{C}, \cite{SU}, \cite{HN}, \cite{Nachman}, \cite{No_1988}. We refer to the surveys \cite{Novikov_survey}, \cite{U_BMS} for more information.

The Calder\'on problem with partial data corresponds to the case where one can only make measurements on subsets of the boundary. Let $\Gamma_D$ and $\Gamma_N$ be open subsets of $\partial \Omega$, and assume that we measure voltages on $\Gamma_D$ and currents on $\Gamma_N$. If the potential is grounded on $\partial \Omega \setminus \Gamma_D$ but can be prescribed on $\Gamma_D$, the partial boundary measurements are given by the partial DN map 
$$
\Lambda_{\gamma}^{DN} f|_{\Gamma_N} \ \ \text{for all $f \in H^{1/2}(\partial \Omega)$ with $\mathrm{supp}(f) \subset \Gamma_D$}.
$$
If instead we can freely prescribe currents on $\Gamma_N$ but no current is input on $\partial \Omega \setminus \Gamma_N$, then we know the partial ND map 
$$
\Lambda_{\gamma}^{ND} g|_{\Gamma_D} \ \ \text{for all $g \in H^{-1/2}_{\diamond}(\partial \Omega)$ with $\mathrm{supp}(g) \subset \Gamma_N$}.
$$
The basic uniqueness question is whether a (sufficiently smooth) conductivity is determined by such boundary measurements. We remark that in the partial data case there seems to be no direct way of obtaining the partial DN map from the partial ND map or vice versa, and the two cases need to be considered separately.

By now there are many uniqueness results for the Calder\'on problem with partial data involving varying assumptions on the sets $\Gamma_D$ and $\Gamma_N$. For further information we refer to the survey \cite{KeSa_survey} for results in dimensions $n \geq 3$ and \cite{GT_survey} for the case $n=2$. We only list here some of the main results for the partial DN map:
\begin{itemize}
\item 
$n \geq 3$, $\Gamma_D$ can be possibly very small but $\Gamma_N$ has to be slightly larger than the complement of $\Gamma_D$ \cite{KSU}
\item 
$n \geq 3$, $\Gamma_D = \Gamma_N = \Gamma$ and the complement of $\Gamma$ has to be part of a hyperplane or a sphere \cite{I}
\item 
$n = 2$, $\Gamma_D = \Gamma_N = \Gamma$ can be an arbitrary open set \cite{IUY}
\item 
$n \geq 3$, $\Gamma_D = \Gamma_N = \Gamma$ and the complement of $\Gamma$ has to be (conformally) flat in one direction and a certain ray transform needs to be injective \cite{KeSa} (a special case of this was proved independently in \cite{IY_cylindrical})
\end{itemize}
The approach of \cite{KSU} is based on Carleman estimates with boundary terms and the approach of \cite{I} is based on reflection arguments. The paper \cite{KeSa} combines these two approaches and extends both. There seem to be fewer results for the partial ND map, especially in dimensions $n \geq 3$; see \cite{I} and \cite{Ch2}. In fact, in dimensions $n \geq 3$ the Carleman estimate approach for the partial ND map seems to be more involved than for the partial DN map. We remark that there are counterexamples for uniqueness when $\Gamma_D$ and $\Gamma_N$ are disjoint \cite{DKN}.

The purpose of this paper is to consider analogous partial data results for elliptic systems. In the full data case ($\Gamma_D = \Gamma_N = \partial \Omega$), many uniqueness results are available for linear elliptic systems such as the Maxwell system \cite{OPS}, \cite{KSaU}, \cite{CaroZhou}, Dirac systems \cite{NT}, \cite{SaloTzou_diracfull}, Schr\"odinger equation with Yang-Mills potentials \cite{E}, elasticity \cite{NU}, \cite{NU_erratum}, \cite{ER}, and equations in fluid flow \cite{HeckLiWang}, \cite{LiWang}. In contrast, the only earlier partial data results for such systems in dimensions $n \geq 3$ that we are aware of are \cite{COS} for the Maxwell system and \cite{SaloTzou} for the Dirac system. One reason for the lack of partial data results for systems is the fact that Carleman estimates for systems often come with boundary terms that do not seem helpful for partial data inverse problems (see \cite{Eller}, \cite{SaloTzou_diracfull} for some such estimates). 

In this paper we establish partial data results analogous to \cite{KSU} for systems involving the Hodge Laplacian for graded differential forms, on certain Riemannian manifolds in dimensions $n \geq 3$. These are elliptic systems that generalize the scalar Schr\"odinger equation $(-\Lap+q)u = 0$ and are very close to the time-harmonic Maxwell equations when $n=3$. In fact, using the results of the present paper, we have finally been able to extend the partial data result of \cite{KSU} to the Maxwell system \cite{COST}. The main technical contribution of the present paper is a Carleman estimate for the Hodge Laplacian, with limiting Carleman weights, that has boundary terms involving the relative and absolute boundary values of graded forms. The boundary terms are of such a form that allows us to carry over the Carleman estimate approach of \cite{KSU} to the Hodge Laplace system. As far as we know, this is the first analogue of \cite{KSU} for systems besides \cite{SaloTzou} which considered a very special case.

In a sense, to deal with boundary terms for systems in a flexible way, one first needs a good understanding of the different splittings of Cauchy data in the scalar case. This encompasses both the scalar DN and ND maps simultaneously, since the "relative-to-absolute" map defined in Section \ref{sec:results} generalizes both the notion of the DN and ND maps. Therefore the methods developed in \cite{Ch2} for the partial ND map, involving Fourier analysis to treat the boundary terms in Carleman estimates, will be very useful in our approach. We expect that the methods developed in this paper open the way for obtaining partial data results via Carleman estimates for various elliptic systems. This has already been achieved for Maxwell equations \cite{COST}.

The plan of this document is as follows. Section \ref{sec:intro} is the introduction, and Section \ref{sec:results} contains precise statements of the main results. Section \ref{sec:notation} collects notation and identities used throughout the paper. Sections \ref{sec:carleman1}-\ref{sec:carleman3} will be devoted to the proofs of the Carleman estimates.  In Section \ref{sec:carleman1}, we will give the basic integration by parts argument for $k$-forms and simplify the boundary terms.  In Section \ref{sec:carleman2}, we prove the Carleman estimates for $0$-forms using the arguments from ~\cite{Ch2}, ~\cite{KeSa}. We will conclude the argument in Section \ref{sec:carleman3} by showing that the Carleman estimates for graded forms follow from an induction argument, given the corresponding result for $0$-forms. In Section \ref{sec:cgo} we will construct relevant complex geometrical optics solutions, following the ideas in ~\cite{DKSaU}.  In Section \ref{sec:tensor} we will present the Green's theorem argument and give the density result based on injectivity of a tensor tomography problem which finishes the proofs of Theorems \ref{RelativeAbsoluteThm} and \ref{AbsoluteRelativeThm}.  Section \ref{sec:higher} will contain the proof of Theorem \ref{EuclideanThm} and make some remarks about the case of dimensions $n \geq 4$. Section \ref{sec:appendix} contains the proofs of some technical lemmas.


\bigskip

\noindent {\bf Acknowledgements.} \ 
F.C.~ and M.S.~ were supported in part by the Academy of Finland (Finnish Centre of Excellence in Inverse Problems Research) and an ERC Starting Grant (grant agreement no 307023), and L.T.~was partly supported by Australian Research Council (FT130101346). The authors would like to thank the organizers of the Institut Mittag-Leffler program on Inverse Problems in 2013, where part of this work was carried out. We would also like to thank an anonymous referee for helpful suggestions.

\section{Statement of results} \label{sec:results}

The results in this paper are new even in Euclidean space, but it will be convenient to state them on certain Riemannian manifolds following ~\cite{DKSaU}, ~\cite{KeSa}, ~\cite{DKLS}. Suppose that $(M_0,g_0)$ is a compact oriented manifold with smooth boundary, and consider a manifold $T = \R \times M_0$ equipped with a Riemannian metric of the form $g = c(e \oplus g_0)$, where $c$ is a smooth conformal factor and $(\R,e)$ is the real line with Euclidean metric.  A compact manifold $(M,g)$ of dimension $n \geq 3$, with boundary $\partial M$, is said to be \emph{CTA} (conformally transversally anisotropic) if it can be expressed as a submanifold of such a $T$.  A CTA manifold is called \emph{admissible} if additionally $(M_0,g_0)$ can be chosen to be simple, meaning that $\partial M_0$ is strictly convex and for any point $x \in M_0$, the exponential map $\exp_x$ is a diffeomorphism from some closed neighbourhood of $0$ in $T_x M_0$ onto $M_0$ (see ~\cite{S}). Most of the geometric notions defined here will be from \cite{T1} and we refer the reader there for a more thorough treatment of the subject.

Let $\Lambda^k M$ be the $k$th exterior power of the cotangent bundle on $M$, and let $\Lambda M$ be the corresponding graded algebra. The corresponding spaces of sections (smooth differential forms) are denoted by $\Omega^k M$ and $\Omega M$. We will define $\Lap$ to be the Hodge Laplacian on $M$, acting on graded forms:
\[
-\Lap = d\delta + \delta d.
\]
Here $d$ is the exterior derivative and $\delta$ is the codifferential (adjoint of $d$ in the $L^2$ inner product). Suppose $Q$ is an $L^{\infty}$ endomorphism of $\Lambda M$; that is, $Q$ associates to almost every point $x \in M$ a linear map $Q(x)$ from $\Lambda_x M$ to itself, and the map $x \mapsto \norm{Q(x)}$ is bounded and measurable. Later will consider continuous endomorphisms, meaning that $x \mapsto Q(x)$ is continuous in $M$. The continuity of $Q$ will simplify matters since the recovery of $Q$ from boundary measurements involves integrals over geodesics, and continuity ensures that these integrals are well defined.

We would like to consider boundary value problems for the operator $-\Lap + Q$.  In order to do this, we will define the tangential trace $t: \Omega M \rightarrow \Omega \partial M$ by
\[
t: \omega \mapsto i^{*} \omega,
\]
where $i: \partial M \rightarrow M$ is the natural inclusion map.  Then the first natural boundary value problem to consider for $-\Lap + Q$, acting on graded forms $u$, is the relative boundary problem
\begin{eqnarray*}
(-\Lap + Q)u &=& 0 \mbox{ in } M \\
          tu &=& f \mbox{ on } \partial M \\
   t\delta u &=& g \mbox{ on } \partial M.
\end{eqnarray*}
If $Q$ is such that $0$ is not an eigenvalue for this problem, then this problem has a unique solution \cite[Section 5.9]{T1} and we may define a relative-to-absolute map 
\[
\RA_Q: H^{\half}(\partial M, \Lambda \partial M) \times H^{-\half}(\partial M, \Lambda \partial M) \rightarrow H^{\half}(\partial M, \Lambda \partial M) \times H^{-\half}(\partial M, \Lambda \partial M)
\]
by
\[
\RA_Q(f,g) = (t*u, t \delta *u)
\]
where $*$ is the Hodge star operator on $M$.  

The second natural boundary value problem to consider is the absolute boundary value problem
\begin{eqnarray*}
(-\Lap + Q)u  &=& 0 \mbox{ in } M \\
          t*u &=& f \mbox{ on } \partial M\\
   t\delta *u &=& g \mbox{ on } \partial M.
\end{eqnarray*}
Assuming that $0$ is not an eigenvalue, this defines an absolute-to-relative map 
\[
\AR_Q: H^{\half}(\partial M, \Lambda \partial M) \times H^{-\half}(\partial M, \Lambda \partial M ) \rightarrow H^{\half}(\partial M, \Lambda \partial M) \times H^{-\half}(\partial M, \Lambda \partial M)
\]
by
\[
\AR_Q(f,g) = (tu, t \delta u)
\]
for appropriate $Q$. For more details on the relative and absolute boundary value problems for the Hodge Laplacian, see \cite[Section 5.9]{T1}.

These maps both give rise to a Calder\'{o}n type inverse problem, which asks if knowledge of $\RA_Q$ or $\AR_Q$ suffices to determine $Q$.  If we restrict ourselves to considering the case of zero-forms only and if $Q$ acts on zero-forms by multiplication by a function $q \in L^{\infty}(M)$, then the relative-to-absolute and absolute-to-relative maps become the DN and ND maps, respectively, for the Schr\"odinger equation 
$$
(-\Lap+q)u = 0 \ \ \text{in $M$}
$$
where $u$ is now a function on $M$ and $\Lap$ is the Laplace-Beltrami operator on functions. Our problem is therefore a generalization of the standard partial data problem for the scalar Schr\"odinger equation on a compact manifold with boundary.

Let us review some earlier results for the Schr\"odinger problem in the scalar case, in dimensions $n \geq 3$. If $M$ is Euclidean, Sylvester and Uhlmann proved that knowledge of the full DN map uniquely determines the potential $q$ in the paper ~\cite{SU}.  Versions of this problem on admissible and CTA manifolds as defined above have been considered in  ~\cite{DKSaU} and ~\cite{DKLS}.  Partial data results for the DN map have been proven in ~\cite{BU}, ~\cite{I}, and ~\cite{KSU} for the Euclidean case, and more recently in ~\cite{KeSa}, the last of which contains the previous three results and extends them to the manifold case. Improved results in the linearized case are in ~\cite{DKSjU09}. Partial data results for the ND map, analogous to the ones in ~\cite{KSU}, were proven in ~\cite{Ch2}.  Other partial data results for scalar equations with first order potentials as well were obtained in ~\cite{DKSjU} and ~\cite{Ch1}, and some of those techniques will be useful to us in this paper as well. 


For the Hodge Laplacian acting on graded forms, we are not aware of previous results dealing with the determination of a potential from the relative-to-absolute or absolute-to-relative maps.  However, ~\cite{KrLU} reconstructs a real analytic metric from these maps in the case of no potential, and ~\cite{ShaSho}, ~\cite{Sho}, ~\cite{BeSha}, and ~\cite{JL} recover various kinds of topological information about the manifold from variants of these maps, again in the case of no potential. We remark that full data problems for the Hodge Laplacian in Euclidean space can be solved in a very similar way as in the scalar case (see Section \ref{sec:higher}), but full data problems on manifolds and partial data problems even in Euclidean space are more involved.


In order to describe the main results precisely, we will define `front' and `back' sets of the boundary $\partial M$ as in ~\cite{KSU}.  If $M \subset T = \mR \times M_0$ is CTA, we can use coordinates $(x_1, x')$ where $x_1$ is the Euclidean variable, and define the function $\ph:T \rightarrow \R$ by $\ph(x_1, x') = x_1$. As discussed in ~\cite{DKSaU}, $\ph$ is a natural limiting Carleman weight in $M$. Now define 
\begin{equation*}
\begin{split}
\partial M_{+} &= \{ p \in \partial M | \partial_{\nu} \ph(p) \geq 0 \} \\
\partial M_{-} &= \{ p \in \partial M | \partial_{\nu} \ph(p) \leq 0 \}. \\
\end{split}
\end{equation*}
Then the main results of this paper are the following.

\begin{theorem}\label{RelativeAbsoluteThm}
Let $M \subset \R \times M_0$ be a three dimensional admissible manifold with conformal factor $c=1$, and let $Q_1$ and $Q_2$ be continuous endomorphisms of $\Lambda M$ such that $\RA_{Q_1}, \RA_{Q_2}$ are defined. Let $\Gamma_+ \subset \partial M$ be a neighbourhood of $\partial M_{+}$, and let $\Gamma_{-} \subset \partial M$ be a neighbourhood of $\partial M_{-}$.  Suppose that
\[
\RA_{Q_1} (f,g)|_{\Gamma_{+}} = \RA_{Q_2} (f,g)|_{\Gamma_{+}}
\]
for all $(f,g) \in H^{\half}(\partial M, \Lambda \partial M) \times H^{-\half}(\partial M, \Lambda \partial M )$ supported in $\Gamma_{-}$.  Then $Q_1 = Q_2$.  
\end{theorem}

\begin{theorem}\label{AbsoluteRelativeThm}
Let $M$ be a three dimensional admissible manifold with conformal factor $c=1$, and let $Q_1$ and $Q_2$ be continuous endomorphisms of $\Lambda M$ such that $\AR_{Q_1}, \AR_{Q_2}$ are defined. Let $\Gamma_+ \subset \partial M$ be a neighbourhood of $\partial M_{+}$, and let $\Gamma_{-} \subset \partial M$ be a neighbourhood of $\partial M_{-}$.  Suppose that
\[
\AR_{Q_1} (f,g)|_{\Gamma_{+}} = \AR_{Q_2} (f,g)|_{\Gamma_{+}}
\]
for all $(f,g) \in H^{\half}(\partial M, \Lambda \partial M) \times H^{-\half}(\partial M, \Lambda \partial M )$ supported in $\Gamma_{-}$.  Then $Q_1 = Q_2$. 
\end{theorem}   

In the case that $M$ is a domain in Euclidean space, we can also extend the results to higher dimensions.

\begin{theorem}\label{EuclideanThm}
Let $M$ be a bounded smooth domain in $\Rn$, with $n \geq 3$, and let $Q_1$ and $Q_2$ be continuous endomorphisms of $\Lambda M$ such that $\RA_{Q_1}, \RA_{Q_2}$ are defined.  Fix a unit vector $\alpha$, and let $\ph(x) = \alpha \cdot x$. Let $\Gamma_+ \subset \partial M$ be a neighbourhood of $\partial M_{+}$, and let $\Gamma_{-} \subset \partial M$ be a neighbourhood of $\partial M_{-}$.  Suppose that
\[
\RA_{Q_1} (f,g)|_{\Gamma_{+}} = \RA_{Q_2} (f,g)|_{\Gamma_{+}}
\]
for all $(f,g) \in H^{\half}(\partial M, \Lambda \partial M) \times H^{-\half}(\partial M, \Lambda \partial M )$ supported in $\Gamma_{-}$.  Then $Q_1 = Q_2$.  The same result holds if we replace the relative-to-absolute map with the absolute-to-relative one.  
\end{theorem}

Theorem \ref{RelativeAbsoluteThm} is a generalization to certain systems of the scalar partial data result of ~\cite{KSU} for the DN map, and similarly Theorem \ref{AbsoluteRelativeThm} is an extension to systems of the scalar result of ~\cite{Ch2} for the ND map. To be precise, the above theorems are stated for the linear Carleman weight and not for the logarithmic weight as in ~\cite{KSU}, ~\cite{Ch2}. This restriction comes from the lack of conformal invariance of the full Hodge Laplacian. However, in the scalar case we could use the conformal invariance of the scalar Schr\"odinger operator together with a reduction from ~\cite{KeSa} to recover the logarithmic weight results of ~\cite{KSU}, ~\cite{Ch2} from the above theorems.

The proof of Theorems \ref{RelativeAbsoluteThm} and \ref{AbsoluteRelativeThm} involve three main ingredients -- the construction of complex geometrical optics (CGO) solutions, a Green's theorem argument, and a density argument relating this inverse problem to a tensor tomography problem where one determines a tensor field from its integrals along geodesics (see Section \ref{sec:tensor}).  Both the construction of CGO solutions and the Green's theorem argument require appropriate Carleman estimates.

To describe them, we will introduce the following notation.   For a CTA manifold $M$, let $N$ be the inward pointing normal vector field along $\partial M$. We can extend $N$ to be a vector field in a neighborhood of $\partial M$ by parallel transporting along normal geodesics, and then to a vector field on $M$ by multiplying by a cutoff function. For $u \in \Omega M$ we will let
\[
u_\perp = N^\flat \wedge i_N u,
\] 
where $N^{\flat}$ is the $1$-form corresponding to $N$ and $i_N$ is the interior product, and
\[
u_{||} = u - u_\perp.
\]
Let $\nabla$ denote the Levi-Civita connection on $M$, and $\nabla'$ denote the pullback connection on the boundary.  Let 
\[
\Delta_{\ph} = e^{\frac{\ph}{h}} h^2 \Delta e^{-\frac{\ph}{h}}
\]
where $\ph$ is a limiting Carleman weight as described in \cite{DKSaU}. Note that by \cite{DKSaU} such weights exist globally if $M$ is a CTA manifold. Then the Carleman estimates are as follows.

\begin{theorem}\label{RBCCarl}
Let $M$ be a CTA manifold, and let $Q$ be a $L^{\infty}$ endomorphism of $\Lambda M$.  Define $\Gamma_+ \subset \partial M$ to be a neighbourhood of $\partial M_{+}$. Suppose $u \in H^2(M, \Lambda M)$ satisfies the boundary conditions
\begin{equation}\label{RBC}
\begin{split}
u|_{\Gamma_+} &= 0 \mbox{ to first order } \\
tu|_{\Gamma_+^c} &= 0 \\
th\delta e^{-\frac{\ph}{h}}u|_{\Gamma_+^c} &= h\sigma ti_{\nu} e^{-\frac{\ph}{h}} u
\end{split}
\end{equation}
for some smooth endomorphism $\s$ independent of $h$.
Then there exists $h_0$ such that if $0 < h< h_0$,
\[
\|(-\Delta_\ph+h^2 Q)u\|_{L^2(M)} \gtrsim h \|u\|_{H^1(M)} +h^{\half}\| u_{\perp}\|_{H^1(\Gamma_+^c)} +h^{\half}\| h\grad_{N}u_{\|}\|_{L^2(\Gamma_+^c)}.
\]
\end{theorem}

Here $H^1$ signifies the semiclassical $H^1$ space with semiclassical parameter $h$, and for instance 
$$
\norm{u}_{H^1(M)} = \norm{u}_{L^2(M)} + \norm{h \nabla u}_{L^2(M)}.
$$
The constant implied in the $\gtrsim$ sign is meant to be independent of $h$.  Note that the last boundary condition in \eqref{RBC} can be rewritten as
\[
th\delta u|_{\partial M} = -t i_{d\ph} u - h \s t i_{N} u .
\]

\begin{theorem}\label{ABCCarl}
Let $M$ be a CTA manifold, and let $Q$ be a $L^{\infty}$ endomorphism of $\Lambda M$.  Define $\Gamma_+ \subset \partial M$ to be a neighbourhood of $\partial M_{+}$. Suppose $u \in H^2(M, \Lambda M)$ satisfies the boundary conditions
\begin{equation}\label{ABC}
\begin{split}
u|_{\Gamma_+} &= 0 \mbox{ to first order } \\
t*u|_{\Gamma_+^c} &= 0 \\
th\delta *e^{-\frac{\ph}{h}}u|_{\Gamma_+^c} &= h \sigma ti_{\nu} * e^{-\frac{\ph}{h}} u
\end{split}
\end{equation}
for some smooth endomorphism $\s$ independent of $h$.
Then there exists $h_0$ such that if $0 < h < h_0$,
\[
\|(-\Delta_\ph+h^2 Q)u\|_{L^2(M)} \gtrsim h \|u\|_{H^1(M)} +h^{\half}\| u_{\|}\|_{H^1(\Gamma_+^c)} +h^{\half}\| h\grad_{N} u_{\perp}\|_{L^2(\Gamma_+^c)}.
\]
\end{theorem}

Note that Theorem \ref{ABCCarl} is actually Theorem \ref{RBCCarl} with $u$ replaced by $*u$, and vice versa.  Therefore it suffices to prove Theorem \ref{ABCCarl} only.  It is also worth noting that the Carleman estimates are proved for CTA manifolds in general, with no restriction on either the dimension, the conformal factor, or the transversal manifold $(M_0,g_0)$. Theorems \ref{RBCCarl} and \ref{ABCCarl} are extensions to the Hodge Laplace system on CTA manifolds of the scalar and Euclidean Carleman estimates in \cite{KSU, Ch2}. 

Finally, we sketch the main ideas in the proofs of the theorems and highlight the new features in our approach. The main difficulty in proving the Carleman estimates is the fact that the standard integration by parts argument, which gives a useful Carleman estimate for scalar equations with Dirichlet boundary condition \cite{KSU}, results in complicated boundary terms when one is dealing with a system of equations (see Proposition \ref{CarlemaInitialForm}). The Fourier analytic methods of \cite{Ch2} will be crucial in handling these boundary terms. We first prove Theorem \ref{ABCCarl} for $0$-forms (i.e.\ scalar equations) by adapting the Euclidean arguments of \cite{Ch2} to the manifold case. After an initial estimate for the vectorial boundary terms in Proposition \ref{ABCbndryterms}, Theorem \ref{ABCCarl} is proved for $k$-forms by induction on $k$. The proof of the Carleman estimates is long and technical, due to the work required to simplify and estimate the boundary terms.

After proving the Carleman estimates, the construction of CGO solutions proceeds as in the scalar case \cite{KSU, DKSaU} and in the full data Maxwell case \cite{KSaU}. The end result is given in Lemma \ref{lemma_cgo_form_construction}. There the amplitude in the solutions is vector valued, and later one needs to use the flexibility in choosing the components of this vector. The inverse problem is solved by inserting the CGO solutions in a standard integral identity, Lemma \ref{RBCtoDensity}. Here an unexpected feature appears: recovering the matrix potential reduces to inverting mixed Fourier/attenuated geodesic ray transforms as in the scalar case \cite{DKSaU}, but the components of the matrix turn out to depend on the geodesic along which they are integrated. We resolve this difficulty when $\dim(M) = 3$ by making use of ray transforms on tensors of order $\leq 2$ and using recent results on tensor tomography \cite{PSaU}. When the underlying space is Euclidean, we can use classical Fourier arguments and prove the uniqueness result also in dimensions $\dim(M) \geq 4$.


\section{Notation and identities}\label{sec:notation}

As stated before, the basic reference for the following facts on Riemannian geometry is \cite{T1}. Let $(M,g)$ be a smooth ($=C^{\infty}$) $n$-dimensional Riemannian manifold with or without boundary. All manifolds will be assumed to be oriented. We write $\langle v, w \rangle$ for the $g$-inner product of tangent vectors, and $\abs{v} = \langle v,v \rangle^{1/2}$ for the $g$-norm. If $x = (x_1,\ldots,x_n)$ are local coordinates and $\partial_j$ the corresponding vector fields, we write $g_{jk} = \langle \partial_j, \partial_k \rangle$ for the metric in these coordinates. The determinant of $(g_{jk})$ is denoted by $\abs{g}$, and $(g^{jk})$ is the matrix inverse of $(g_{jk})$.

We shall sometimes do computations in normal coordinates. These are coordinates $x$ defined in a neighborhood of a point $p \in M^{\text{int}}$ such that $x(p) = 0$ and geodesics through $p$ correspond to rays through the origin in the $x$ coordinates. The metric in these coordinates satisfies 
\begin{equation*}
g_{jk}(0) = \delta_{jk}, \quad \partial_l g_{jk}(0) = 0.
\end{equation*}

The Einstein convention of summing over repeated upper and lower indices will be used. We convert vector fields to $1$-forms and vice versa by the musical isomorphisms, which are given by 
\begin{align*}
(X^j \partial_j)^{\flat} = X_k \,dx^k, \quad &X_k = g_{jk} X^j, \\
(\omega_k \,dx^k)^{\sharp} = \omega^j \partial_j, \quad &\omega^j = g^{jk} \omega_k.
\end{align*}
The set of smooth $k$-forms on $M$ is denoted by $\Omega^k M$, and the graded algebra of differential forms is written as 
\begin{equation*}
\Omega M = \oplus_{k=0}^n \Omega^k M.
\end{equation*}
The set of $k$-forms with $L^2$ or $H^s$ coefficients are denoted by $L^2(M,\Lambda^k M)$ and $H^s(M,\Lambda^k M)$, respectively. Here $H^s$ for $s \in \mR$ are the usual Sobolev spaces on $M$. The inner product $\langle \,\cdot\,,\,\cdot\, \rangle$ and norm $\abs{\,\cdot\,}$ are extended to forms and more generally tensors on $M$ in the usual way, and we also extend the inner product $\langle \,\cdot\,,\,\cdot\, \rangle$ to complex valued tensors as a complex bilinear form.

Let $d: \Omega^k M \to \Omega^{k+1} M$ be the exterior derivative, and let $*: \Omega^k M \to \Omega^{n-k} M$ be the Hodge star operator. We introduce the sesquilinear inner product on $\Omega^k M$, 
\begin{equation*}
(\eta|\zeta) = \int_M \langle \eta, \bar{\zeta} \rangle \,dV = \int_M \eta \wedge * \bar{\zeta} = (*\eta | *\zeta).
\end{equation*}
Here $dV = *1 = \abs{g}^{1/2} \,dx^1 \cdots \,dx^n$ is the volume form. The codifferential $\delta: \Omega^k M \to \Omega^{k-1} M$ is defined as the formal adjoint of $d$ in the inner product on real valued forms, so that 
\begin{equation*}
(d\eta|\zeta) = (\eta|\delta \zeta), \ \ \text{for } \eta \in \Omega^{k-1} M, \zeta \in \Omega^k M \text{ compactly supported and real}.
\end{equation*}
These operators satisfy the following relations on $k$-forms in $M$:
\begin{equation*}
** = (-1)^{k(n-k)}, \quad \delta = (-1)^{k(n-k)-n+k-1} *d*.
\end{equation*}
If $X$ is a vector field, the interior product $i_X: \Omega^k M \to \Omega^{k-1} M$ is defined by 
$$
i_X \omega(Y_1,\ldots,Y_{k-1}) = \omega(X,Y_1,\ldots,Y_{k-1}).
$$
If $\xi$ is a $1$-form then the interior product $i_{\xi} = i_{\xi^{\sharp}}$ is the formal adjoint of $\xi \wedge $ in the inner product on real valued forms, and on $k$-forms it has the expression 
\begin{equation*}
i_{\xi} = (-1)^{n(k-1)} *\xi \wedge *.
\end{equation*}
The interior and exterior products interact by the formula
\[
i_{\xi} \alpha \wedge \beta = (i_{\xi} \alpha) \wedge \beta + (-1)^k\alpha \wedge i_{\xi} \beta,
\]
where $\alpha$ is a $k$-form and $\beta$ an $m$-form. In particular if $\alpha$ and $\xi$ are one-forms then
\[
i_{\xi} \alpha \wedge \beta + \alpha \wedge i_{\xi} \beta = \langle \alpha, \xi \rangle \beta.
\]
In addition, the differential and codifferential satisfy the following product rules:
\[
d(f\eta) = df \wedge \eta + f d\eta 
\]
and
\[
\delta(f\eta) = -i_{df} \eta + f \delta \eta.
\]

The Hodge Laplacian on $k$-forms is defined by 
\begin{equation*}
-\Delta = (d+\delta)^2 = d\delta + \delta d.
\end{equation*}
It satisfies $\Delta * = * \Delta$. The above quantities may be naturally extended to graded forms.

We will also have to deal with forms that are not compactly supported on $M$.  We have already introduced the tangential trace $t: \Omega M \rightarrow \Omega \partial M$ by
\[
t: \omega \mapsto i^{*} \omega,
\]
so if $u$ is a graded form on $M$, then $tu$ is a graded form on $\partial M$.  Then
\[
(tu|tv)_{\partial M} 
\]
is interpreted in the same manner as $(u|v)_M$ above.  If $u$ and $v$ are graded forms on $M$, we will also define
\[
(u|v)_{\partial M} = \int_{\partial M} \langle u, \bar{v} \rangle dS = \int_{\partial M} ti_{\nu} u \wedge * \bar{v} \, dS
\]
where $dS$ is the volume form on $\partial M$.  Now if $\eta \in \Omega^{k-1} M$ and $\zeta \in \Omega^k M$ then $d$ and $\delta$ satisfy the following integration by parts formulas.
\begin{align}
 & (d\eta|\zeta)_{M} = (\nu \wedge \eta|\zeta)_{\partial M} + (\eta|\delta \zeta)_{M}, \label{iparts1} \\
 & (\delta \zeta|\eta)_{M} = -(i_{\nu} \zeta| \eta)_{\partial M} + (\zeta|d\eta)_{M}. \label{iparts2}
\end{align}
Note also that 
\[
(i_{\nu} \zeta| \eta)_{\partial M} = (\nu \wedge \eta|\zeta)_{\partial M}.
\]
Here $\nu$ denotes both the unit outer normal of $\partial M$ and the corresponding $1$-form.

Applying these formulas for the Hodge Laplacian gives
\begin{eqnarray*}
(-\Delta u|v)_M &=& (u| -\Delta v)_M  \\
                & & + (\nu \wedge \delta u| v)_{\partial M} - (i_{\nu} du|v)_{\partial M} - (i_{\nu} u| dv)_{\partial M} + (\nu \wedge u| \delta v)_{\partial M}
\end{eqnarray*}
where $u$ and $v$ are $k$-forms, or graded forms.  We can also redo the integration by parts to write the boundary terms in terms of absolute and relative boundary conditions, so 
\begin{eqnarray*}
(-\Delta u| v)_M &=& (u| -\Delta v)_M +(t u | ti_{\nu}dv)_{\partial M}\\
                 & & +(t \delta *u|ti_{\nu} *v)_{\partial M} + (t *u|ti_{\nu} d*v)_{\partial M} + (t\delta u|ti_{\nu}v)_{\partial M}.
\end{eqnarray*}

The Levi-Civita connection, defined on tensors in $M$, is denoted by $\nabla$ and it satisfies $\nabla_X * = * \nabla_X$. We will sometimes write $\nabla f$ (where $f$ is any function) for the metric gradient of $f$, defined by 
\begin{equation*}
\nabla f = (df)^{\sharp} = g^{jk} \partial_j f \partial_k.
\end{equation*}
If $X$ is a vector field and $\eta$, $\zeta$ are differential forms we have 
$$
\nabla_X (\eta \wedge \zeta) = (\nabla_X \eta) \wedge \zeta + \eta \wedge (\nabla_X \zeta).
$$
If $X, Y$ are vector fields then 
$$
[\nabla_X, i_Y] = i_{\nabla_X Y}.
$$

We can also express $d$ using the $\nabla$ operator, as follows:  if $\omega$ is a $k$-form on $M$, and $X_1,..,X_{k+1}$ are vector fields on $M$, then
\[
d\omega (X_1,..,X_{k+1}) = \sum\limits_{l= 1}^{k+1} (-1)^{l+1}( \nabla_{X_l} \omega)(X_1,..,\hat X_l,..,X_{k+1})
\]
where $\hat X_l$ means that we omit the $X_l$ argument. Moreover if $e_1, \ldots, e_n$ are an orthonormal frame of $TM$ defined in a neighborhood $U\subset M$ we have
\[
-\delta \omega = \sum\limits_{j= 1}^n i_{e_j} \nabla_{e_j} \omega.
\]

For the statements of the Carleman estimates, we introduced the notation
\[
u_\perp = N^\flat \wedge i_N u
\] 
and
\[
u_{||} = u - u_\perp
\]
where $N$ is a smooth vector field which coincides with the inward pointing normal vector field at the boundary $\partial M$, and is extended into $M$ by parallel transport.  Note that $i_N u_{\|} = 0$, $N\wedge u_{\perp} = 0$, and $tu_{\perp} = 0$ at $\partial M$.  In addition, if $u$ and $v$ are graded forms on $M$, then
\[
(tu|tv)_{\partial M} = (tu_{\|}|tv_{\|})_{\partial M} = (u_{\|}|v_{\|})_{\partial M}
\]
and
\[
(ti_N u|ti_N v)_{\partial M} = (ti_N u_{\perp}|ti_N v_{\perp})_{\partial M} = (u_{\perp}|v_{\perp})_{\partial M}.
\]
If $X$ is a vector field, we can break down $X$ into parallel and perpendicular components in the same way by using $(X^{\flat}_{\|})^{\sharp}$ and $(X^{\flat}_{\perp})^{\sharp}$.  The $\perp$ and $\|$ signs are interchanged by the Hodge star operator:
\[
*(u_{\|}) = (*u)_{\perp} \text{ and } *(u_{\perp}) = (*u)_{\|}.
\]
Note that by its definition in terms of parallel transport, $\nabla_N N = 0$.  Thus $\nabla_N$ commutes with $N \wedge $ and $i_N$.  

If we view $\partial M$ as a submanifold embedded into $M$, then $TM$ splits into $T\partial M \oplus N\partial M$, where $T\partial M$ is the tangent bundle of $\partial M$ and $N\partial M$ is the normal bundle.  Then the second fundamental form $II : T\partial M \oplus T\partial M \to N\partial M$ of $\partial M$ relative to this embedding is defined by
\[
II(X,Y) = (\nabla_X Y|N)N.
\]
The second fundamental form can also be defined in terms of the shape operator $s: T\partial M \rightarrow T\partial M$ by 
\[
s(X) = \nabla_X N.
\]
Then
\[
II(X,Y) = (s(X)|Y)N.
\]
These two operators carry information about the shape of the $\partial M$ in $M$, and thus show up in our boundary computations.  

Now we move to some more specific technical formulas used in the paper.  The proofs involve routine computations and are given in Section \ref{sec:appendix}. We begin with a simple computation.

\begin{lemma} \label{lemma_xieta_product}
If $\xi$ and $\eta$ are real valued $1$-forms on $M$ and if $u$ is a $k$-form, then 
$$
\xi \wedge i_{\eta} u + i_{\xi} (\eta \wedge u) + \eta \wedge i_{\xi} u + i_{\eta}(\xi \wedge u) = 2 \langle \xi, \eta \rangle u.
$$
\end{lemma}

We also give an expression for the conjugated Laplacian.

\begin{lemma} \label{lemma_conjugated_hodge_expression}
Let $(M,g)$ be an oriented Riemannian manifold, let $\rho \in C^2(M)$ be a complex valued function, and let $s$ be a complex number. If $u$ is a $k$-form on $M$, then 
$$
e^{s\rho} (-\Delta) (e^{-s\rho} u) = -s^2 \langle d\rho, d\rho \rangle u + s \left[ 2\nabla_{{\rm grad}(\rho)} + \Delta \rho \right] u - \Delta u.
$$
\end{lemma}

Next, an expansion for the expression $t\delta$.  

\begin{lemma}\label{Tdelta}
Let $u \in \Om^k(M)$.  Then
\[
-t (\delta u) = -\delta ' tu_{\|} + (S - (n-1)\kappa)ti_{N}u_{\perp} + t \grad_{N} i_N u,
\]
where $\kappa$ is the mean curvature of $\partial M$, and $S:\Om^{k-1}(\partial M) \rightarrow \Om^{k-1}(\partial M)$ is defined by 
\[
S\omega(X_1, \ldots, X_{k-1}) = \sum_{\ell = 1}^{k-1}\omega(X_1, \ldots, s X_{\ell}, \ldots X_{k-1}),
\]
with $s:T \partial M \rightarrow T\partial M$ being the shape operator of $\partial M$.  
\end{lemma}

Now for $ti_N d$.  

\begin{lemma}\label{iNdu}
Let $u \in \Om^k(M)$.  Then on $\partial M$,
\[
t i_N d u = t \nabla_{N} u_{\|} + St u_{\|} - d'ti_N u.
\]
\end{lemma}

We also need an expansion for $t\delta B$, where $B$ is the operator 
\begin{equation*}
\begin{split}
B &= \frac{h}{i} \left[ d \circ i_{d\varphi_c} + i_{d\varphi_c} \circ d - d\varphi_c \wedge \delta - \delta(d\varphi_c \wedge \,\cdot\,) \right] \\
  &= \frac{h}{i} \left[ 2 \nabla_{\grad \varphi_c} + \Delta \varphi_c \right].
\end{split}
\end{equation*}

\begin{lemma}\label{deltaBu}
If $u \in \Om^k(M)$ is such that $tu = 0$, then
\begin{eqnarray*}
& & t \delta Bu \\
&=& \delta ' tBu + 2ih\grad_{(\grad \ph_c)_{\|}}' t \grad_{N}i_N u -2ih\partial_{\nu}\ph_c t\grad_{N} \grad_{N}i_N u\\
& & + ih(2((n-1)\kappa-S)\partial_{\nu} \ph_{c} + 2\partial_{\nu}^2\ph_c +\Lap\ph_c)t\grad_{N} i_N u + 2ih(S-(n-1)\kappa) t \grad_{(\grad \ph_c)_{\|}} i_N u \\
& & + ih((S - (n-1)\kappa) \Lap \ph_c + \grad_{N} \Lap \ph_c)  t i_N u \\
& & + 2ihti_N R(N,\grad(\ph_c)_{||} ) u_\perp+2iht\grad_{[(\grad \ph_c)_{\|},N]}i_N u -2ihi_{s(\grad \ph_c)_{\|}} t \grad_{N} u_{\|}. \\
\end{eqnarray*}
\end{lemma}

Finally, we will need to do a computation to split the Hodge Laplacian into normal and tangential parts.  To do this, we will take advantage of a Weitzenbock identity, which says that
\[
\Lap = \tilde{\Lap} + R
\]  
where $R$ is a zeroth order linear operator depending only on the curvature of $M$, $\Lap$ is the Hodge Laplacian, and $\tilde{\Lap}$ is the connection Laplacian:
\[
\tilde \Delta u := \nabla^*\nabla u.
\] 
We then have the following result for $\tilde \Delta$.  
\begin{lemma} \label{ConxnLaplace}
Let $u \in \Om^k(M)$ satisfy $t u =0$.  Then
\[
t i_N \tilde\Delta u = \tilde \Delta' t i_Nu + t \nabla_{N}\nabla_N i_N u+tr(s^2)i_Nu -S_2i_Nu
\]
where $S_2 \omega (X_1,..,X_{k-1}) := \sum\limits_{l=1}^{k-1}\omega(.., s^2 X_l, ..)$.
\end{lemma}

\section{Carleman estimates and boundary terms} \label{sec:carleman1}

As noted in the introduction, Theorem \ref{RBCCarl} follows from Theorem \ref{ABCCarl}, so it suffice to show that we can prove Theorem \ref{ABCCarl}.

In proving the Carleman estimates, it will suffice to work with smooth sections of $\Lambda M$ and apply a density argument to get the final result.  Let $\Om^k(M)$ denote the space of smooth sections of $\Lambda^k M$, and $\Om(M)$ denote the space of smooth sections of $\Lambda M$.  

In this section we give an initial form of the Carleman estimates by using an integration by parts argument as in ~\cite{KSU}.  To do this, we will first need to understand the relevant boundary terms.  We will use the integration by parts formulas 
\begin{align}
 & (du|v)_{M} = (\nu \wedge u|v)_{\partial M} + (u|\delta v)_{M},  \\
 & (\delta u|v)_{M} = -(i_{\nu} u|v)_{\partial M} + (u|dv)_{M} 
\end{align}
for $u,v \in \Om(M)$.

As in ~\cite{KSU}, we will need to work with the convexified weight
\[
\ph_c = \ph + \frac{h\ph^2}{2\e}.
\]

Then
\begin{equation*}
-\Delta_{\ph_c} = e^{\frac{\ph_c}{h}} (-h^2 \Delta) e^{-\frac{\ph_c}{h}}.
\end{equation*}
Writing 
\begin{align*}
d_{\ph_c} &= e^{\frac{\ph_c}{h}} h d e^{-\frac{\ph_c}{h}} = hd - d \ph_c \wedge, \\
\delta_{\ph_c} &= e^{\frac{\ph_c}{h}} h \delta e^{-\frac{\ph_c}{h}} = h \delta + i_{d\ph_c},
\end{align*}
we have 
\begin{align*}
-\Delta_{\ph_c} = d_{\ph_c} \delta_{\ph_c} + \delta_{\ph_c} d_{\ph_c}.
\end{align*}
By Lemma \ref{lemma_conjugated_hodge_expression} we can write this as $A+iB$ where $A$ and $B$ are self-adjoint operators given by 
\begin{align*}
A &= -h^2 \Delta - (d\varphi_c \wedge i_{d\varphi_c} + i_{d\varphi_c} (d\varphi_c \wedge \,\cdot\,) ) \\
  &= -h^2 \Delta - \abs{d\varphi_c}^2, \\
B &= \frac{h}{i} \left[ d \circ i_{d\varphi_c} + i_{d\varphi_c} \circ d - d\varphi_c \wedge \delta - \delta(d\varphi_c \wedge \,\cdot\,) \right] \\
  &= \frac{h}{i} \left[ 2 \nabla_{\grad \varphi_c} + \Delta \varphi_c \right].
\end{align*}
Let $\| \cdot \|$ indicate the $L^2$ norm on $M$, unless otherwise stated.  Then, for $u \in \Omega^k(M)$,  
\begin{align*}
\norm{\Delta_{\ph_c} u}^2 &= ((A+iB)u|(A+iB)u) \\
 &= \norm{Au}^2 + \norm{Bu}^2 + i(Bu|Au) - i(Au|Bu).
\end{align*}
Integrating by parts gives 
\begin{align*}
 & (Bu|Au) = (Bu|h^2 d\delta u + h^2 \delta du - \abs{d\varphi_c}^2 u) = (hdBu|hdu) + (h \delta Bu| h\delta u) \\
 &\qquad - (\abs{d\varphi_c}^2 Bu|u) + h(Bu|\nu \wedge h \delta u - i_{\nu} h du)_{\partial M} \\
 &= (A B u|u) + h(hdBu|\nu \wedge u)_{\partial M} - h(h\delta Bu|i_{\nu} u)_{\partial M} + h(Bu|\nu \wedge h\delta u - i_{\nu} hdu)_{\partial M}
\end{align*}
and after a short computation 
\begin{align*}
(Au|Bu) = (B A u|u) - \frac{2h}{i} ((\partial_{\nu} \varphi_c) Au|u)_{\partial M}.
\end{align*}
This finishes the basic integration by parts argument and shows the following:

\begin{prop}\label{CarlemaInitialForm}
If $u \in \Omega M$, then 
\begin{equation}\label{IbyPidentity}
\begin{split}
\norm{\Delta_{\ph_c} u}^2 = &\norm{Au}^2 + \norm{Bu}^2 + (i[A,B]u|u) \\
 &+ ih(hdBu|\nu \wedge u)_{\partial M} - ih(h\delta Bu|i_{\nu} u)_{\partial M} + ih(Bu|\nu \wedge h\delta u - i_{\nu} hdu)_{\partial M} \\
 &+ 2h((\partial_{\nu} \varphi_c) Au|u)_{\partial M}.
\end{split}
\end{equation}
\end{prop}

Now we invoke the absolute boundary conditions to estimate the non-boundary terms and to simplify the boundary terms in \eqref{IbyPidentity}. It is enough to consider differential forms $u \in \Omega^k(M)$ for fixed $k$.

\begin{prop}\label{ABCbndryterms}
Let $u \in \Om^k(M)$ such that 
\begin{equation}\label{convexABC}
\begin{split}
t*u &= 0 \\
th\delta *u  &= -t i_{d\ph} *u +h \s t i_{N} *u .
\end{split}
\end{equation}
for some smooth bounded endomorphism $\s$ whose bounds are uniform in $h$.

Then the non-boundary terms in \eqref{IbyPidentity} satisfy 
\begin{equation}\label{NonBndryTermsABCfirst}
\norm{Au}^2 + \norm{Bu}^2 + (i[A,B]u|u) \gtrsim \frac{h^2}{\e}\|u\|_{H^1(M)}^2 - \frac{h^3}{\e}(\|u_{\|}\|_{H^1(\partial M)}^2 + \|h\nabla_{N} u_{\perp}\|_{L^2(\partial M)}^2)
\end{equation}
for $h \ll \e \ll 1$. Also, the boundary terms in \eqref{IbyPidentity} have the form
\begin{equation}\label{BndryTermsABCfirst}
-2h^3 (\partial_\nu \ph \nabla_N u_{\perp}|\nabla_N u_{\perp})_{\partial M}-  2h(\partial_\nu \ph (|d\ph|^2 + |\partial_\nu \ph|^2)u_{\|}|u_{\|})_{\partial M} + R  
\end{equation}
where 
\[
|R| \lesssim Kh^3 \|\nabla' t u_{\|}\|^2_{\partial M} + \frac{h}{K}\| u_{\|}\|^2_{\partial M} + \frac{h^3}{K}\|\nabla_N u_{\perp}\|^2_{\partial M} .
\]
for any large enough $K$ independent of $h$.  
\end{prop}

\begin{proof}[Proof of Proposition \ref{ABCbndryterms}]

We will prove \eqref{NonBndryTermsABCfirst} first.  The argument follows the one given in ~\cite{Ch2}, for scalar functions.  

Note that $A$ and $B$ have the same scalar principal symbols as they do for zero-forms: that is, given a local basis $dx^1, \ldots, dx^n$ for the cotangent space with $dx^I = dx^{i_1} \wedge \ldots \wedge dx^{i_k}$, 
\begin{equation*} 
A = A_s + h E_1, \qquad A_s(f dx^I) = (Af) dx^I
\end{equation*}
and
\[
B = B_s + h E_0, \qquad B_s(f dx^I) = (Bf) dx^I,
\]
where $E_1$ and $E_0$ are first and zeroth order operators, respectively, with uniform bounds in $h$ and $\e$.  Therefore locally 
\[
[A,B](f dx^I) = ([A,B]f) dx^I + h([E_1,B_s] + [A_s,E_0] +hR)(f dx^I)
\]
where $R$ is a first order operator with uniform bounds in $h$ and $\e$. Choosing a partition of unity $\chi_1, \ldots, \chi_m$ of $M$ such that this operation can be performed near each $\mathrm{supp}(\chi_j)$, the argument for scalar functions in the proof of Proposition 3.1 from ~\cite{Ch2} implies that 
\[
i([A,B]u|u) = \sum_{j=1}^m i([A,B]u| \chi_j u) = 4\frac{h^2}{\e}\| (1 + h\e^{-1} \ph) u \|_{L^2}^2 + h(B \beta B u| u) + h^2 (Qu|u),
\]
where $Q$ is a second order operator.  Recall that
\[
B = \frac{h}{i} \left( d \circ i_{d\varphi_c} + i_{d\varphi_c} \circ d - d\varphi_c \wedge \delta - \delta(d\varphi_c \wedge \,\cdot\,) \right),
\]
so using integration by parts with the above formula, we get 
\begin{eqnarray*}
h(B \beta Bu|u) &=& h(\beta Bu|Bu) - ih^2(i_{\nu} \beta Bu| i_{d\varphi_c} u)_{\partial M} -ih^2(\nu \wedge i_{d\varphi_c}\beta Bu| u )_{\partial M}\\
                & & -ih^2(\nu \wedge \beta Bu|d\varphi_c \wedge u)_{\partial M} -ih^2(i_{\nu}(d\varphi_c \wedge \beta Bu)|u)_{\partial M}\\
                &=& h(\beta Bu|Bu) - ih^2(d\varphi_c \wedge i_{\nu} \beta Bu|  u)_{\partial M} -ih^2(\nu \wedge i_{d\varphi_c}\beta Bu| u )_{\partial M}\\
                & & -ih^2(i_{d\varphi_c} \nu \wedge \beta Bu| u)_{\partial M} -ih^2(i_{\nu}(d\varphi_c \wedge \beta Bu)|u)_{\partial M}.\\
\end{eqnarray*}
By Lemma \ref{lemma_xieta_product} we obtain 
\[
h(B \beta Bu|u) = h(\beta Bu|Bu) - 2ih^2(\partial_{\nu} \varphi_c \beta Bu| u)_{\partial M}.
\]
The absolute boundary condition says that $t*u = 0$, so $u_{\perp} = 0$ at the boundary.  Therefore
\begin{eqnarray*}
h(B \beta Bu|u) &=& h(\beta Bu|Bu) - 2ih^2(\partial_{\nu} \varphi_c \beta Bu| u_{\|})_{\partial M} \\
                &=& h(\beta Bu|Bu) - 2ih^2(t\partial_{\nu} \varphi_c \beta Bu| tu_{\|})_{\partial M}. \\
\end{eqnarray*}
The boundary term in the last expression is bounded by 
\[
h^3\e^{-1}\| tBu\|_{L^2(\partial M)}^2 + h^3\e^{-1}\| u_{\|}\|_{L^2(\partial M)}^2.
\]
At the boundary,
\begin{eqnarray*}
tBu &=& \frac{h}{i} t\left[ 2 \nabla_{\grad \varphi_c} + \Delta \varphi_c \right]u \\
    &=& \frac{h}{i} \left[ -2 \partial_{\nu} \varphi_c t\nabla_{N}u_{\|} -2 \partial_{\nu} \varphi_c t\nabla_{N}u_{\perp}+ 2t\nabla_{(\grad \varphi_c)_{\|}}u_{\|} + \Delta \varphi_c tu_{\|} \right], \\
\end{eqnarray*}
so
\begin{eqnarray*}
\|tBu\|_{L^2(\partial M)}^2 &\lesssim& \|th\nabla_{N}u_{\|} \|_{L^2(\partial M)}^2 +\|th\nabla_{N}u_{\perp} \|_{L^2(\partial M)}^2 \\
                            & &        + \|th\nabla_{(\grad \varphi_c)_{\|}}u_{\|}\|_{L^2(\partial M)}^2 + h^2\|tu_{\|}\|_{L^2(\partial M)}^2. \\
                            &\lesssim& \|th\nabla_{N} u_{\|} \|_{L^2(\partial M)}^2 +\|th\nabla_{N}u_{\perp} \|_{L^2(\partial M)}^2+ \|u_{\|}\|_{H^1(\partial M)}^2. \\
\end{eqnarray*}
Now by Lemma \ref{iNdu},
\[
t i_N hd u = t h\nabla_{N} u_{\|} + hSt u_{\|} - hd'ti_N u .
\]
Since $t*u=0$, we have $i_Nu, u_{\perp} = 0$ at the boundary, and thus
\[
t i_N hd u = t h\nabla_{N} u_{\|} + hSt u_{\|}.
\]
Therefore
\begin{eqnarray*}
\|th\nabla_{N}u_{\|}\|_{L^2(\partial M)}^2  &\lesssim& \|t i_N hd u\|_{L^2(\partial M)}^2+ h^2\|u_{\|}\|_{L^2(\partial M)}^2 \\
                                           &\lesssim& \|t i_N * (h\delta *u)\|_{L^2(\partial M)}^2+ h^2\|u_{\|}\|_{L^2(\partial M)}^2 \\
                                           &\lesssim& \|t h\delta *u\|_{L^2(\partial M)}^2+ h^2\|u_{\|}\|_{L^2(\partial M)}^2 \\
                                           &\lesssim& \|u\|_{L^2(\partial M)}^2
\end{eqnarray*}
where in the last step we invoked the absolute boundary condition.  Therefore
\[
\|tBu\|_{L^2(\partial M)}^2 \lesssim \|th\nabla_{N}u_{\perp} \|_{L^2(\partial M)}^2+ \|u_{\|}\|_{H^1(\partial M)}^2,
\]
and thus 
\[
h(B \beta Bu|u) \lesssim \frac{h^2}{\e}\|Bu\|_{L^2}^2 + \frac{h^3}{\e}\|u_{\|}\|_{H^1(\partial M)}^2 + \frac{h^3}{\e}\|h\nabla_{N}u_{\perp}\|_{L^2(\partial M)}^2. 
\]
Similarly
\[
h^2 (Qu|u) \lesssim h^2\|u\|_{H^1}^2 + h^3\|u_{\|}\|_{H^1(\partial M)}^2 + h^3\|h\nabla_{N}u_{\perp}\|_{L^2(\partial M)}^2.
\]
Therefore
\begin{eqnarray*}
& & i([A,B]u|u) \gtrsim  \\
& &\frac{h^2}{\e}\|u\|_{L^2}^2 - \frac{h^2}{\e}\|Bu\|_{L^2}^2 - h^2\|u\|_{H^1}^2 - h^3\e^{-1}\|u_{\|}\|_{H^1(\partial M)}^2 - h^3\e^{-1}\|h\nabla_{N} u_{\perp}\|_{L^2(\partial M)}^2. \\
\end{eqnarray*}

Meanwhile, since $t*u = 0$ on $\partial M$ we can write
\begin{eqnarray*}
& & h^2(\|hdu\|_{L^2}^2 +\|h\delta u\|_{L^2}^2) = h^2( (hd*u, hd*u) + (h\delta * u, h\delta * u) ) \\
&=& h^2(-h^2\Lap *u| *u) - h^3(\nu \wedge h\delta *u| *u)_{\partial M} \\
&=& h^2(Au|u)+h^2(|d\ph_c|^2 u|u) -h^3(\nu \wedge h\delta *u| *u)_{\partial M} \\
&=& h^2(Au|u)+h^2(|d\ph_c|^2 u|u) +h^3(th\delta *u| ti_N *u)_{\partial M}.\\
\end{eqnarray*} 
Using the absolute boundary conditions again, we have 
\begin{eqnarray*}
th\delta *u &=& -t i_{d\ph} *u +h \s t i_{N} *u \\
            &=& \partial_{\nu} \ph t i_{N} *u +h \s t i_{N} *u, \\
\end{eqnarray*}
so
\[
h^2(\|hdu\|_{L^2}^2 +\|h\delta u\|_{L^2}^2) \lesssim \frac{1}{K}\|Au\|_{L^2}^2 + Kh^4\|u\|_{L^2}^2 +h^2\|u\|_{L^2}^2+ h^3\|ti_N*u\|_{L^2(\partial M)}^2,
\]
or 
\[
\|Au\|_{L^2}^2 \gtrsim Kh^2(\|hdu\|_{L^2}^2 + \|h\delta u\|_{L^2}^2) - K^2 h^4\|u\|_{L^2}^2 - Kh^2\|u\|_{L^2}^2 - Kh^3\|u_{\|}\|_{L^2(\partial M)}^2. 
\]
We take $K \sim \frac{1}{\alpha \e}$ with $\alpha$ large and fixed. Putting this together with the inequality for $(i[A,B]u|u)$ and Gaffney's inequality $\norm{u}_{H^1} \sim \norm{u}_{L^2} + \norm{hdu}_{L^2} + \norm{h\delta u}_{L^2}$ when $t*u = 0$, we obtain that 
\begin{equation*} 
\norm{Au}^2 + \norm{Bu}^2 + (i[A,B]u|u) \gtrsim \frac{h^2}{\e}\|u\|_{H^1}^2 - h^3\e^{-1}(\|u_{\|}\|_{H^1(\partial M)}^2 + \|h\nabla_{N} u_{\perp}\|_{L^2(\partial M)}^2)
\end{equation*}
for $h \ll \e \ll 1$. This proves \eqref{NonBndryTermsABCfirst}.

We will now show the expression \eqref{BndryTermsABCfirst} for the boundary terms in \eqref{IbyPidentity}. Recall that these boundary terms are given by 
\begin{equation}\label{BndryTermsAgain}
\begin{split}
  &ih(hdBu|\nu \wedge u)_{\partial M} - ih(h\delta Bu|i_{\nu} u)_{\partial M} + ih(Bu|\nu \wedge h\delta u - i_{\nu} hdu)_{\partial M} \\
+ &2h((\partial_{\nu} \ph_c) Au|u)_{\partial M}.
\end{split}
\end{equation}
Note that 
\begin{eqnarray*}
& & ih(hdB*u|\nu\wedge *u)_{\partial M} - ih(h\delta B*u| i_\nu *u)_{\partial M} + ih(B*u |\nu\wedge h\delta *u - i_\nu hd*u)_{\partial M} \\
& & + 2h ((\partial_\nu\ph_c)A*u |*u)_{\partial M} \\ 
&=& ih(hdBu|\nu \wedge u)_{\partial M} - ih(h\delta Bu|i_{\nu} u)_{\partial M} + ih(Bu|\nu \wedge h\delta u - i_{\nu} hdu)_{\partial M} \\
& & + 2h((\partial_{\nu} \ph_c) Au|u)_{\partial M},
\end{eqnarray*}
Moreover, if $u$ satisfies the absolute boundary conditions \eqref{convexABC}, then $*u$ satisfies the relative boundary conditions 
\begin{equation}\label{convexRBC}
\begin{split}
tu &= 0 \\
th\delta u  &= -t i_{d\ph} u +h \s t i_{N} u,
\end{split}
\end{equation}
and vice versa.  Therefore it suffices to prove that if $u$ satisfies \eqref{convexRBC} then the boundary terms \eqref{BndryTermsAgain} become
\begin{equation}\label{BndryTermsRBC}
-2h^3 (\partial_\nu \ph \nabla_N u_{\|}|\nabla_N u_{\|})_{\partial M}-  2h(\partial_\nu \ph (|d\ph|^2 + |\partial_\nu \ph|^2)u_{\perp}| u_{\perp})_{\partial M} + R  
\end{equation}
where 
\begin{equation}\label{Rbound}
|R| \lesssim  Kh^3 \|\nabla' ti_N u\|^2_{\partial M} + \frac{h}{K}\| u_{\perp}\|^2_{\partial M} + \frac{h^3}{K}\|\nabla_N u_{\|}\|^2_{\partial M} 
\end{equation}
for any large enough $K$ independent of $h$. 

So let's return to \eqref{BndryTermsAgain}, and assume $u$ satisfies \eqref{convexRBC}.  The condition $tu = 0$ implies that the first term $ih(hdBu|\nu \wedge u)_{\partial M}$ is zero.  Therefore we are left with 
\begin{equation*}
-ih(h\delta Bu|i_{\nu} u)_{\partial M} + ih(Bu|\nu \wedge h\delta u)_{\partial M} - ih(Bu| i_{\nu} hdu)_{\partial M} + 2h((\partial_{\nu} \ph_c) Au|u)_{\partial M}.
\end{equation*}

We calculate each of the terms individually.  

Firstly,
\begin{equation*}
\begin{split}
ih(Bu|\nu \wedge h\delta u)_{\partial M} &= ih(Bu|\nu \wedge h(\delta u)_{\|})_{\partial M} \\
                                         &= -ih(i_{N} Bu| h(\delta u)_{\|})_{\partial M} \\
                                         &= -ih(t i_{N} Bu| t h\delta u)_{\partial M}.\\
\end{split}
\end{equation*}
Now
\[
Bu = \frac{h}{i} \left( 2 \nabla_{\grad \ph_c} + \Delta \varphi_c \right)u,
\]
so
\begin{equation*}
\begin{split}
ti_N Bu &= \frac{h}{i} ti_N \left( 2 \nabla_{(\grad \ph_c)_{\|}} - 2\partial_{\nu} \ph_c \nabla_N + \Delta \varphi_c \right)u\\
       &= \frac{h}{i} \left( 2 \nabla_{(\grad \ph_c)_{\|}}ti_N  - 2\partial_{\nu} \ph_c t\nabla_N i_N  + t\Delta \varphi_c i_N \right)u.\\
\end{split}
\end{equation*}
Therefore,
\begin{equation*}
\begin{split}
 -ih(t i_{N} Bu| t h\delta u)_{\partial M} = & 2h(\partial_{\nu} \ph_c t h\nabla_N i_N u| th\delta u)_{\partial M} \\
 & - 2h(h \nabla_{(\grad \ph_c)_{\|}} t i_Nu |th\delta u)_{\partial M} \\
 & - h^2(t \Delta \varphi_c i_N u | th\delta u)_{\partial M}. \\
\end{split}
\end{equation*}
Now if $th\delta u|_{\partial M} = -t i_{d\ph} u + h \s t i_{N} u$,  and $tu = 0$, then
\begin{equation}\label{altRBC}
th\delta u|_{\partial M} = (\partial_{\nu}\ph + h\sigma) t i_{N} u.
\end{equation}
Therefore
\begin{equation*}
\begin{split}
 -ih(t i_{N} Bu| t h(\delta u))_{\partial M} = &2h(\partial_{\nu} \ph_c t h\nabla_N i_N u| (\partial_{\nu}\ph + h\sigma)ti_{N} u)_{\partial M} \\
 & - 2h(h \nabla_{(\grad \ph_c)_{\|}} t i_Nu |(\partial_{\nu}\ph + h\sigma)ti_{N} u)_{\partial M} \\
 & - h^2(t \Delta \varphi_c i_N u | (\partial_{\nu}\ph + h\sigma)ti_{N} u )_{\partial M}. \\
\end{split}
\end{equation*}

Moreover, by Lemma \ref{Tdelta},
\[
t h(\delta u) = h \delta ' tu_{\|} + h((n-1)\kappa - S)ti_{N}u_{\perp} - t h \grad_{N} i_N u.
\]
Since $tu = 0$,
\[
t h(\delta u) = h((n-1)\kappa - S)ti_{N}u_{\perp} - th \grad_{N} i_N u.
\]
Substituting this into \eqref{altRBC} gives
\begin{equation}\label{tnablaNiN}
th\grad_{N} i_{N} u = (-\partial_{\nu}\ph - h\sigma + h(n-1)\kappa - hS)ti_{N} u.
\end{equation}
Therefore
\begin{equation*}
\begin{split}
 -ih(ti_{N} Bu|th(\delta u))_{\partial M} = &-2h(\partial_{\nu} \ph_c (\partial_{\nu}\ph + h\sigma - h(n-1)\kappa + hS)ti_{N} u| (\partial_{\nu}\ph + h\sigma)ti_{N} u)_{\partial M} \\
& - 2h(h\nabla_{(\grad \ph_c)_{\|}} t i_Nu |(\partial_{\nu}\ph + h\sigma)ti_{N} u)_{\partial M} \\
& - h^2(t \Delta \varphi_c i_N u | (\partial_{\nu}\ph + h\sigma)ti_{N} u )_{\partial M}. \\
\end{split}
\end{equation*}
We can write this as
\begin{equation}\label{Term2}
ih(Bu|\nu \wedge h\delta u)_{\partial M} = -2h(\partial_{\nu} \ph |\partial_{\nu} \ph|^2 ti_{N} u| ti_N u)_{\partial M} + R_2
\end{equation}
where $R_2$ satisfies the bound on $R$ in \eqref{Rbound}.


Secondly, 
\begin{eqnarray*}
- ih(Bu| i_{\nu} hdu)_{\partial M} &=&  ih((Bu)_{\|}| i_{N} hdu)_{\partial M} \\
                                   &=&  ih(t(Bu)_{\|}| ti_{N} hdu)_{\partial M}. \\
\end{eqnarray*}
By Lemma \ref{iNdu}, 
\[
t i_N hd u = t h\nabla_{N} u_{\|} + hSt u_{\|} - hd'ti_N u,
\]
so if $tu = 0$, 
\[
t i_N hd u = t h\nabla_{N} u_{\|} - hd'ti_N u.
\]
Therefore
\[
- ih(Bu| i_{\nu} hdu)_{\partial M} = ih(tBu| t h\nabla_{N} u_{\|} - hd'ti_N u)_{\partial M}. 
\]
Expanding $B$, this becomes 
\[
h(th(-2\partial_{\nu} \ph_c \nabla_N u + 2\nabla_{(\grad \ph_c)_{\|}}u + (\Lap \ph_c) u)| t h\nabla_{N} u_{\|} - hd'ti_N u)_{\partial M}.
\]
Since $tu = 0$, the last expression is equal to 
\begin{equation}\label{thirdterm}
- 2h(\partial_{\nu} \ph_c t h\nabla_N u - th \nabla_{(\grad \ph_c)_{\|}}u | t h\nabla_{N} u_{\|} - hd'ti_N u)_{\partial M}.
\end{equation}                      
The
\[
- 2h(\partial_{\nu} \ph_c t h\nabla_N u_{\|}|- hd'ti_N u)_{\partial M}
\]
part has the same type of bound as in \eqref{Rbound}, so 
\begin{equation}\label{Term3}
- ih(Bu| i_{\nu} hdu)_{\partial M} = - 2h(\partial_{\nu} \ph_c t h\nabla_N u_{\|}| t h\nabla_{N} u_{\|})_{\partial M} + R_3,
\end{equation}
where $R_3$ has the same bound as in \eqref{Rbound}.

Thirdly, 
\begin{eqnarray*}
ih(h\delta Bu|i_{\nu} u)_{\partial M} &=& ih(h(\delta Bu)_{\|}|i_{\nu} u)_{\partial M} \\
                                      &=& -ih(ht(\delta Bu)|t i_{N} u)_{\partial M}. \\
\end{eqnarray*}
By Lemma \ref{deltaBu}, 
\begin{eqnarray*}
ht \delta Bu &=& h\delta ' tBu + 2ih^2\grad_{(\grad \ph_c)_{\|}}' t \grad_{N}i_N u -2ih^2\partial_{\nu}\ph_c t\grad_{N} \grad_{N}i_N u\\
            & & + ih^2(2((n-1)\kappa-S)\partial_{\nu} \ph_{c} + 2\partial_{\nu}^2\ph_c + \Lap\ph_c)t\grad_{N} i_N u \\
            & & + 2ih^2 (S -(n-1) \kappa) t \grad_{(\grad \ph_c)_{\|}} i_N u + ih^2((S - (n-1)\kappa) \Lap \ph_c + \grad_{N} \Lap \ph_c)  t i_N u \\
            & &+2ih^2 ti_N R(N,\grad(\ph_c)_{||} ) u_\perp +2ih^2t\grad_{[(\grad \ph_c)_{\|},N]}i_N u -2ih^2i_{s(\grad \ph_c)_{\|}} t \grad_{N} u_{\|}. \\
\end{eqnarray*}

The terms on the last two lines, when paired with $ih t i_{N} u$, are bounded by \eqref{Rbound}. 

Moreover, using the boundary conditions in the form of equation \eqref{tnablaNiN} on the 
\[
h^3((2((n-1)\kappa - S)\partial_{\nu} \ph_{c} + 2\partial_{\nu}^2\ph_c + \Lap\ph_c)t\grad_{N} i_N u | t i_{N} u)_{\partial M} 
\]
term shows that this too is bounded by \eqref{Rbound}. Therefore we need only worry about the first three terms.  

For the $-ih(h\delta ' tBu| ti_{N} u)$ term, we can integrate by parts to get
\[
-ih( tBu| hd' ti_{N} u)_{\partial M} = -2h(ht \nabla_{\grad(\varphi_c)}u +\frac{1}{2} h\Delta \varphi_c tu|hd' ti_{N} u)_{\partial M}.
\]
Since $tu = 0$, we get
\[
ih( tBu| hd' ti_{N} u)_{\partial M} = 2h(ht \nabla_{\grad(\varphi_c)}u|hd' ti_{N} u)_{\partial M}.
\]
Now 
\[
t \nabla_{\grad(\varphi_c)}u = t \nabla_{\nabla(\varphi_c)_{\|}} u_{\perp} + t \nabla_{\nabla(\varphi_c)_{\perp}} u_{\|}
\]
since $tu = 0$.  Therefore 
\[
|ih( tBu| hd' ti_{N} u)_{\partial M}| \leq Kh^3\|\grad ' t i_{N} u \|^2_{\partial M} + K h^3 \norm{ u_{\perp} }_{\partial M}^2 + K h^3 \norm{\nabla_N u_{\|}}^2,
\]
and so this term is bounded by \eqref{Rbound}.  

For the $2h^3(\grad_{(\grad \ph_c)_{\|}}' t \grad_{N}i_N u|t i_{N} u)_{\partial M} $ term, we can use equation \eqref{tnablaNiN} to get
\begin{eqnarray*}
& & 2h^3(\grad_{(\grad \ph_c)_{\|}}' t \grad_{N}i_N u|t i_{N} u)_{\partial M} \\
&=& -2h^2(\grad_{(\grad \ph_c)_{\|}}' (-\partial_{\nu}\ph - h\sigma + h(n-1)\kappa - hS)ti_{N} u |t i_{N} u)_{\partial M} .
\end{eqnarray*}
and then use Cauchy-Schwartz, so this term is bounded by \eqref{Rbound} too.  Therefore
\begin{equation}\label{Term1}
-ih(h\delta Bu|i_{\nu} u)_{\partial M} = 2h^3(\partial_{\nu}\ph_c t\grad_{N} \grad_{N}i_N u|ti_{N} u)_{\partial M} + R_1
\end{equation}
where $R_1$ is bounded by \eqref{Rbound}.

Finally,
\begin{equation*}
\begin{split}
2h((\partial_{\nu} \ph_c) Au|u)_{\partial M} &= 2h((\partial_{\nu} \ph_c) Au|u_{\perp})_{\partial M} \\
                                             &= 2h((\partial_{\nu} \ph_c) (Au)_{\perp}|u_{\perp})_{\partial M} \\
                                             &= 2h((\partial_{\nu} \ph_c) t i_N Au| t i_N u)_{\partial M} \\
\end{split}
\end{equation*}
because of the boundary condition $tu = 0$.  Now $A = -h^2 \Delta - \abs{d\varphi_c}^2$, so
\begin{equation*}
\begin{split}
2h((\partial_{\nu} \ph_c) ti_N Au|ti_N u)_{\partial M} = &-2h((\partial_{\nu} \ph_c)h^2 ti_N \Delta u|ti_N u)_{\partial M} \\
                                                         &-2h((\partial_{\nu} \ph_c)\abs{d\varphi_c}^2 ti_N u|ti_N u)_{\partial M}. \\
\end{split}
\end{equation*}
Using the Weitzenbock identity, we can write $-2h((\partial_{\nu} \ph_c)h^2 ti_N \Delta u|ti_N u)_{\partial M} $ as 
\[
-2h((\partial_{\nu} \ph_c)h^2 ti_N \tilde{\Delta} u|ti_N u)_{\partial M} + 2h((\partial_{\nu} \ph_c)h^2 R ti_N u|ti_N u)_{\partial M}.  
\]
The second term is bounded by \eqref{Rbound}.  For the first term, we can apply Lemma \ref{ConxnLaplace} to get
\[
-2h((\partial_{\nu} \ph_c)h^2 t \nabla_{N} \nabla_N i_N u|ti_N u)_{\partial M} -2h((\partial_{\nu} \ph_c)h^2 \tilde{\Lap} ' t i_N u|ti_N u)_{\partial M} + h^3(tr(s^2)i_Nu -S_2i_Nu|ti_N u)_{\partial M}
\]
where $S_2 \omega (X_1,..,X_{k-1}) := \sum\limits_{l=1}^{k-1}\omega(.., s^2 X_l, ..)$. The last term is bounded again by \eqref{Rbound} and we can integrate by parts in the $\tilde{\Lap} '$ part to get something bounded by \eqref{Rbound} as well.  Therefore
\begin{equation*}
\begin{split}
2h((\partial_{\nu} \ph_c) Au|u)_{\partial M} = &-2h((\partial_{\nu} \ph_c)\abs{d\varphi_c}^2 ti_N u|ti_N u)_{\partial M} \\
                                               &-2h((\partial_{\nu} \ph_c)h^2 t \nabla_{N} \nabla_N i_N u|ti_N u)_{\partial M} + R_4 \\
\end{split}
\end{equation*}
where $R_4$ is bounded by \eqref{Rbound}.  

Now putting this together with \eqref{Term2}, \eqref{Term3}, and \eqref{Term1}, we get that the boundary terms in \eqref{IbyPidentity} have the form
\begin{equation*}
\begin{split}
 &-2h(\partial_{\nu} \ph |\partial_{\nu} \ph|^2 ti_{N} u| ti_N u)_{\partial M}- 2h(\partial_{\nu} \ph_c t h\nabla_N u_{\|}| t h\nabla_{N} u_{\|})_{\partial M}\\
+&2h^3(\partial_{\nu}\ph_c t\grad_{N} \grad_{N}i_N u|ti_{\nu} u)_{\partial M}-2h((\partial_{\nu} \ph_c)\abs{d\varphi_c}^2 ti_N u|ti_N u)_{\partial M}\\
-&2h((\partial_{\nu} \ph_c)h^2 t \nabla_{N} \nabla_N i_N u|ti_N u)_{\partial M} + R.
\end{split}
\end{equation*}
The $\pm 2h^3(\partial_{\nu}\ph_c t\grad_{N} \grad_{N}i_N u|ti_{\nu} u)_{\partial M}$ terms cancel, leaving us with 
\begin{equation*}
\begin{split}
&-2h(\partial_{\nu} \ph|\partial_{\nu} \ph|^2 ti_{N} u| ti_N u)_{\partial M}- 2h(\partial_{\nu} \ph_c t h\nabla_N u_{\|}| t h\nabla_{N} u_{\|})_{\partial M}\\
&-2h((\partial_{\nu} \ph_c)\abs{d\varphi_c}^2 ti_N u|ti_N u)_{\partial M}+ R. \\
\end{split}
\end{equation*}
We can replace $\ph_c$ by $\ph$ and incorporate the error into $R$ without affecting the bound on $R$, to get
\begin{equation*}
\begin{split}
&-2h(\partial_{\nu} \ph|\partial_{\nu} \ph|^2 ti_{N} u| ti_N u)_{\partial M}- 2h(\partial_{\nu} \ph t h\nabla_N u_{\|}| t h\nabla_{N} u_{\|})_{\partial M}\\
&-2h(\partial_{\nu} \ph\abs{d\varphi}^2 ti_N u|ti_N u)_{\partial M}+ R \\
\end{split}
\end{equation*}
and the proposition follows.

\end{proof}

\section{The $0$-form case} \label{sec:carleman2}

We will now prove Theorem \ref{ABCCarl} in the 0-form case.  In the case where $(M,g)$ is a domain in Euclidean space, Theorem \ref{ABCCarl} for $0$-forms is  the Carleman estimate given in ~\cite[Theorem 1.3]{Ch2}. In this section we will deal with the added complication of being on a CTA manifold, rather than in Euclidean space. Most of the ideas are from ~\cite{Ch2} with necessary modifications added to adapt to the manifold case.

If $u$ is a zero form, then $i_N u = 0$, so $u_{\perp} = 0$ and $u = u_{\|}$.  Theorem \ref{ABCCarl} reduces to the estimate 
\begin{equation} \label{ABCCarl_zeroform}
\|(-\Delta_\ph+h^2 Q)u\|_{L^2(M)} \gtrsim h \|u\|_{H^1(M)} + h^{\half}\| u_{\|}\|_{H^1(\Gamma_+^c)}
\end{equation}
where $Q \in L^{\infty}(M)$ and $0 < h < h_0$, for functions $u \in H^2(M)$ with $u|_{\Gamma_+} = 0$ to first order and $h\partial_{\nu} (e^{-\frac{\ph}{h}} u) = h\s e^{-\frac{\ph}{h}} u$ on $\Gamma_+^c$. By arguing as in the beginning of Section \ref{sec:carleman3} below, the estimate \eqref{ABCCarl_zeroform} will be a consequence of the following proposition.

\begin{prop}\label{MainBndryEst}
Suppose $u$ is a function in $H^2(M)$ which satisfies the following boundary conditions:
\begin{equation}\label{BC}
\begin{split}
u, \partial_{\nu} u &= 0 \mbox{ on } \Gamma_+ \\
h\partial_{\nu} (e^{-\frac{\ph}{h}} u) &= h\s e^{-\frac{\ph}{h}} u \mbox{ on } \Gamma_+^c
\end{split}
\end{equation}
for some smooth function $\s$ independent of $h$.

Then 
\[
h^{\half} \|h \grad ' u \|_{L^2(\Gamma_+^c)} \lesssim \|\Delta_{\ph_c} u \|_{L^2(M)} + h\|u\|_{H^1(M)} + h^{\frac{3}{2}}\|u\|_{L^2(\Gamma_+^c)},  
\]
\end{prop}

We will prove this proposition in the case where the metric $g$ has the form $g = e \oplus g_0$. However, if $g$ were of the form $g = c(e \oplus g_0)$, we could write
\begin{equation}
\begin{split}
\|\Delta_{\ph_c} u \|_{L^2(M)} &=       \| h^2 e^{\frac{\ph_c}{h}} \Lap_{c(e \oplus g_0)} e^{-\frac{\ph_c}{h}} u \|_{L^2(M)} \\
                                 &\gtrsim \| h^2 e^{\frac{\ph_c}{h}} \Lap_{e \oplus g_0} e^{-\frac{\ph_c}{h}} u \|_{L^2(M)} - h\|u\|_{H^1(M)}.\\
\end{split}
\end{equation}
Therefore the proposition remains true even in the case when the conformal factor is not constant.  More generally, the proofs of the Carleman estimates work for any smooth conformal factor, and thus as noted earlier, the Carleman estimates hold on CTA manifolds in general.  



\vspace{10pt}

\noindent \ref{sec:carleman2}.1. {\bf The operators.} Here we introduce the operators we will use in the proof of Proposition \ref{MainBndryEst}.  Similar operators are found in ~\cite{Ch1} and ~\cite{Ch2}.  Suppose $F(\xi)$ is a complex valued function on $\R^{n-1}$, with the properties that $|F(\xi)|, \mathrm{Re} F(\xi) \simeq 1 + |\xi|$.  Fix coordinates $(x_1, x')$ on $\Rn$, and define $\Rnp$ to be the subset of $\R^n$ with $x_1 > 0$.  Define $\S(\Rnp)$ as the set of restrictions to $\Rnp$ of Schwartz functions on $\Rn$.  Finally, if $u \in \S(\Rnp)$, then define $\hat{u}(x_1,\xi)$ to be the semiclassical Fourier transform of $u$ in the $x'$ variables only.

Now for $u \in \S(\Rnp)$, define $J$ by
\[
\widehat{Ju}(x_1,\xi) = (F(\xi) + h\partial_1)\hat{u}(x_1,\xi).
\]
This has adjoint $J^*$ defined by 
\[
\widehat{J^*u}(x_1, \xi) = (\overline{F}(\xi) - h\partial_1)\hat{u}(x_1, \xi).
\]
These operators have right inverses given by
\[
\widehat{J^{-1}u} = \frac{1}{h}\int_0^{x_1} \hat{u}(t,\xi) e^{F(\xi)\frac{t-x_1}{h}}dt
\]
and
\[
\widehat{J^{*-1}u} = \frac{1}{h}\int_{x_1}^\infty \hat{u}(t,\xi) e^{\overline{F}(\xi)\frac{x_1-t}{h}}dt.
\]
Now we have the following boundedness result, given in ~\cite{Ch2}.

\begin{lemma}\label{bddness}
The operators $J$, $J^{*}$, $J^{-1}$, and $J^{*-1}$, initially defined on $\S(\Rnp)$, extend to bounded operators
\[
J, J^{*}: H^1(\Rnp) \rightarrow L^2(\Rnp)
\]
and
\[
J^{-1}, J^{*-1}: L^2(\Rnp) \rightarrow H^1(\Rnp).
\]
Moreover, these extensions for $J^{*}$ and $J^{*-1}$ are isomorphisms.
\end{lemma}

Note that similar mapping properties hold between $H^1(\Rnp)$ and $H^2(\Rnp)$, by the same reasoning.  

We'll record the other operator fact from ~\cite{Ch2} here, too.  

Let $m,k \in \mathbb{Z}$, with $m,k\geq 0$.  Suppose $a(x,\xi,y)$ are smooth functions on $\R^{n-1} \times \R^{n-1} \times \R$ that satisfy the bounds
\[
|\partial_x^\b \partial_\xi^{\a} \partial_y^j a(x,\xi,y)| \leq C_{\a,\b} (1 + |\xi|)^{m-|\a|}
\]
for all multiindices $\a$ and $\b$, and for $0 \leq j \leq k$.  In other words, each $\partial_y^j a(x,\xi,y)$ is a symbol on $\R^{n-1}$ of order $m$, with bounds uniform in $y$, for $0 \leq j \leq k$.  Then we can define an operator $A$ on Schwartz functions in $\Rn$ by applying the pseudodifferential operator on $\R^{n-1}$ with symbol $a(x,\xi,y)$, defined by the Kohn-Nirenberg quantization, to $f(x,y)$ for each fixed $y$.  

\begin{lemma}\label{flatops}
If $A$ is as above, then $A$ extends to a bounded operator from $H^{k+m}(\Rn)$ to $H^k(\Rn)$.
\end{lemma}



\vspace{10pt}

\noindent \ref{sec:carleman2}.2. {\bf The graph case.} Suppose $f: M_0 \rightarrow \R$ is smooth.  In this section, we'll examine the case where $M$ lies in the set $\{ x_1 \geq f(x')\}$, and $\Gamma_+^c$ lies in the graph $\{ x_1 = f(x') \}$.  For this section we'll make two additional assumptions on $f$ and $M_0$.

First, we'll assume that $g_0$ is nearly constant:  that there exists a choice of coordinates on the subset $P(M)$ which consists of the projection of $M$ onto $M_0$, such that when represented in these coordinates,
\[
|g_0 - I| \leq \d
\]
on $P(M)$, where $\d$ is a positive constant to be chosen later.

Second, we'll assume that $f$ is such that $\grad_{g_0} f$ is nearly constant on $P(M)$:  that there exists a constant vector field $K$ on $TM_0$ such that
\[
|\grad_{g_0} f - K|_{g_0} \leq \d
\]
where $\d$ is the same constant from above.  The choice of $\d$ will depend ultimately only on $K$.  In the next subsection we'll see how to remove these two assumptions.

Now we can do the change of variables $(x_1,x') \mapsto (x_1 - f(x'), x')$.  Define $\tilde{M}'$ and $\tilde{\Gamma}_+'$ to be the images of $M$ and $\Gamma_+$ respectively, under this map. Note that $\{ x_1 \geq f(x')\}$ maps to $(0,\infty) \times M_0$, and $\Gamma_+^{c}$ maps to a subset of $0 \times M_0$.  Observe that in the new coordinates $\varphi(x) = x_1 + f(x')$.

Now it suffices to prove the following proposition.  

\begin{prop}\label{FirstChange}
Suppose $w \in H^2(\tilde{M} ')$, and  
\begin{equation}\label{Change1BC}
\begin{split}
w, \partial_{\nu} w &= 0 \mbox{ on } \tilde{\Gamma}_+' \\
h \partial_{y}w|_{\tilde{\Gamma}_+'^c} &= \frac{w + \grad_{g_0} f \cdot h \grad_{g_0} w - h\s w}{1 + |\grad_{g_0} f|^2}.
\end{split}
\end{equation}
where $\s$ is smooth and bounded on $\tilde{M} '$.  
Then
\[
h^{\half} \|h \grad_{g_0} w \|_{L^2(\tilde{\Gamma}_+'^c)} \lesssim \|\Lphet ' w \|_{L^2(\tilde{M} ')} + h\|w\|_{H^1(\tilde{M} ')} + h^{\frac{3}{2}}\|w\|_{L^2(\tilde{\Gamma}_+'^c )} 
\]
where
\[
\Lphet ' = (1+|\grad_{g_0} f|^2)h^2\partial_1^2 - 2(\a + \grad_{g_0} f \cdot h\grad_{g_0})h\partial_1 + \a^2 + h^2\Lap_{g_0}
\]
and $\a = 1 + \frac{h}{\e}(x_1 + f(x '))$.  Note that on $\tilde{M} '$, $\a$ is very close to $1$.
\end{prop}

This proposition implies Proposition \ref{MainBndryEst}, in the graph case described above.

\begin{proof}[Proof of Proposition \ref{MainBndryEst} in the graph case]

Suppose $u \in H^2(M)$, and $u$ satisfies \eqref{BC}.  Let $w$ be the function on $\tilde{M}$ defined by $w(x_1, x') = u(x_1 + f(x'), x')$.  Then $w \in H^2(\tilde{M} ')$, and $w$ satisfies \eqref{Change1BC}.  Therefore by Proposition \ref{FirstChange}, 
\[
h^{\half} \|h \grad ' w \|_{L^2(\tilde{\Gamma}_+'^c)} \lesssim \|\Lphet ' w \|_{L^2(\tilde{M} ')} + h\|w\|_{H^1(\tilde{M} ')} + h^{\frac{3}{2}}\|w\|_{L^2(\tilde{\Gamma}_+'^c )} .
\]
Now by a change of variables,
\begin{equation*}
\begin{split}
\|u\|_{L^2(\Gamma_+^c)} &\simeq \|w\|_{L^2(\tilde{\Gamma}_+'^c)}, \\
\|u\|_{H^1(M)} &\simeq \|w\|_{H^1(\tilde{M} ')}, \\
\end{split}
\end{equation*}
and
\[
\|h\grad ' u\|_{L^2(\Gamma_+^c)} \simeq \|h\grad_{g_0} w\|_{L^2(\tilde{\Gamma}_+'^c)}.
\]
Moreover, 
\[
\left( \Lphet ' w \right)(x_1 -f(x'), x') = \Lphe \left( u(x_1,x') \right) + hE_1 u(x_1,x')
\]
where $E_1$ is a first order semiclassical differential operator.  Therefore by a change of variables,
\[
\|\Lphet ' w \|_{L^2(\tilde{M} ')} \lesssim \|\Lphe u \|_{L^2(M)} + h\|u\|_{H^1(M)}.
\]
Putting this all together gives
\[
h^{\half} \|h \grad_{g_0} u \|_{L^2(\Gamma_+^c)} \lesssim \|\Lphe u \|_{L^2(M)} + h\|u\|_{H^1(M)} + h^{\frac{3}{2}}\|u\|_{L^2(\Gamma_+^c )} .
\]
\end{proof}

We can do a second change of variables to move to Euclidean space.  By our assumption on $M_0$, we can choose coordinates on $P(\tilde{M} ') = P(M)$ such that 
\[
|g_0 - I| \leq \d.
\]
Now we have a change of variables giving a map from $P(\tilde{M} ')$ to a subset of $\R^{n-1}$, and hence a map from $\tilde{M} '$ to a subset of $\Rnp$, where the image of $\tilde{\Gamma}_+'$ lies in the plane $x_1 = 0$.  Let $\tilde{M}$ and $\tilde{\Gamma}_+$ be the images of $\tilde{M} '$ and $\tilde{\Gamma}_+'$ respectively, under this map.  We'll use the notation $(x_1, x')$ to describe points in $\Rnp$, where now $x'$ ranges over $\R^{n-1}$.  Now it suffices to prove the following proposition.  

\begin{prop}\label{SecondChange}
Suppose $w \in H^2(\tilde{M})$, and  
\begin{equation}\label{tildeBC}
\begin{split}
w, \partial_{\nu} w &= 0 \mbox{ on } \tilde{\Gamma}_{+} \\
h \partial_{y}w|_{\tilde{\Gamma}_+^c} &= \frac{w + \b \cdot h \grad_{x} w - h\s w}{1 + |\g |^2}.
\end{split}
\end{equation}
where $\s$ is smooth and bounded on $\tilde{M}$, and $\b$ and $\g$ are a vector valued and scalar valued function, respectively, which coincide with the coordinate representations of $\grad_{g_0} f$ and $|\grad_{g_0} f|_{g_0}$.  
Then
\[
h^{\half} \|h \grad_{x'} w \|_{L^2(\tilde{\Gamma}_+^c)} \lesssim \|\Lphet w \|_{L^2(\tilde{M})} + h\|w\|_{H^1(\tilde{M} )} + h^{\frac{3}{2}}\|w\|_{L^2(\tilde{\Gamma}_+^c )} 
\]
where
\[
\Lphet  = (1+|\g|^2)h^2\partial_1^2 - 2(\a + \b \cdot h \grad_x)h\partial_1 + \a^2 + h^2\L,
\]
and $\L$ is the second order differential operator in the $x'$ variables given by 
\[
\L = g_0^{ij}\partial_i \partial_j.
\]
\end{prop}

Proposition \ref{FirstChange} can be obtained from Proposition \ref{SecondChange} in the same manner as before, with errors from the change of variables being absorbed into the appropriate terms.  Therefore it suffices to prove Proposition \ref{SecondChange}.  

To do this, we'll split $w$ into small and large frequency parts, using a Fourier transform.  Recall that we are assuming 
\[
|\grad_{g_0} f - K|_{g_0} \leq \d.
\]
Translating down to $\tilde{M}$, and recalling that $g_0$ is nearly the identity, we get that there is a constant vector field $\tilde{K}$  on $\tilde{M}$ such that
\[
|\b - \tilde{K}| \leq C_{\d}
\]
and
\[
|\g - |\tilde{K}|| \leq C_{\d}
\]
where $C_{\d}$ goes to zero as $\d$ goes to zero.  Now choose $m_2 > m_1 > 0$, and $\mu_1$ and $\mu_2$ such that 
\[
\frac{|\tilde{K}|}{\sqrt{1+|\tilde{K}|^2}} < \mu_1 < \mu_2 < \half + \frac{|\tilde{K}|}{2\sqrt{1+|\tilde{K}|^2}} < 1.
\]
The eventual choice of $\mu_j$ and $m_j$ will depend only on $\tilde{K}$.  

Define $\rho \in \Czinf(\Rn)$ such that $\rho(\xi) = 1$ if $|\xi| < \mu_1$ and $|\tilde{K} \cdot \xi| < m_1$, and $\rho(\xi) = 0$ if $|\xi| > \mu_2$ or $|\tilde{K} \cdot \xi| > m_2$.  

Now suppose $w \in C^{\infty}(\tilde{M})$ such that $w \equiv 0$ in a neighbourhood of $\tilde{\Gamma}_+$, and $w$ satisfies \eqref{tildeBC}.  We can extend $w$ by zero to the rest of $\Rnp$.  Then $w \in \S(\Rnp)$, and we can write our desired estimate as
\[
h^{\half} \| w \|_{\dot{H}^1(\partial \Rnp)} \lesssim \|\Lphet w \|_{L^2(\Rnp)} + h\|w\|_{H^1(\Rnp)}+ h^{\frac{3}{2}}\|w\|_{L^2(\partial \Rnp )}. 
\]

Recall that  $\hat{w}(x_1, \xi)$ is the semiclassical Fourier transform of $w$ in the $x'$ directions, and define $w_s$ and $w_{\ell}$ by $\hat{w}_s = \rho \hat{w}$ and $\hat{w}_{\ell} = (1-\rho) \hat{w}$, so $w = w_s + w_{\ell}$.  

Now we can address each of these parts separately.  

\begin{prop}\label{smallprop}
Suppose $w$ is as above.  There exist choices of $m_1, m_2, \mu_1,$ and $\mu_2$, depending only on $\tilde{K}$, such that if $\d$ is small enough,
\[
h^{\half} \|w_s \|_{\dot{H}^1(\partial \Rnp)} \lesssim \|\Lphet w \|_{L^2(\Rnp)} + h\|w\|_{H^1(\Rnp)} + h^{\frac{3}{2}}\|w\|_{L^2(\partial \Rnp)}.
\]
\end{prop}

Before proceeding to the proof, let's make some definitions.  If $V \in \R^{n-1}$ and $a \in \R$, define $A_{\pm}(a,V, \xi)$ by
\[
A_{\pm}(a,V,\xi) = \frac{1 + iV \cdot \xi \pm \sqrt{(1+iV\cdot \xi)^2 - (1+|a|^2)(1-|\xi|^2)}}{1 + |a|^2},
\]
In other words, $A_{\pm}(a,V, \xi)$ are defined to be the roots of the polynomial
\[
(1 + |a|^2)X^2 - 2(1+ iV \cdot \xi)X + (1 - |\xi|^2).
\]
In the definition, we'll choose the branch of the square root which has non-negative real part, so the branch cut occurs on the negative real axis.  

\begin{proof}
Now consider the behaviour of $A_{\pm}(|\tilde{K}|, \tilde{K},\xi)$ on the support of $\rho$, or equivalently, on the support of $\hat{w}_s$.  If $\eta > 0$, we can choose $\mu_2$ such that on the support of $\hat{w}_s$, 
\[
1 - (1+|\tilde{K}|^2)(1-|\xi|^2) < \eta.
\]
Then on the support of $\hat{w}_s$, the expression
\[
(1+i\tilde{K} \cdot \xi)^2 - (1+| \tilde{K} |^2)(1-|\xi|^2)
\]
has real part confined to the interval $[ - \tilde{K}^2 - m_2^2, \eta + m_2^2 ]$, and imaginary part confined to the interval $[-2m_2, 2m_2]$.  Therefore, by correct choice of $\eta$ and $m_2$, we can ensure
\[
\mathrm{Re}A_{\pm}(|\tilde{K}|, \tilde{K},\xi) > \frac{1}{2(1+|\tilde{K}|^2)}.
\]
on the support of $\hat{w}_s$.  This allows us to fix the choice of $\mu_1, \mu_2, m_1,$ and $m_2$.  Note that the choices depend only on $\tilde{K}$, as promised.  

The bounds on $A_{\pm}(|\tilde{K}|, \tilde{K},\xi)$ allow us to choose $F_{\pm}$ so that $F_{\pm} = A_{\pm}(|\tilde{K}|, \tilde{K},\xi)$ on the support of $\hat{w}_s$, and $\mathrm{Re}F_{\pm}, |F_{\pm}| \simeq 1 + |\xi|$ on $\Rn$, with constant depending only on $K$.  Therefore $F_{+}$ and $F_{-}$ both satisfy the conditions on $F$ in Section 2.  If $T_{\psi}$ represents the operator with Fourier multiplier $\psi$ (in the $x'$ variables), then it follows that the operators $h \partial_y - T_{F_{+}}$ and $h\partial_y - T_{F_{-}}$ both have the properties of $J^{*}$ in that section.  

Up until now, the operator $\Lphet$ has only been applied to functions supported in $\tilde{M}$.  However, we can extend the coefficients of $\Lphet$ to $\Rnp$ while retaining the $|\b - \tilde{K}| < C_{\d}$  and $|\g - |\tilde{K}|| \leq C_{\d}$ conditions.  Then
\begin{eqnarray*}
\|\Lphet w_s\|_{L^2(\Rnp)} &=& \|((1+|\g|^2)h^2\partial_y^2 - 2(\a + \b \cdot h\grad_x)h\partial_y + \a^2 + h^2\L)w_s\|_{L^2(\Rnp)} \\
                            &\geq& \|((1+|\tilde{K}|^2)h^2\partial_y^2 - 2(1 + \tilde{K} \cdot h\grad_x)h\partial_y + 1 + h^2\Lap_{x'})w_s\|_{L^2(\Rnp)} \\
                            & & - C_\delta \|w_s\|_{H^2(\Rnp)}\\
\end{eqnarray*}
for sufficiently small $h$ and some $C_{\delta}$ which goes to zero as $\delta$ goes to zero.  Meanwhile, 
\begin{eqnarray*}
& & (1+|\tilde{K}|^2)(h\partial_y - T_{F_{+}})(h\partial_y - T_{F_{-}})w_s \\
&=& (1+|\tilde{K}|^2)(h^2 \partial_y^2 - T_{F_{+} + F_{-}} h\partial_y + T_{F_{+}F_{-}})w_s. \\
\end{eqnarray*}
Since $F_{\pm} = A_{\pm}(\tilde{K}, K, \xi)$ on the support of $\hat{w}_s$, this can be written as 
\begin{eqnarray*}
& & (1+|\tilde{K}|^2)(h^2 \partial_y^2 - T_{A_{+} + A_{-}}h\partial_y + T_{A_{+}A_{-}})w_s\\
&=& ((1+|\tilde{K}|^2)h^2\partial_y^2 - 2(1 + \tilde{K} \cdot h\grad_x)h\partial_y + 1 + h^2\Lap_x)w_s.\\
\end{eqnarray*}
Therefore
\[
\|\Lphet w_s\|_{L^2(\Rnp)} \geq \|(h\partial_y - T_{F_{+}})(h\partial_y - T_{F_{-}})w_s \|_{L^2(\Rnp)} - C_\delta\|w_s\|_{H^2(\Rnp)}.
\]
Now by the boundedness properties,
\[
\|(h\partial_y - T_{F_{+}})(h\partial_y - T_{F_{-}})w_s \|_{L^2(\Rnp)} \simeq \|w_s\|_{H^2(\Rnp)},
\]
so for small enough $\delta$, 
\[
\|\Lphet w_s\|_{L^2(\Rnp)} \gtrsim \|w_s\|_{H^2(\Rnp)}.
\]
Then by the semiclassical trace formula, 
\[
\|\Lphet w_s\|_{L^2(\Rnp)} \gtrsim h^{\half}\|w_s\|_{\dot{H}^1(\partial \Rnp)}.
\]

Finally, note that 
\begin{equation*}
\begin{split}
\|\Lphet w_s\|_{L^2(\Rnp)} &=        \|\Lphet T_{\rho} w\|_{L^2(\Rnp)}\\
                           &\lesssim \|(1 + |\gamma|^2)^{-1} \Lphet T_{\rho} w\|_{L^2(\Rnp)}\\
                           &\lesssim \|T_{\rho} (1 + |\gamma|^2)^{-1} \Lphet  w\|_{L^2(\Rnp)} + \|hE_1w\|_{L^2(\Rnp)}.
\end{split}
\end{equation*}
where $hE_1$ comes from the commutator of $T_{\rho}$ and $(1 + |\gamma|^2)^{-1} \Lphet$.  By Lemma \ref{flatops}, $E_1$ is bounded from $H^1(\Rnp)$ to $L^2(\Rnp)$, so
\[
\|\Lphet w_s\|_{L^2(\Rnp)} \lesssim \|\Lphet  w\|_{L^2(\Rnp)} + h\|w\|_{H^1(\Rnp)}.
\]
Therefore
\[
\|\Lphet  w\|_{L^2(\Rnp)} + h\|w\|_{H^1(\Rnp)} \gtrsim h^{\half}\|w_s\|_{\dot{H}^1(\partial \Rnp)}
\]
as desired.

\end{proof}

Now we have to deal with the large frequency term.  

\begin{prop}\label{largeprop}
Suppose $w$ is the extension by zero to $\Rnp$ of a function in $C^{\infty}(\tilde{M})$ which is $0$ in a neighbourhood of $\tilde{\Gamma}_+$, and satisfies \eqref{tildeBC}, and let $w_{\ell}$ be defined as above.  Then if $\d$ is small enough,
\[
h^{\half} \|w_{\ell}\|_{\dot{H}^1(\partial \Rnp)} \lesssim \|\Lphet w \|_{L^2(\Rnp)} + h\|w\|_{H^1(\RnIp)} + h^{\frac{3}{2}}\|w\|_{L^2(\partial \Rnp)}.
\]
\end{prop}

\begin{proof}

Suppose $V \in \Rn$.  Recall that we defined 
\[
A_{\pm}(a,V,\xi) = \frac{1 + iV \cdot \xi \pm \sqrt{(1+iV\cdot \xi)^2 - (1+|a|^2)(1-|\xi|^2)}}{1 + |a|^2},
\]
so $A_{\pm}(a,V,\xi)$ are roots of the polynomial
\[
(1 + |a|^2)X^2 - 2(1+ iV \cdot \xi)X + (1 - |\xi|^2).
\]
Now let's define
\[
A^{\e}_{\pm}(a, V,\xi) = \frac{\a + iV \cdot \xi \pm \sqrt{(\a+iV\cdot \xi)^2 - (1+|a|^2)(\a^2-g_0^{ij}\xi_i \xi_j)}}{1 + |a|^2},
\]
so $A^{\e}_{\pm}(V,\xi)$ are the roots of the polynomial
\[
(1 + |a|^2)X^2 - 2(\a+ iV \cdot \xi)X + (\a^2 - g_0^{ij}\xi_i \xi_j).
\]
(Recall that $\a$ is defined by $\a = 1 + \frac{h}{\e}(x_1 + f(x'))$.)  Again we'll use the branch of the square root with non-negative real part.  

Now set $\zeta \in C^{\infty}_0(\R^{n-1})$ to be a smooth cutoff function such that $\zeta = 1$ if 
\[
|\tilde{K} \cdot \xi| < \half m_1 \mbox{ and } |\xi| < \half \frac{|\tilde{K}|}{\sqrt{1+|\tilde{K}|^2}} +  \half \mu_1, 
\]
and $\zeta = 0$ if $|\tilde{K} \cdot \xi| \geq m_1$ or $|\xi| \geq \mu_1$.  

Now define 
\[
G_{\pm}(a,V, \xi) = (1-\zeta)A_{\pm}(a,V,\xi) + \zeta 
\]
and
\[
G^{\e}_{\pm}(a,V, \xi) = (1-\zeta)A^{\e}_{\pm}(a, V,\xi) + \zeta.
\]

Consider the singular support of $A^{\e}_{\pm}(\g,\b,\xi)$.  These are smooth as functions of $x$ and $\xi$ except when the argument of the square root falls on the non-positive real axis.  This occurs when $\b \cdot \xi = 0$ and
\[
g_0^{ij}\xi_i \xi_j \leq \frac{\a^2|\g|^2}{1 + |\g|^2}.
\] 
Now for $\d$ sufficiently small, depending on $\tilde{K}$, this does not occur on the support of $1 -\z$.  Therefore
\[
G^{\e}_{\pm}(\g, \b, \xi) = (1-\zeta)A^{\e}_{\pm}(\g,\b,\xi) + \zeta
\]
are smooth, and one can check that they are symbols of first order on $\Rn$.  

Then by properties of pseudodifferential operators,
\begin{eqnarray*}
& & (1+|\g|^2)(h\partial_y - T_{G^{\e}_{+}(\g,\b, \xi)})(h\partial_y - T_{G^{\e}_{-}(\g,\b, \xi)}) \\
&=& (1+|\g|^2)(h^2\partial_y^2 - T_{G^{\e}_{+}(\g,\b, \xi) + G^{\e}_{-}(\g,\b, \xi)}h\partial_y + T_{G^{\e}_{+}(\g,\b, \xi)G^{\e}_{-}(\g,\b, \xi)})+hE_1,
\end{eqnarray*}
where $E_1$ is bounded from $H^1(\RnIp)$ to $L^2(\RnIp)$.  
This last line can be written out as
\begin{equation*}
\begin{split}
&(1+|\g|^2)h^2\partial_y^2 - 2(\a + \b \cdot h \grad_x)h\partial_y T_{1 - \zeta}T_{1+\z} + (\a + h^2\L)T_{(1-\z)^2} \\
&+ hE_1 + T_{\zeta^2} -2 h\partial_y T_{\zeta},
\end{split}
\end{equation*}
by modfiying $E_1$ as necessary.  Now $T_{\zeta}w_{\ell} = 0$, so
\begin{eqnarray*}
& & (1+|\g|^2)(h\partial_y - T_{G^{\e}_{+}(\g,\b,\xi)})(h\partial_y - T_{G^{\e}_{-}(\g,\b,\xi)})w_{\ell}\\
&=& \Lphet w_{\ell} - h E_1 w_{\ell}.
\end{eqnarray*}

Therefore
\begin{eqnarray*}
\|\Lphet w_{\ell} \|_{L^2(\RnIp)} &\gtrsim& \|(h\partial_y - T_{G^{\e}_{+}(\g,\b,\xi)})(h\partial_y - T_{G^{\e}_{-}(\g,\b,\xi)})w_{\ell}\|_{L^2(\RnIp)} \\
                                     & & - h\|w_{\ell}\|_{H^1(\RnIp)}.\\
\end{eqnarray*}
Now
\[
G^{\e}_{+}(\g,\b, \xi) =  G_{+}(|\tilde{K}|,\tilde{K}, \xi) + (G^{\e}_{+}(\g, \b, \xi) -  G_{+}(|\tilde{K}|,\tilde{K}, \xi)),
\]
and
\[
T_{G^{\e}_{+}(\g,\b, \xi) -  G_{+}(|\tilde{K}|,\tilde{K}, \xi)}
\]
involves multiplication by functions bounded by $O(\d)$, so
\[
\|T_{G^{\e}_{+}(\g,\b, \xi) -  G_{+}(|\tilde{K}|,\tilde{K}, \xi)}v\|_{L^2(\RnIp)} \lesssim \d \|v\|_{H^1(\RnIp)}.
\]
Therefore
\begin{eqnarray*}
\|\Lphet w_{\ell} \|_{L^2(\RnIp)} &\gtrsim& \|(h\partial_y - T_{G_{+}(|\tilde{K}|,\tilde{K},\xi)})(h\partial_y - T_{G^{\e}_{-}(\g,\b,\xi)})w_{\ell}\|_{L^2(\RnIp)} \\
                                     & & - h\|w_{\ell}\|_{H^1(\RnIp)} - \d\|(h\partial_y - T_{G^{\e}_{-}(\g,\b, \xi)})w_{\ell}\|_{H^1(\RnIp)} .\\
\end{eqnarray*}
Now we can check that $G_{+}(|\tilde{K}|,\tilde{K},\xi)$ satisfies the necessary properties of $F$ from Section 5, so
\begin{eqnarray*}
\|\Lphet w_{\ell} \|_{L^2(\RnIp)} &\gtrsim& \|(h\partial_y - T_{G^{\e}_{-}(\g,\b, \xi)})w_{\ell}\|_{H^1(\RnIp)} \\
                                     & & - h\|w_{\ell}\|_{H^1(\RnIp)} - \d\|(h\partial_y - T_{G^{\e}_{-}(\g,\b,\xi)})w_{\ell}\|_{H^1(\RnIp)} .\\
\end{eqnarray*}
Then for small enough $\d$,
\begin{equation*}
\begin{split}
\|\Lphet w_{\ell} \|_{L^2(\RnIp)} &\gtrsim \|(h\partial_y - T_{G^{\e}_{-}(\g,\b, \xi)})w_{\ell}\|_{H^1(\RnIp)} - h\|w_{\ell}\|_{H^1(\RnIp)} \\
                                  &\gtrsim h^{\half}\|(h\partial_y - T_{G^{\e}_{-}(\g,\b, \xi)})w_{\ell}\|_{L^2(\Rno)} - h\|w_{\ell}\|_{H^1(\RnIp)}.
\end{split}
\end{equation*}

Now by \eqref{tildeBC},
\[
h\partial_y w = \frac{w + \b \cdot h \grad_x w + h\s w}{1 + |\g|^2} 
\]
on $\partial \Rnp$, so
\[
h\partial_y w_{\ell} = \frac{w_{\ell} + \b \cdot \grad_x w_{\ell}}{1 + |\g|^2} + hE_0 w
\]
on $\partial \Rnp$, where $E_0$ is bounded from $L^2(\R^{n-1})$ to $L^2(\R^{n-1})$.  Therefore
\begin{eqnarray*}
\|\Lphet w_{\ell}\|_{L^2(\Rnp)} &\gtrsim& h^{\half}\left\| \frac{w_{\ell}+\b \cdot \grad_x w_{\ell}}{1+|\g|^2} -T_{G^{\e}_{-}(\g,\b,\xi)}w_{\ell} \right\|_{L^2(\partial \Rnp)}\\
                                & &       - h\|w_{\ell}\|_{H^1(\Rnp)} - h^{\frac{3}{2}}\|w\|_{L^2(\partial \Rn)} \\
                                &\gtrsim& h^{\half} \|w_{\ell} \|_{\dot{H}^1(\partial \Rnp)} - h\|w_{\ell}\|_{H^1(\Rnp)} - h^{\frac{3}{2}}\|w\|_{L^2(\partial \Rnp)}.      
\end{eqnarray*}

Now 
\[
\|w_{\ell}\|_{H^1(\Rnp)} \lesssim \|w\|_{H^1(\Rnp)}
\]
and
\[
\|\Lphet w_{\ell}\|_{L^2(\Rnp)} \lesssim \|\Lphet w \|_{L^2(\Rnp)} + h\|w\|_{H^1(\Rnp)}.
\]
Therefore
\[
\|\Lphet w \|_{L^2(\Rnp)} + h\|w\|_{H^1(\Rnp)} + h^{\frac{3}{2}}\|w\|_{L^2(\partial \Rnp)}  \gtrsim h^{\half} \|w_{\ell} \|_{\dot{H}^1(\partial \Rnp)}
\]
as desired.

\end{proof}

Now combing the results of Propositions \ref{smallprop} and \ref{largeprop} gives 
\[
h^{\half} \|w_{\ell}\|_{\dot{H}^1(\partial \Rnp)} + h^{\half} \|w_{s}\|_{\dot{H}^1(\partial \Rnp)} \lesssim \|\Lphet w \|_{L^2(\Rnp)} + h\|w\|_{H^1(\RnIp)} + h^{\frac{3}{2}}\|w\|_{L^2(\partial \Rnp)}.
\]
Since $w = w_s + w_{\ell}$, we get
\[
h^{\half} \|w\|_{\dot{H}^1(\partial \Rnp)} \lesssim \|\Lphet w \|_{L^2(\Rnp)} + h\|w\|_{H^1(\RnIp)} + h^{\frac{3}{2}}\|w\|_{L^2(\partial \Rnp)}.
\]
for $w \in C^{\infty}(\tilde{M})$ such that $w \equiv 0$ in a neighbourhood of $\tilde{\Gamma}_+$, and $w$ satisfies \eqref{tildeBC}.  A density argument now proves Proposition \ref{SecondChange}, and hence Proposition \ref{MainBndryEst}, at least under the assumptions on $g_0$ and $f$ made at the beginning of this section.  



\vspace{10pt}

\noindent \ref{sec:carleman2}.3. {\bf Finishing the proof.} Now we need to remove the graph conditions on $\Gamma_+^c$, and the conditions on the metric $g_0$.  Since $\Gamma_+$ is a neighbourhood of $\partial M_{+}$, in a small enough neighbourhood $U$ around any point $p$ on $\Gamma_+^c$, $\Gamma_+^c$ coincides locally with a subset of a graph of the form $x_1 = f(x')$, with $M \cap U$ lying in the set $x_1 > f(x')$.  Moreover, for any $\d >0$, if $\grad_{g_0} f(p) = K$, then in some small neighbourhood of $p$, $|\grad_{g_0} f - K|_{g_0} < \d$.  Additionally, since we can choose coordinates at $p$ such that $g_0 = I$ in those coordinates, for any $\d > 0$ we can ensure that there are coordinates such that $|g_0 - I| \leq \d$ in a small neighbourhood of $p$.  We can choose $\d$ to be small enough for Proposition \ref{MainBndryEst} to hold, by the proof in the previous subsection.  

Now we can let $U_j$ be open sets in $M$ such that $\{U_1, \ldots U_m\}$ is a finite open cover of $M$ such that each $M \cap U_j$ has smooth boundary, and each $\Gamma_+^c \cap U_j$ is represented as a graph of the form $x_1 = f_j(x')$, with $|\grad_{g_0} f_j - K_j|_{g_0} < \d_j$, and there is a choice of coordinates on the projection of $M \cap U_j$ in which $|g_0 - I| \leq \d_j$, where $\d_j$ are small enough for 
\[
h^{\half} \|h \grad_t v_j \|_{L^2(\Gamma_+^c \cap U_j)} \lesssim \|\L_{\ph,\e} v_j \|_{L^2(M \cap U_j)} + h\|v_j\|_{H^1(M \cap U_j)} + h^{\frac{3}{2}}\|v_j\|_{L^2(\Gamma_+^c \cap U_j)},  
\]
to hold for all $v_j \in H^2(M \cap U_j)$ such that 
\begin{equation}\label{jBC}
\begin{split}
v_j, \partial_{\nu} v_j &= 0 \mbox{ on } \partial (U_j \cap M) \setminus \Gamma_+^c  \\
h\partial_{\nu} (e^{-\frac{\ph}{h}} v_j) &= h\s e^{-\frac{\ph}{h}} v_j \mbox{ on } \Gamma_+^c \cap U_j.
\end{split}
\end{equation}
Without loss of generality we may assume that each $U_j$ is compactly contained in $U_j^0\times (0,1)$ where $U_j^0$ is a coordinate chart of $M_0$. 

Now let $\chi_1, \ldots \chi_m$ be a partition of unity subordinate to $U_1, \ldots U_m$, and for $w \in H^2(M)$ satisfying \eqref{BC}, define $w_j = \chi_j w$.  Then if $\Gamma_+^c \cap U_j \neq \varnothing$, $w_j$ satisfies \eqref{jBC} for some $\s$, and so 
\[
h^{\half} \|h \grad_t w_j \|_{L^2(\Gamma_+^c \cap U_j)} \lesssim \|\L_{\ph,\e} w_j \|_{L^2(M \cap U_j)} + h\|w_j\|_{H^1(M \cap U_j)} + h^{\frac{3}{2}}\|w_j\|_{L^2(\Gamma_+^c \cap U_j)}.
\]
Adding together these estimates gives
\[
h^{\half} \|h \grad_t w \|_{L^2(\Gamma_+^c)} \lesssim \sum_{j=1}^m \|\L_{\ph,\e} w_j \|_{L^2(M)} + h\|w\|_{H^1(M)} + h^{\frac{3}{2}}\|w\|_{L^2(\Gamma_+^c)}.
\]
Now each $\| \Lphe w_j \|_{L^2(M)} = \| \Lphe \chi_j w \|_{L^2(M)}$ is bounded by a constant times $\| \Lphe w \|_{L^2(M)} + h\|w\|_{H^1(M)}$, so
\[
h^{\half} \|h \grad_t w \|_{L^2(\Gamma_+^c)} \lesssim \|\L_{\ph,\e} w \|_{L^2(M)} + h\|w\|_{H^1(M)} + h^{\frac{3}{2}}\|w\|_{L^2(\Gamma_+^c)}.
\]
This finishes the proof of Proposition \ref{MainBndryEst}.

\section{The $k$-form case} \label{sec:carleman3}

We will prove Theorem \ref{ABCCarl} for $u \in \Om^k(M)$ by induction.  If $k=0$, then $i_N u = 0$, so $u_{\perp} = 0$ and $u = u_{\|}$. Then Theorem  \ref{ABCCarl} for $k=0$ becomes the Carleman estimate \eqref{ABCCarl_zeroform} that was established in Section \ref{sec:carleman2}.  

Note that it suffices to prove Theorem \ref{ABCCarl} for $u \in \Om^k(M)$, with the appropriate boundary conditions, for each $k$, and $Q=0$.  Then the final theorem follows by adding the resulting estimates and noting that the extra $h^2Qu$ term on the right can be absorbed into the terms on the left, for sufficiently small $h$.

\vspace{12pt}

\noindent \ref{sec:carleman3}.1. {\bf Proof of Theorem \ref{ABCCarl} for $k \geq 1$.} Suppose $u \in \Om^k(M)$ with $k \geq 1$. First note that if we impose the boundary conditions \eqref{ABC} of  Theorem \ref{ABCCarl}, substituting the result of Proposition \ref{ABCbndryterms} into \eqref{IbyPidentity} gives 
\begin{equation}\label{ABCIbyP}
\begin{split}
\norm{\Delta_{\ph_c} u}^2 = & \norm{Au}^2 + \norm{Bu}^2 + (i[A,B]u|u) - \\
   & 2h^3 (\partial_\nu \ph \nabla_N u_{\perp}|\nabla_N u_{\perp})_{\Gamma_{+}^c}-  h(\partial_\nu \ph (|d\ph|^2 + |\partial_\nu \ph|^2)u_{\|}|u_{\|})_{\Gamma_+^c} + R \\
\end{split}
\end{equation}
where
\[
|R| \leq C\left( Kh^3 \|\nabla' t u_{\|}\|^2_{\Gamma_{+}^c} + \frac{h}{K}\| u_{\|}\|^2_{\Gamma_{+}^c} + \frac{h^3}{K}\|\nabla_N u_{\perp}\|^2_{\Gamma_{+}^c} \right).
\]
Recall also from Proposition \ref{ABCbndryterms} that the non-boundary terms $\norm{Au}^2 + \norm{Bu}^2 + (i[A,B]u|u)$ satisfy 
\begin{equation}\label{NonBndryTermsABC}
\norm{Au}^2 + \norm{Bu}^2 + (i[A,B]u|u) \gtrsim \frac{h^2}{\e}\|u\|_{H^1(M)}^2 - \frac{h^3}{\e}(\|u_{\|}\|_{H^1(\partial M)}^2 + \|h\nabla_{N} u_{\perp}\|_{L^2(\partial M)}^2)
\end{equation}
for $h \ll \e \ll 1$.
We now return to \eqref{ABCIbyP} and examine the boundary terms.  On $\Gamma_{+}^c$, there exists $\e_1 > 0$ such that $\partial_{\nu} \ph < -\e_1$.  Using this together with \eqref{ABCIbyP} and \eqref{NonBndryTermsABC} gives 
\begin{equation*}
\begin{split}
        &\norm{\Delta_{\ph_c} u}^2 + Kh^3 \|\nabla' t u_{||}\|^2_{\Gamma_+^c} + \frac{h}{K}\|u_{||}\|^2_{\Gamma_+^c} + \frac{h^3}{K}\|\nabla_N u_{\perp}\|^2_{\Gamma_+^c} \\
\gtrsim & \frac{h^2}{\e}\|u\|^2_{H^1(M)} + h^3 \|\nabla_N u_{\perp}\|^2_{\Gamma_+^c} + h\|u_{||}\|^2_{\Gamma_+^c}, \\
\end{split}
\end{equation*}
for large enough $K$.  The last two terms on the left side can be absorbed into the right side, giving
\begin{equation*}
\norm{\Delta_{\ph_c} u}^2 + Kh^3 \|\nabla' t u_{||}\|^2_{\Gamma_+^c} \gtrsim  \frac{h^2}{\e}\|u\|^2_{H^1(M)} + h^3 \|\nabla_N u_{\perp}\|^2_{\Gamma_+^c} + h\|u_{||}\|^2_{\Gamma_+^c}. 
\end{equation*}

Now we want to analyze the boundary term on the left, and this is the part where we will use induction on $k$:

\begin{lemma} \label{lemma_kformreduction_inductionstep}
If $u \in \Om^k(M)$ and $u$ satisfies the boundary conditions \eqref{ABC}, then
\begin{equation}\label{InductionStep}
h^3 \|\nabla' tu_{||}\|^2_{\Gamma_+^c} \lesssim \norm{\Delta_{\ph_c} u}^2 + h^2 \|u\|^2_{H^1(M)} + h^2\|u_{||}\|^2_{\Gamma_+^c}.
\end{equation}
\end{lemma}

If \eqref{InductionStep} is granted, fix $K$ sufficiently large and then take $h \ll \eps \ll 1$ to obtain that 
\begin{equation*}
\norm{\Delta_{\ph_c} u}^2 \gtrsim \frac{h^2}{\e}\|u\|^2_{H^1(M)} + h^3 \|\nabla_N u_{\perp}\|^2_{\Gamma_+^c} + h\|u_{||}\|^2_{\Gamma_+^c} +h^3 \|\nabla' tu_{||}\|^2_{\Gamma_+^c}.
\end{equation*}
Rewriting without the squares, 
\begin{equation*}
\norm{\Delta_{\ph_c} u} \gtrsim \frac{h}{\sqrt{\e}}\|u\|_{H^1(M)} + h^{\half} \|h\nabla_N u_{\perp}\|_{\Gamma_+^c} + h^{\half}\|u_{||}\|_{H^1(\Gamma_+^c)}.
\end{equation*}
Now if $u$ satisfies \eqref{ABC} then so does $e^{\frac{\ph^2}{2\e}}u$ since $\e$ is fixed.  Therefore
\begin{equation*}
\norm{e^{\frac{\ph^2}{2\e}} \Delta_{\ph} u} \gtrsim \frac{h}{\sqrt{\e}}\|e^{\frac{\ph^2}{2\e}}u\|_{H^1(M)} + h^{\half} \|h\nabla_N e^{\frac{\ph^2}{2\e}}u_{\perp}\|_{\Gamma_+^c} + h^{\half}\|e^{\frac{\ph^2}{2\e}}u_{||}\|_{H^1(\Gamma_+^c)}.
\end{equation*}
Since $e^{\frac{\ph^2}{2\e}}$ is smooth and bounded on $M$, we get
\begin{equation*}
\norm{ \Delta_{\ph} u} \gtrsim h\|u\|_{H^1(M)} + h^{\half} \|h\nabla_N u_{\perp}\|_{\Gamma_+^c} + h^{\half}\|u_{||}\|_{H^1(\Gamma_+^c)}.
\end{equation*}
Thus Theorem \ref{ABCCarl} for $k \geq 1$ will follow after we have proved Lemma \ref{lemma_kformreduction_inductionstep}.


\begin{proof}[Proof of Lemma \ref{lemma_kformreduction_inductionstep}]
For the $0$-form case, this follows from Theorem \ref{ABCCarl} for $0$-forms, which in this section we are assuming has been proved.  Therefore we can seek to prove \eqref{InductionStep} for $k$-forms by induction on $k$.  

Let $k > 0$, and assume that \eqref{InductionStep} holds for $k-1$ forms satisfying \eqref{ABC}.  Now let $U_1, \ldots, U_m \subset T$ be an open cover of $\Gamma_{+}^c$ such that each $U_i \cap \Gamma_+^c$ has a coordinate patch, and let $\chi_1, \ldots, \chi_m$ be a partition of unity with respect to $\{U_i\}$ such that $\sum \chi_i = 1$ near $\Gamma_+^c$ and $\nabla_N \chi_i = 0$ for each $i$.  It will suffice to show that 
\[
h^3 \|\nabla' t\chi_i u_{||}\|^2_{\Gamma_+^c} \lesssim \norm{\Delta_{\ph_c} u}^2 + h^2\|u\|^2_{H^1(M)} + h^2\|u_{||}\|^2_{\Gamma_+^c}.
\]
Now on $U_i \cap \Gamma_+^c$, let $\{e_1,..,e_{n-1}\}$ be an orthonormal frame for the tangent space, and extend these vector fields into $M$ by parallel transport along normal geodesics.  

Observe for all $\omega \in \Om^k(U_j \cap \Gamma_+^c)$ one can write 
\begin{equation}\label{splitting a k form}
\omega = \frac{1}{k} \sum\limits_{j=1}^{n-1} e_j^\flat \wedge i_{e_j} \omega.
\end{equation}
Therefore we can write 
\begin{eqnarray*}
\nabla' t \chi_i u_{||} &=& \frac{1}{k}\nabla' \sum\limits_{j=1}^{n-1} e_j^\flat \wedge i_{e_j} t \chi_i u_{||} \\
                        &=& \frac{1}{k}\nabla' \sum\limits_{j=1}^{n-1} e_j^\flat \wedge t i_{e_j} \chi_i u_{||}. \\
\end{eqnarray*}
Then it suffices to show that 
\[
h^3 \| \nabla' (e_j^\flat \wedge t i_{e_j} \chi_i u_{||})\|^2_{\Gamma_+^c} \lesssim \norm{\Delta_{\ph_c} u}^2 + h^2\|u\|^2_{H^1(M)} + h^2\|u_{||}\|^2_{\Gamma_+^c},
\]
or equivalently, that 
\begin{equation}\label{ReduxInductionStep}
h^3 \| \nabla' t i_{e_j} \chi_i u_{||}\|^2_{\Gamma_+^c} \lesssim \norm{\Delta_{\ph_c} u}^2 + h^2\|u\|^2_{H^1(M)} + h^2\|u_{||}\|^2_{\Gamma_+^c}. 
\end{equation}

Now we want to apply the induction hypothesis to $i_{e_j} \chi_i u_{||}$, so we have to check that it satisfies the boundary conditions \eqref{ABC}.  In fact we will have to modify $i_{e_j} \chi_i u_{||}$ slightly to achieve this.  Let $\rho(x)$ be a function defined in a neighbourhood of the boundary as the distance to the boundary along a normal geodesic, and extend it to the rest of $M$ by multiplication by a cutoff function.  Then the claim is that $v = i_{e_j} \chi_i (u_{||} + h (1 - e^{\frac{-\rho}{h}}) Z  u_{\|})$ satisfies the absolute boundary conditions \eqref{ABC}, where $Z$ is an endomorphism yet to be chosen.

Since $u$ satisfies \eqref{ABC}, $i_{e_j} \chi_i u_{||}$ and $i_{e_j} \chi_i (h (1 - e^{\frac{-\rho}{h}})Z u_{\|}))$ both vanish to first order on $\Gamma_+$.  Therefore $v$ does as well.  

Moreover, $t* i_{e_j} \chi_i u_{||} = 0$ on $\Gamma_+^c$ if $i_N i_{e_j} \chi_i u_{\|} = -\chi_i i_{e_j} i_N u_{\|} = 0$ on $\Gamma_+^c$, and this again follows from the fact that $u$ satisfies \eqref{ABC}.  Note that $(1 - e^{\frac{-\rho}{h}}) = 0$ at $\partial M$, so $t*v = 0$ on $\Gamma_+^c$.

Finally, by Lemma \ref{Tdelta}, 
\[
-t \delta * i_{e_j} \chi_i u_{||} =  -\delta ' t (* i_{e_j} \chi_i u_{||})_{\|} + (S - (n-1)\kappa)ti_{N}(* i_{e_j} \chi_i u_{||})_{\perp} + t \grad_{N} i_N * i_{e_j} \chi_i u_{||}.
\]
Since $t *i_{e_j}  \chi_i u_{||} = 0$ on $\Gamma_+^c$, the first term vanishes there as well.  Therefore on $\Gamma_+^c$,
\[
-t h\delta * i_{e_j} \chi_i u_{||} = h(S - (n-1)\kappa)ti_{N}(* i_{e_j} \chi_i u_{||})_{\perp} + t h\grad_{N} \chi_i i_N * i_{e_j} u_{||}.
\]
Now 
\begin{equation*}
\begin{split}
t h\grad_{N} \chi_i i_N * i_{e_j} u_{||} &= t h\grad_{N} \chi_i i_N e_j^{\flat} \wedge * u_{||} (-1)^{k-1} \\
                                         &= t h\grad_{N} \chi_i i_N e_j^{\flat} \wedge (* u)_{\perp} (-1)^{k-1}\\
                                         &= t h\grad_{N} \chi_i i_N e_j^{\flat} \wedge * u (-1)^{k-1}\\
                                         &= (-1)^{k} \chi_i e_j^{\flat} \wedge t h\grad_{N} i_N * u, \\
\end{split}
\end{equation*}
so
\begin{equation}\label{thdeltastariejchiuparallel_expansion}
-t h\delta * i_{e_j} \chi_i u_{||} = h(S - (n-1)\kappa)ti_{N}(* i_{e_j} \chi_i u_{||})_{\perp} + (-1)^{k} \chi_i e_j^{\flat} \wedge t h\grad_{N} i_N * u.
\end{equation}
Applying the same calculation to the $i_{e_j} \chi_i  h(1 - e^{\frac{-\rho}{h}})Z u_{\|}$ term gives
\[
-t h\delta * i_{e_j} i_{e_j} \chi_i  h(1 - e^{\frac{-\rho}{h}})Z u_{\|} = (-1)^{k} \chi_i e_j^{\flat} \wedge t h^2\grad_{N} (1 - e^{\frac{-\rho}{h}})  i_N * Z u_{\|};
\]
the other term vanishes since $(1 - e^{\frac{-\rho}{h}}) = 0$ at the boundary.  Thus
\[
-t h\delta * i_{e_j} i_{e_j} \chi_i  h(1 - e^{\frac{-\rho}{h}})Z u_{\|} = (-1)^{k} \chi_i e_j^{\flat} \wedge t i_N * hZ  u_{\|}
\]
Meanwhile, by Lemma \ref{Tdelta} and by \eqref{ABC}, 
\[
-t h\delta * u = h(S - (n-1)\kappa)ti_{N}(*u)_{\perp} + t h\grad_{N} i_N (* u) = t i_{d\ph} *u - h \s t i_{N} *u .
\]
Viewing this as an equation for $t h\grad_{N} i_N (* u)$ and substituting into \eqref{thdeltastariejchiuparallel_expansion} gives
\begin{multline*}
-t h\delta * i_{e_j} \chi_i u_{||} = h(S - (n-1)\kappa)ti_{N}(* i_{e_j} \chi_i u_{||})_{\perp} \\
 + (-1)^{k}\chi_i e_j^{\flat} \wedge( -h(S - (n-1)\kappa)ti_{N}(*u)_{\perp} +t i_{d\ph} *u -h \s ti_{N} *u).
\end{multline*}
Therefore
\begin{multline*}
-t h\delta * i_{e_j} \chi_i (u_{||} + h (1 - e^{\frac{-\rho}{h}})Z u_{\|}) = h(S - (n-1)\kappa)ti_{N}(* i_{e_j} \chi_i u_{||})_{\perp} \\
 + (-1)^{k}\chi_i e_j^{\flat} \wedge( -h(S - (n-1)\kappa)ti_{N}(*u)_{\perp} +t i_{d\ph} *u -h \s ti_{N} *u + t i_N * hZ  u_{\|}).
\end{multline*}
Now if we let 
\[
Z = * N \wedge (S +\sigma - (n-1)\kappa) i_N *,
\]
where here we identify $S$ and $\sigma$ with their extensions by parallel transport to a neighbourhood of the boundary, then
\[
t i_N * Z u_{\|} = (S +\sigma - (n-1)\kappa) ti_N * u_{\|} = (S +\sigma - (n-1)\kappa) t i_N * u,
\]
and 
\[
-t h\delta * i_{e_j} \chi_i (u_{||} + h (1 - e^{\frac{-\rho}{h}})Z u_{\|}) = h(S - (n-1)\kappa)ti_{N}(* i_{e_j} \chi_i u_{||})_{\perp} + (-1)^{k}\chi_i e_j^{\flat} \wedge t i_{d\ph} *u.
\]

Since $t*u = 0$ on $\Gamma_+^c$, we can replace the $d\ph$ in $t i_{d\ph} * u$ with its normal component:
\[
t i_{d\ph} * u = -\partial_{\nu} \ph t i_{N} * u.
\]
Then
\begin{equation*}
\begin{split}
\chi_i e_j^{\flat} \wedge -t i_{d\ph} *u &= \partial_{\nu} \ph \chi_i e_j^{\flat} \wedge t i_{N} (*u)_{\perp} \\
                                 &= \partial_{\nu} \ph \chi_i e_j^{\flat} \wedge t i_{N} * u_{\|} \\  
                                 &= -\partial_{\nu} \ph t i_{N} \chi_i e_j^{\flat} \wedge * u_{\|} \\  
                                 &= \partial_{\nu} \ph t i_{N} * i_{e_j} \chi_i u_{\|} (-1)^{k}. \\  
\end{split}
\end{equation*}
Since $t *i_{e_j}  \chi_i u_{||} = 0$ on $\Gamma_+^c$,
\[
\chi_i e_j^{\flat} \wedge -t i_{d\ph} *u = -t i_{d\ph} * i_{e_j} \chi_i u_{\|} (-1)^{k},
\]
and
\[
(-1)^{k} \chi_i e_j^{\flat} \wedge -t i_{d\ph} *u = -t i_{d\ph} * i_{e_j} \chi_i u_{\|}.
\]
Therefore  
\[
-t h\delta * i_{e_j} \chi_i (u_{||} + h (1 - e^{\frac{-\rho}{h}})Z u_{\|}) = t i_{d\ph} * i_{e_j} \chi_i u_{||} -h \s ' t i_{N} *  i_{e_j} \chi_i u_{||}
\]
where $\sigma '$ is a smooth bounded endomorphism.  We can replace $u_{\|}$ on the right side by $u_{||} + h (1 - e^{\frac{-\rho}{h}})Z u_{\|}$, since $(1 - e^{\frac{-\rho}{h}})$ is zero at the boundary. Therefore $v = i_{e_j} \chi_i (u_{||} + h (1 - e^{\frac{-\rho}{h}})Z u_{\|})$ satisfies the boundary conditions \eqref{ABC}, and so by the induction hypothesis,
\[
h^3 \|\nabla' t v\|^2_{\Gamma_+^c} \lesssim \norm{\Delta_{\ph_c} v}^2 + h^2\|v\|^2_{H^1(M)} + h^2\| v\|^2_{\Gamma_+^c},
\]
Keeping in mind that the second term of $v$ is zero at the boundary, and $O(h)$ elsewhere, we get
\begin{equation}\label{first_nablaprimetiejchiiuparallel_bound}
h^3 \|\nabla' t i_{e_j} \chi_i u_{||}\|^2_{\Gamma_+^c} \lesssim \norm{\Delta_{\ph_c} v}^2 + h^2\|u_{\|}\|^2_{H^1(M)} + h^2\| u_{\|} \|^2_{\Gamma_+^c}.
\end{equation}
Now
\[
\norm{\Delta_{\ph_c} v} \lesssim \|\Delta_{\ph_c} i_{e_j} \chi_i u_{||}\| + h\| \Delta_{\ph_c} i_{e_j} \chi_i (1 - e^{\frac{-\rho}{h}})Z u_{\|}\|.
\]
The commutators of $\Delta_{\ph_c}$ with $i_{e_j} \chi_i$ and $ i_{e_j} \chi_i (1 - e^{\frac{-\rho}{h}})Z$ are $O(h)$ and first order, so
\begin{eqnarray*}
\norm{\Delta_{\ph_c} v} &\lesssim& \|i_{e_j} \chi_i \Delta_{\ph_c}u_{||}\| + h\| i_{e_j} \chi_i (1-e^{\frac{-\rho}{h}})Z \Delta_{\ph_c}  u_{\|}\| + h\|u_{\|}\|_{H^1(M)} \\
                        &\lesssim& \|\Delta_{\ph_c}u_{||}\| + h\|u_{\|}\|_{H^1(M)}. \\ 
\end{eqnarray*}
Substituting back into \eqref{first_nablaprimetiejchiiuparallel_bound} gives
\begin{equation*}
h^3 \|\nabla' t i_{e_j} \chi_i u_{||}\|^2_{\Gamma_+^c} \lesssim \norm{\Delta_{\ph_c} u_{||}}^2 + h^2\| u_{||}\|^2_{H^1(M)} + h^2\| u_{||}\|^2_{\Gamma_+^c}.
\end{equation*}

This proves \eqref{ReduxInductionStep}, which finishes the proof of the lemma.
\end{proof}

\section{Complex geometrical optics solutions} \label{sec:cgo}

We will begin by constructing CGOs for the relative boundary case.  To start, we can use the Carleman estimate from Theorem \ref{RBCCarl} to generate solutions via a Hahn-Banach argument. The notations are as in Section \ref{sec:results}.

\begin{prop}\label{RBCHahnBanach}
Let $Q$ be an $L^{\infty}$ endomorphism on $\Lambda M$, and let $\Gamma_{+}$ be a neighbourhood of $\partial M_{+}$.  For all $v \in L^2(M, \Lambda M)$, and $f,g \in L^2(M, \Lambda \partial M)$ with support in $\Gamma_{+}^c$, there exists $u \in L^2(M, \Lambda M)$ such that 
\begin{equation*}
\begin{split}
(-\Lap_{-\ph} + h^2Q^{*})u &= v \mbox{ on } M\\
                  tu &= f \mbox{ on } \Gamma_{+}^c\\
    th\delta_{-\ph}u &= g \mbox{ on } \Gamma_{+}^c, \\ 
\end{split}    
\end{equation*}
with
\[
\|u\|_{L^2(M)} \lesssim h^{-1}\|v\|_{L^2(M)} + h^{\half}\|f\|_{L^2(\Gamma^c_{+})} + h^{\half}\|g\|_{L^2(\Gamma_+^c)}.
\]
\end{prop}

\begin{proof}
Suppose $w \in \Om(M)$ satisfies the relative boundary conditions \eqref{RBC} with $\s = 0$, and examine the expression
\begin{equation}\label{fxnal}
|(w|v) - (t i_{\nu}hd_{\ph} w|hf)_{\Gamma_+^c} - (ti_{\nu} w|hg)_{\Gamma_+^c}|.
\end{equation}
This is bounded above by
\[
h\|w\|_{L^2(M)} h^{-1}\|v\|_{L^2(M)} + h^{\half}\|t i_{\nu}hd_{\ph} w\|_{L^2(\Gamma_+^c)} h^{\half}\|f\|_{L^2(\Gamma_+^c)} + h^{\half}\|ti_{\nu} w\|_{L^2(\Gamma_+^c)} \|g\|_{L^2(\Gamma_+^c)}.
\]
By Lemma \ref{iNdu}, 
\[
t i_{\nu}hd_{\ph} w = he^{\frac{\ph}{h}}t\grad_N (e^{-\frac{\ph}{h}}w)_{\|} + hStw_{\|} - he^{\frac{\ph}{h}}d'ti_N(e^{-\frac{\ph}{h}}w).
\]
Since $tw = 0$, 
\[
t i_{\nu}hd_{\ph} w = ht\grad_N w_{\|} - he^{\frac{\ph}{h}}d'ti_N(e^{-\frac{\ph}{h}}w).
\]
Therefore
\[
\|t i_{\nu}hd_{\ph} w\|_{L^2(\Gamma_+^c)} \leq \|h\grad_N w_{\|}\|_{L^2(\Gamma_+^c)} + \|w_{\perp}\|_{H^1(\Gamma_+^c)}.
\]
Then by Theorem \ref{RBCCarl},
\begin{eqnarray*}
& &        |(w|v) + (t i_{\nu}hd_{\ph} w|hf)_{\Gamma_+^c} + (ti_{\nu} w|hg)_{\Gamma_+^c}| \\
&\lesssim& \|(-\Lap_{\ph} + h^2Q)w\|_{L^2(M)}\left( h^{-1}\|v\|_{L^2(M)} + h^{\half}\|f\|_{L^2(\Gamma_+^c)} + h^{\half}\|g\|_{L^2(\Gamma_+^c)} \right) . \\
\end{eqnarray*}
Therefore on the subspace 
\[
\{(-\Lap_{\ph} + h^2Q)w| w \in \Om(M) \mbox{ satisfies \eqref{RBC} with } \s = 0 \} \subset L^2(M, \Lambda M),
\]
the map 
\[
(-\Lap_{\ph} + h^2Q)w \mapsto (w|v) - (t i_{\nu}hd_{\ph} w|hf)_{\Gamma_+^c} - (ti_{\nu} w|hg)_{\Gamma_+^c}
\]
defines a bounded linear functional with the bound 
\[
h^{-1}\|v\|_{L^2(M)} + h^{\half}\|f\|_{L^2(\Gamma_+^c)} + h^{\half}\|g\|_{L^2(\Gamma_+^c)}.
\]
By Hahn-Banach, this functional extends to the whole space, and thus there exists a $u \in L^2(M, \Lambda M)$ such that
\[
\|u\|_{L^2(M)} \lesssim h^{-1}\|v\|_{L^2(M)} + h^{\half}\|f\|_{L^2(\Gamma_+^c)} + h^{\half}\|g\|_{L^2(\Gamma_+^c)} 
\]
and
\[
(w|v) - (t i_{\nu}hd_{\ph} w|hf)_{\Gamma_+^c} - (ti_{\nu} w|hg)_{\Gamma_+^c} = ((-\Lap_{\ph} + h^2Q)w | u).
\]
Integrating by parts and applying the boundary conditions \eqref{RBC} gives
\begin{eqnarray*}
& & (w|v) - (t i_{\nu}hd_{\ph} w|hf)_{\Gamma_+^c} - (ti_{\nu} w|hg)_{\Gamma_+^c} \\
&=& (w |(-\Lap_{-\ph} + h^2Q^{*})u) - h(t i_{\nu}hd_{\ph} w|tu)_{\partial M} - h(ti_{\nu} w|th\delta_{-\ph}u)_{\partial M}
\end{eqnarray*}
for all $w \in \Om(M)$ satisfying the relative boundary conditions \eqref{RBC} with $\s = 0$. Varying $w$ over the compactly supported elements of $\Om(M)$ one sees that $(-\Lap_{-\ph} + h^2Q^{*})u = v$ on $M$ which reduces the above relation to 
\begin{eqnarray*}
 - (t i_{\nu}hd_{\ph} w|hf)_{\Gamma_+^c} - (ti_{\nu} w|hg)_{\Gamma_+^c}  =  - h(t i_{\nu}hd_{\ph} w|tu)_{\partial M} - h(ti_{\nu} w|th\delta_{-\ph}u)_{\partial M}
\end{eqnarray*}
for all $w \in \Om(M)$ satisfying the relative boundary conditions \eqref{RBC} with $\s = 0$. We now vary $w$ satisfying condition \eqref{RBC} with $\s = 0$ and $i_\nu w = 0$ to obtain $tu = f$ on $\Gamma_{+}^c$. Finally, by varying $w$ over all forms satisfying conditions \eqref{RBC} with $\s = 0$ we see that $th\delta_{-\ph}u = g$ on $\Gamma_{+}^c$.

To summarize, we can see that 
\begin{equation*}
\begin{split}
(-\Lap_{-\ph} + h^2Q^{*})u &= v \mbox{ on } M\\
                  tu &= f \mbox{ on } \Gamma_{+}^c\\
    th\delta_{-\ph}u &= g \mbox{ on } \Gamma_{+}^c, \\ 
\end{split}    
\end{equation*}
as desired.  
\end{proof}

To match notations with previous papers, we will begin by rewriting this result, along with the Carleman estimate, in $\tau$ notation, as follows.  

Theorem \ref{RBCCarl} becomes the following.

\begin{theorem}\label{RBCCarltau}
Let $Q$ be an $L^{\infty}$ endomorphism on $\Lambda M$.  Define $\Gamma_+ \subset \partial M$ to be a neighbourhood of $\partial M_{+}$. Suppose $u \in H^2(M, \Lambda M)$ satisfies the boundary conditions
\begin{equation}\label{RBCtau}
\begin{split}
u|_{\Gamma_+} &= 0\ \  \mbox{and}\ \  \nabla_\nu u\mid_{\Gamma_+} = 0 \\
tu|_{\Gamma_+^c} &= 0 \\
t\delta e^{-\tau \ph}u|_{\Gamma_+^c} &= \sigma ti_N e^{-\tau \ph} u
\end{split}
\end{equation}
for some smooth endomorphism $\s$ independent of $\tau$.
Then there exists $\tau_0 > 0$ such that if $\tau > \tau_0$,
\[
\|(-\Delta_{\tau} + Q)u\|_{L^2(M)} \gtrsim \tau \|u\|_{L^2(M)} + \|\grad u\|_{L^2(M)} +\tau^{\frac{3}{2}}\|u_{\perp}\|_{L^2(\Gamma_+^c)} +\tau^{\half}\|\grad ' ti_N u\|_{L^2(\Gamma_+^c)} +\tau^{\half}\| \grad_{N} u_{\|}\|_{L^2(\Gamma_+^c)}.
\]
where 
\[
\Delta_{\tau} = e^{\tau \ph}\Delta e^{-\tau \ph}.
\]
\end{theorem}

By choice of coordinates, note that the same theorem holds for $\tau < 0$, with $\Gamma_{+}$ replaced by $\Gamma_{-}$.  

Then Proposition \ref{RBCHahnBanach} becomes the following.

\begin{prop}\label{RBCHahnBanachtau}
Let $Q$ be an $L^{\infty}$ endomorphism on $\Lambda M$.  For all $v \in L^2(M, \Lambda M)$, and $f,g \in L^2(\Gamma_{+}^c, \Lambda \Gamma_{+}^c)$, there exists $u \in L^2(M, \Lambda M)$ such that 
\begin{equation*}
\begin{split}
(-\Lap_{-\tau} + Q^{*})u &= v \mbox{ on } M\\
                      tu &= f \mbox{ on } \Gamma_{+}^c \\
        t\delta_{-\tau}u &= g \mbox{ on } \Gamma_{+}^c, \\ 
\end{split}    
\end{equation*}
with
\[
\|u\|_{L^2(M)} \lesssim \tau^{-1} \| v\|_{L^2(M)} + \tau^{-\half}\|f\|_{L^2(\Gamma_{+}^c)} + \tau^{-\frac{3}{2}}\|g\|_{L^2(\Gamma_{+}^c)}.
\]
\end{prop}


Now we turn to the construction of the CGOs themselves. From now on we will invoke the assumption that the conformal factor $c$ in the definition of $M$ as an admissible manifold satisfies $c=1$. Below we will consider complex valued $1$-forms, and $\langle \,\cdot\,,\,\cdot\, \rangle$ will denote the complex bilinear extension of the Riemannian inner product to complex valued forms.

We assume that 
$$
(M,g) \subset \subset (\mR \times M_0, g), \quad g = e \oplus g_0,
$$
where $(M_0,g_0)$ is a compact $(n-1)$-dimensional manifold with smooth boundary. We write $x = (x_1,x')$ for points in $\mR \times M_0$, where $x_1$ is the Euclidean coordinate and $x'$ is a point in $M_0$. Let $Q$ be an $L^{\infty}$ endomorphism of $\Lambda M$. We next wish to construct solutions to the equation 
$$
(-\Delta + Q)Z = 0 \text{ in } M
$$
where $Z$ is a graded differential form in $L^2(M, \Lambda M)$ having the form 
$$
Z = e^{-s x_1}(A + R).
$$
Here $s = \tau + i\lambda$ is a complex parameter where $\tau, \lambda \in \mR$ and $\abs{\tau}$ is large, the graded form $A$ is a smooth amplitude, and $R$ will be a correction term obtained from the Carleman estimate. Inserting the expression for $Z$ in the equation results in 
$$
e^{s x_1}(-\Delta+Q) e^{-s x_1} R = -F
$$
where 
$$
F = e^{s x_1}(-\Delta+Q) e^{-s x_1} A.
$$
The point is to choose $A$ so that $\norm{F}_{L^2(M)} = O(1)$ as $\abs{\tau} \to \infty$.

By Lemma \ref{lemma_conjugated_hodge_expression}, we have 
$$
F = (-\Delta - s^2 + 2s \nabla_{\partial_1} + Q)A.
$$
We wish to choose $A$ so that $\nabla_{\partial_1} A = 0$. The following lemma explains this condition. Below, we identify a differential form in $M_0$ with the corresponding differential form in $\mR \times M_0$ which is constant in $x_1$.

\begin{lemma} \label{lemma_product_differentialform_decomposition}
If $u$ is a $k$-form in $\mR \times M_0$ with local coordinate expression $u = u_I \,dx^I$, then 
$$
\nabla_{\partial_1} u = 0 \quad \Longleftrightarrow \quad u_I = u_I(x') \text{ for all } I.
$$
If $\nabla_{\partial_1} u = 0$, then there is a unique decomposition 
$$
u = dx^1 \wedge u' + u''
$$
where $u'$ is a $(k-1)$-form in $M_0$ and $u''$ is a $k$-form in $M_0$. For such a $k$-form $u$, one has 
$$
\Delta u = dx^1 \wedge \Delta_{x'} u' + \Delta_{x'} u''
$$
where $\Delta$ and $\Delta_{x'}$ are the Hodge Laplacians in $\mR \times M_0$ and in $M_0$, respectively.
\end{lemma}
\begin{proof}
In the $(x_1,x')$ coordinates $g$ has the form 
$$
g(x_1,x') = \left( \begin{array}{cc} 1 & 0 \\ 0 & g_0(x') \end{array} \right).
$$
Consequently, for any $k, l$ the Christoffel symbols satisfy 
$$
\Gamma_{1k}^l = \frac{1}{2} g^{lm}(\partial_1 g_{km} + \partial_k g_{1m} - \partial_m g_{1k}) = 0.
$$
This shows that $\nabla_{\partial_1} \,dx^I = 0$ for all $I$, and therefore any $k$-form $u = u_I \,dx^I$ satisfies 
$$
\nabla_{\partial_1} (u_I \,dx^I) = \partial_1 u_I \,dx^I.
$$
Thus $\nabla_{\partial_1} u = 0$ if and only if each $u_I$ only depends on $x'$. In general, if $u$ is a $k$-form on $\mR \times M_0$ we have the unique decomposition 
$$
u = dx^1 \wedge u' + u''
$$
where $u'(x_1,\,\cdot\,)$ is a $(k-1)$-form in $M_0$ and $u''(x_1,\,\cdot\,)$ is a $k$-form in $M_0$, depending smoothly on the parameter $x_1$. If $\nabla_{\partial_1} u = 0$, then $u = dx^1 \wedge u' + u''$ where $u'$ and $u''$ are differential forms in $M_0$.

Suppose now that $u = dx^1 \wedge u' + u''$ where $u'$ and $u''$ are forms in $M_0$. Denote by $d_{x'}$ and $\delta_{x'}$ the exterior derivative and codifferential in $x'$. Clearly 
$$
d(dx^1 \wedge u') = -dx^1 \wedge d_{x'} u', \quad du'' = d_{x'} u''.
$$
The identity $\delta = -\sum_{j=1}^n i_{e_j} \nabla_{e_j}$, where $e_j$ is an orthonormal frame in $T(\mR \times M_0)$ with $e_1 = \partial_1$, together with the fact that $\nabla_{\partial_1} u'' = 0$, implies that 
$$
\delta u'' = \delta_{x'} u''.
$$
Finally, computing in Riemannian normal coordinates at $p$ gives that 
\begin{align*}
\delta(dx^1 \wedge u')|_p &= -\sum_{j=1}^n i_{\partial_j} \nabla_{\partial_j} (u_J' \,dx^1 \wedge dx^J)|_p \\
 &= -\sum_{j=2}^n i_{\partial_j} (dx^1 \wedge \nabla_{\partial_j} u')|_p \\
 &= -dx^1 \wedge \delta_{x'} u'|_p.
\end{align*}
Thus 
$$
\delta(dx^1 \wedge u') = -dx^1 \wedge \delta_{x'} u'.
$$
It follows directly from these facts that 
\begin{align*}
\Delta(dx^1 \wedge u' + u'') &= -(d \delta + \delta d)(dx^1 \wedge u' + u'') \\
 &= dx^1 \wedge \Delta_{x'} u' + \Delta_{x'} u''. \qedhere
\end{align*}
\end{proof}

Returning to the expression for $F$, the assumption $\nabla_{\partial_1} A = 0$ gives that 
$$
F = (-\Delta-s^2+Q)A.
$$
Writing $Y^k$ for the $k$-form part of a graded form $Y$, and decomposing $A^k = dx^1 \wedge (A^k)' + (A^k)''$ as in Lemma \ref{lemma_product_differentialform_decomposition}, we obtain that 
$$
F^k = dx^1 \wedge (-\Delta_{x'} - s^2)(A^k)' + (-\Delta_{x'}-s^2)(A^k)'' + (QA)^k.
$$
Thus, in order to have $\norm{F}_{L^2(M)} = O(1)$ as $\abs{\tau} \to \infty$, it is enough to find for each $k$ a smooth $(k-1)$-form $(A^k)'$ and a smooth $k$-form $(A^k)''$ in $M_0$ such that 
\begin{gather*}
\norm{(-\Delta_{x'}-s^2)(A^k)'}_{L^2(M_0)} = O(1), \quad \norm{(A^k)'}_{L^2(M_0)} = O(1), \\
\norm{(-\Delta_{x'}-s^2)(A^k)''}_{L^2(M_0)} = O(1), \quad \norm{(A^k)''}_{L^2(M_0)} = O(1).
\end{gather*}
If $(M_0,g_0)$ is simple, there is a straightforward quasimode construction for achieving this.

\begin{lemma} \label{lemma_form_quasimode}
Let $(M_0,g_0)$ be a simple $m$-dimensional manifold, and let $0 \leq k \leq m$. Suppose that $(\hat{M}_0,g_0)$ is another simple manifold with $(M_0,g_0) \subset \subset (\hat{M}_0,g_0)$, fix a point $\omega \in \hat{M}_0^{\rm int} \setminus M_0$, and let $(r,\theta)$ be polar normal coordinates in $(\hat{M}_0,g_0)$ with center $\omega$. Suppose that $\eta^1, \ldots, \eta^m$ is a global orthonormal frame of $T^* M_0$ with $\eta^1 = dr$ and $\nabla_{\partial_r} \eta^j = 0$ for $2 \leq j \leq m$, and let $\{ \eta^I \}$ be a corresponding orthonormal frame of $\Lambda^k M_0$. Then for any $\lambda \in \mR$ and for any $\binom{m}{k}$ complex functions $b_I \in C^{\infty}(S^{m-1})$, the smooth $k$-form 
$$
u = e^{isr} \abs{g_0(r,\theta)}^{-1/4} \sum_I b_I(\theta) \eta^I,
$$
with $s = \tau + i\lambda$ for $\tau$ real, satisfies as $\abs{\tau} \to \infty$ 
$$
\norm{(-\Delta_{x'}-s^2)u}_{L^2(M_0)} = O(1), \quad \norm{u}_{L^2(M_0)} = O(1).
$$
\end{lemma}
\begin{proof}
We first try to find the quasimode in the form $u = e^{is\psi} a$ for some smooth real valued phase function $\psi$ and some smooth $k$-form $a$. Lemma \ref{lemma_conjugated_hodge_expression} implies that 
\begin{equation*}
(-\Delta_{x'}-s^2) (e^{is\psi} a) = e^{is\psi} \big[ s^2 (\abs{d\psi}^2-1) a -is \left[ 2\nabla_{{\rm grad}(\psi)} a + (\Delta_{x'} \psi) a \right] - \Delta_{x'} a \big].
\end{equation*}
Let $(r,\theta)$ be polar normal coordinates as in the statement of the lemma, and note that 
$$
g_0(r,\theta) = \left( \begin{array}{cc} 1 & 0 \\ 0 & h(r,\theta) \end{array} \right)
$$
globally in $M_0$ for some $(m-1) \times (m-1)$ symmetric positive definite matrix $h$.

Define 
$$
\psi(r,\theta) = r.
$$
Then $\psi \in C^{\infty}(M_0)$ and $\abs{d\psi}^2 = 1$, so that the $s^2$ term will be zero.  We next want to choose $a$ so that $ 2\nabla_{{\rm grad}(\psi)} a + (\Delta_{x'} \psi) a = 0$. Note that 
$$
\nabla_{{\rm grad}(\psi)} = \nabla_{\partial_r}, \qquad \Delta_{x'} \psi = \frac{1}{2} \frac{\partial_r \abs{g_0(r,\theta)}}{\abs{g_0(r,\theta)}}.
$$
Thus, choosing $a = \abs{g_0}^{-1/4} \tilde{a}$ for some $k$-form $\tilde{a}$, it is enough to arrange that 
$$
\nabla_{\partial_r} \tilde{a} = 0.
$$
Using the frame $\{ \eta^j \}$ above, with $\eta^1 = dr$, we write 
$$
\tilde{a} = \eta^1 \wedge \tilde{a}' + \tilde{a}''
$$
where $\tilde{a}'$ is a $(k-1)$-form and $\tilde{a}''$ is a $k$-form in $M_0$ of the form 
$$
\tilde{a}' = \sum_{\underset{\abs{J}=k-1}{J \subset \{2,\ldots,m\}}} \alpha_{1,J} \eta^J, \quad \tilde{a}'' = \sum_{\underset{\abs{J}=k}{J \subset \{2,\ldots,m\}}} \beta_{J} \eta^J
$$
for some functions $\alpha_{1,J}$ and $\beta_{J}$ in $M_0$. Now, the form of the metric implies that $\nabla_{\partial_r} \eta^1 = 0$, and by assumption $\nabla_{\partial_r} \eta^j = 0$ for $2 \leq j \leq m$. Therefore 
$$
\nabla_{\partial_r} \tilde{a} = \sum_{\underset{\abs{J}=k-1}{J \subset \{2,\ldots,m\}}} \partial_r \alpha_{1,J} \eta^1 \wedge \eta^J + \sum_{\underset{\abs{J}=k}{J \subset \{2,\ldots,m\}}} \partial_r \beta_{J} \eta^J.
$$
In the definitions of $\tilde{a}'$ and $\tilde{a}''$, we may now choose 
$$
\alpha_{1,J} = b_{\{1\} \cup J}(\theta), \quad \beta_{J} = b_J(\theta)
$$
where $b_I$ are the given functions in $C^{\infty}(S^{m-1})$. The resulting $k$-form $u = e^{is\psi} \abs{g_0}^{-1/4} \tilde{a}$ satisfies the required conditions.
\end{proof}

The next result gives the full construction of the complex geometrical optics solutions.

\begin{lemma} \label{lemma_cgo_form_construction}
Let $(M,g) \subset \subset (\mR \times M_0,g)$ where $g = e \oplus g_0$, assume that $(M_0,g_0)$ is simple, and let $Q$ be an $L^{\infty}$ endomorphism of $\Lambda M$. Let $(\hat{M}_0,g_0)$ be another simple manifold with $(M_0,g_0) \subset \subset (\hat{M}_0,g_0)$, fix a point $\omega \in \hat{M}_0^{\rm int} \setminus M_0$, and let $(r,\theta)$ be polar normal coordinates in $(\hat{M}_0,g_0)$ with center $\omega$. Suppose that $\eta^1, \ldots, \eta^n$ is a global orthonormal frame of $T^*(\mR \times M_0)$ with $\eta^1 = dx^1$, $\eta^2 = dr$, and $\nabla_{\partial_r} \eta^j = 0$ for $3 \leq j \leq n$, and let $\{ \eta^I \}$ be a corresponding orthonormal frame of $\Lambda (\mR \times M_0)$. Let also $\lambda \in \mR$. If $\abs{\tau}$ is sufficiently large and if $s = \tau + i\lambda$, then for any $2^n$ complex functions $b_I \in C^{\infty}(S^{n-2})$ there exists a solution $Z \in L^2(M, \Lambda M)$ of the equation 
$$
(-\Delta+Q)Z = 0 \text{ in } M
$$
having the form 
$$
Z = e^{-sx_1} \left[ e^{i s r} \abs{g_0(r,\theta)}^{-1/4} \left[ \sum_I b_I(\theta) \eta^I \right] + R \right]
$$
where $\norm{R}_{L^2(M)} = O(\abs{\tau}^{-1})$. Further, one can arrange that the relative boundary values of $Z$ vanish on $\Gamma_{+}^c$ or $\Gamma_{-}^c$ (depending on the sign of $\tau$).
\end{lemma}
\begin{proof}
Try first $Z = e^{-sx_1}(A+R)$ where $\nabla_{\partial_1} A = 0$. By the discussion in this section, we need to solve the equation 
$$
e^{s x_1}(-\Delta+Q)(e^{-s x_1} R) = -F
$$
where 
$$
F = (-\Delta-s^2+Q)A.
$$
Decomposing the $k$-form part of $A$ as $A^k = \eta^1 \wedge (A^k)' + (A^k)''$ as in Lemma \ref{lemma_product_differentialform_decomposition}, where $\eta^1 = dx^1$, we obtain that 
$$
F^k = \eta^1 \wedge (-\Delta_{x'} - s^2)(A^k)' + (-\Delta_{x'}-s^2)(A^k)'' + (QA)^k.
$$

Let $\eta^1, \ldots, \eta^n$ and $\{ \eta^I \}$ be orthonormal frames as in the statement of the result. We can use Lemma \ref{lemma_form_quasimode} to find, for any $\binom{n-1}{k-1}$ functions $b_J'(\theta)$ and for any $\binom{n-1}{k}$ functions $b_J''(\theta)$, quasimodes 
\begin{gather*}
(A^k)' = e^{isr} \abs{g_0}^{-1/4} \sum_{\underset{\abs{J}=k-1}{J \subset \{2,\ldots,n\}}} b_J'(\theta) \eta^J, \\
(A^k)'' = e^{isr} \abs{g_0}^{-1/4} \sum_{\underset{\abs{J}=k}{J \subset \{2,\ldots,n\}}} b_J''(\theta) \eta^J.
\end{gather*}
Recalling that $A^k = \eta^1 \wedge (A^k)' + (A^k)''$ and relabeling functions, this shows that for any $\binom{n}{k}$ functions $b_I \in C^{\infty}(S^{n-2})$ we may find $A^k$ of the form 
$$
A^k = e^{isr} \abs{g_0}^{-1/4} \sum_{\underset{\abs{I}=k}{I \subset \{1,\ldots,n\}}} b_I(\theta) \eta^I
$$
with $\norm{(-\Delta-s^2)A^k}_{L^2(M)} = O(1)$, $\norm{A^k}_{L^2(M)} = O(1)$ as $\abs{\tau} \to \infty$. Repeating this construction for all $k$, we obtain the amplitude 
$$
A = e^{i s r} \abs{g_0(r,\theta)}^{-1/4} \sum_I b_I(\theta) \eta^I
$$
with the same norm estimates as those for $A^k$. Then also $\norm{F}_{L^2(M)} = O(1)$. Then Proposition \ref{RBCHahnBanachtau} allows us to find $R$ with the right properties. This finishes the proof.
\end{proof}

Note that if $Z$ is a solution to $(-\Delta+*Q*^{-1})Z = 0 $ in $M$, and $Z$ has relative boundary values that vanish on $\Gamma_+^c$, then $*Z$ is a solution to $(-\Delta + Q)*Z = 0$ in $M$, and $*Z$ has absolute boundary values that vanish on $\Gamma_{+}^c$.  Thus this construction also gives us solutions with vanishing absolute boundary values on $\Gamma_+^c$. 



\section{The tensor tomography problem} \label{sec:tensor}

Now we can begin the proof of Theorems \ref{RelativeAbsoluteThm} and \ref{AbsoluteRelativeThm}. First we will use the hypotheses of Theorem \ref{RelativeAbsoluteThm} to obtain some vanishing integrals involving $(Q_2-Q_1)$. 

\begin{lemma}\label{RBCtoDensity}
Suppose the hypotheses of Theorem \ref{RelativeAbsoluteThm} hold.  Using the notation in Lemma \ref{lemma_cgo_form_construction}, let $Z_j \in L^2(M, \Lambda M)$ be solutions of $(-\Delta+Q_1)Z_1 = (-\Delta+\bar{Q}_2)Z_2 = 0$ in $M$ of the form 
\begin{align*}
Z_1 &= e^{-sx_1} \left[ e^{i s r} \abs{g_0}^{-1/4} \left[ \sum_I c_I(\theta) \eta^I \right] + R_1 \right], \\
Z_2 &= e^{sx_1} \left[ e^{i s r} \abs{g_0}^{-1/4} \left[ \sum_I d_I(\theta) \eta^I \right] + R_2 \right]
\end{align*}  
with vanishing relative boundary conditions on $\Gamma_{-}^c$ and $\Gamma_{+}^c$ respectively.   
Then
\[
((Q_2 - Q_1) Z_1| Z_2)_M = 0.
\]
\end{lemma}
Note that while the orthogonality condition derived in the Lemma does not use the particular form of the solution, we will only apply this identity to solutions of the given form.
\begin{proof}
Let $Y$ be a solution of $(-\Delta+Q_2)Y = 0$ in $M$ with the same relative boundary conditions as $Z_1$; such a solution exists by the assumption on $Q_2$.  Then consider the integral
\[
((\RA_{Q_1} - \RA_{Q_2})(tZ_1,t\delta Z_1) | (ti_{N} d * Z_2, ti_{N}*Z_2))_{\partial M}.
\]
By definition of the $\RA$ map, this is
\begin{eqnarray*}
& & ((t*(Z_1-Y),t\delta *(Z_1 - Y)) | (ti_{N} d * Z_2, ti_{N}*Z_2))_{\partial_M} \\
&=& ((t*(Z_1-Y)|ti_{N} d * Z_2)_{\partial M} + (t\delta *(Z_1 - Y) | ti_{N}*Z_2)_{\partial M}. \\
\end{eqnarray*}
Recall from the section on notation and identities that 
\begin{eqnarray*}
(-\Delta u| v)_M &=& (u| -\Delta v)_M +(t u | ti_{\nu}dv)_{\partial M}\\
                 & & +(t \delta *u|ti_{\nu} *v)_{\partial M} + (t *u|ti_{\nu} d*v)_{\partial M} + (t\delta u|ti_{\nu}v)_{\partial M}.
\end{eqnarray*}
Since the relative boundary values of $(Z_1 - Y)$ vanish, by definition, the integration by parts formula above implies that
\begin{eqnarray*}
& & ((t*(Z_1-Y)|ti_{N} d * Z_2)_{\partial M} + (t\delta *(Z_1 - Y) | ti_{N}*Z_2)_{\partial M} \\
&=& (-\Lap (Z_1- Y)| Z_2)_M - (Z_1- Y| -\Lap Z_2)_M \\
&=& (Q_2 Y - Q_1 Z_1| Z_2)_M - (Z_1- Y| -\overline{Q}_2 Z_2)_M \\
&=& ((Q_2 - Q_1) Z_1| Z_2)_M.\\
\end{eqnarray*}
Meanwhile, by the hypothesis on $\RA_{Q_1}$ and $\RA_{Q_2}$, we have that $\RA_{Q_1}(Z_1 - Y) = \RA_{Q_2}(Z_1 - Y)$ on $\Gamma_{+}$.  Therefore
\begin{eqnarray*}
& & ((t*(Z_1-Y)|ti_{N} d * Z_2)_{\partial M} + (t\delta *(Z_1 - Y) | ti_{N}*Z_2)_{\partial M} \\
&=& ((t*(Z_1-Y)|ti_{N} d * Z_2)_{\Gamma_{+}^c} + (t\delta *(Z_1 - Y) | ti_{N}*Z_2)_{\Gamma_{+}^c}.
\end{eqnarray*}
Now by construction, $Z_2$ has relative boundary values that vanish on $\Gamma_{+}^c$.  But 
\begin{equation*}
\begin{split}
ti_{N}*Z_2|_{\Gamma_{+}^c} = 0 &\Leftrightarrow (*Z_2)_{\perp}|_{\Gamma_{+}^c} = 0 \\
                               &\Leftrightarrow *(Z_2)_{\|}|_{\Gamma_{+}^c} = 0 \\
                               &\Leftrightarrow (Z_2)_{\|}|_{\Gamma_{+}^c} = 0 \\
                               &\Leftrightarrow t Z_2|_{\Gamma_{+}^c} = 0. \\
\end{split}
\end{equation*}
Similarly, 
\[
ti_{N} d * Z_2|_{\Gamma_{+}^c} = 0 \Leftrightarrow t\delta * Z_2|_{\Gamma_{+}^c} = 0.
\]
Therefore the fact that $Z_2$ has relative boundary values that vanish on $\Gamma_{+}^c$ implies that 
\[
((t*(Z_1-Y)|ti_{N} d * Z_2)_{\Gamma_{+}^c} + (t\delta *(Z_1 - Y) | ti_{N}*Z_2)_{\Gamma_{+}^c} = 0.
\]
Therefore
\[
((Q_2 - Q_1) Z_1| Z_2)_M = 0
\]
for each such pair of CGO solutions $Z_1$ and $Z_2$.  
\end{proof}
\begin{remarknonum}
The proof of the Lemma \ref{RBCtoDensity} does not use the actual forms of the CGO solutions. The integral identity holds for all solutions $Z_1$ and $Z_2$ with vanishing relative boundary conditions on $\Gamma_{-}^c$ and $\Gamma_{+}^c$ respectively. However, the identity is only of interest to us for the particular forms of CGO solutions which we stated.
\end{remarknonum}
Working through the same argument with $*Z_1$ and $*Z_2$ gives us the following lemma as well.
\begin{lemma}\label{ABCtoDensity}
Suppose the hypotheses of Theorem \ref{AbsoluteRelativeThm} hold.  Using the notation in Lemma \ref{lemma_cgo_form_construction}, let $*Z_j \in L^2(M, \Lambda M)$ be solutions of $(-\Delta+Q_1)*Z_1 = (-\Delta+\bar{Q}_2)*Z_2 = 0$ in $M$ of the form 
\begin{align*}
Z_1 &= e^{-sx_1} \left[ e^{i s r} \abs{g_0}^{-1/4} \left[ \sum_I c_I(\theta) \eta^I \right] + R_1 \right], \\
Z_2 &= e^{sx_1} \left[ e^{i s r} \abs{g_0}^{-1/4} \left[ \sum_I d_I(\theta) \eta^I \right] + R_2 \right].
\end{align*}  
Then
\[
((Q_2 - Q_1) Z_1| Z_2)_M = 0.
\]
\end{lemma}

Therefore both of the main theorems reduce to using the condition $(Q Z_1, Z_2)_{L^2(M)} = 0$ for solutions of the type given in Lemma \ref{lemma_cgo_form_construction} to show that $Q = 0$. 

The next result shows that from the condition $(Q Z_1, Z_2)_{L^2(M)} = 0$ for solutions of the type given in Lemma \ref{lemma_cgo_form_construction}, it follows that certain exponentially attenuated integrals over geodesics in $(M_0,g_0)$ of matrix elements of $Q$, further Fourier transformed in $x_1$, must vanish.

\begin{prop} \label{prop_transform_matrixelements}
Assume the hypotheses in Theorem \ref{RelativeAbsoluteThm} or \ref{AbsoluteRelativeThm}, with $Q = Q_2-Q_1$ extended by zero to $\mR \times M_0$. Fix a geodesic $\gamma: [0,L] \to M_0$ with $\gamma(0), \gamma(L) \in \partial M_0$, let $\partial_r$ be the vector field in $M_0$ tangent to geodesic rays starting at $\gamma(0)$, and suppose that $\{ \eta^I \}$ is an orthonormal frame of $\Lambda(\mR \times M_0^{\rm int})$ with $\eta^1 = dx^1$, $\eta^2 = dr$, and $\nabla_{\partial_r} \eta^j = 0$ for $3 \leq j \leq n$. (Such a frame always exists.) Then for any $\lambda \in \mR$ and any $I, J$ one has 
$$
\int_0^L e^{-2\lambda r} \left[ \int_{-\infty}^{\infty} e^{-2i\lambda x_1} \langle Q(x_1,\gamma(r)) \eta^I, \eta^J \rangle \,dx_1 \right] dr = 0.
$$
\end{prop}
\begin{proof}
Using the notation in Lemma \ref{lemma_cgo_form_construction}, let $Z_j \in L^2(M, \Lambda M)$ be solutions of $(-\Delta+Q_1)Z_1 = (-\Delta+\bar{Q}_2)Z_2 = 0$ in $M$ of the form 
\begin{align*}
Z_1 &= e^{-sx_1} \left[ e^{i s r} \abs{g_0}^{-1/4} \left[ \sum_I c_I(\theta) \eta^I \right] + R_1 \right], \\
Z_2 &= e^{sx_1} \left[ e^{i s r} \abs{g_0}^{-1/4} \left[ \sum_I d_I(\theta) \eta^I \right] + R_2 \right],
\end{align*}
where $s = \tau + i \lambda$, $\tau > 0$ is large, $\lambda \in \mR$, and $c_I, d_I \in C^{\infty}(S^{n-2})$. We can assume that $\norm{R_j}_{L^2(M)} = O(\tau^{-1})$ as $\tau \to \infty$, and that the relative (absolute) boundary values of $Z_1$ are supported in $\tilde{F}$ and the relative (absolute) boundary values of $Z_2$ are supported in $\tilde{B}$. By Lemma \ref{RBCtoDensity} (Lemma \ref{ABCtoDensity}), we have 
\begin{multline*}
0 = \lim_{\tau \to \infty} (Q Z_1, Z_2)_{L^2(M)} = \int_{S^{n-2}} \int_0^{\infty} e^{-2\lambda r} \\
 \times \left[ \sum_{I,J} \left[ \int_{-\infty}^{\infty} e^{-2i\lambda x_1} \langle Q(x_1,r,\theta) \eta^I, \eta^J \rangle \,dx_1 \right] c_I(\theta) \overline{d_J(\theta)} \right] dr \,d\theta.
\end{multline*}

We now extend the $M_0$-geodesic $\gamma$ to $\hat{M}_0$, choose $\omega = \gamma(-\eps)$ for small $\eps > 0$, and choose $\theta_0$ so that $\gamma(t) = (t,\theta_0)$. The functions $c_I$ and $d_J$ can be chosen freely, and by varying them we obtain that 
$$
\int_0^{\infty} e^{-2\lambda r} \left[ \int_{-\infty}^{\infty} e^{-2i\lambda x_1} \langle Q(x_1,r,\theta_0) \eta^I, \eta^J \rangle \,dx_1 \right] dr = 0
$$
for each fixed $I$ and $J$. Since $Q$ is compactly supported in $M_0^{\rm int}$, this implies the required result.

It remains to show that a frame $\{ \eta^I \}$ with the required properties exists. Let $\omega = \gamma(0)$, and let $(\hat{M}_0,g_0)$ be a simple manifold with $(M_0, g_0) \subset \subset (\hat{M}_0,g_0)$ such that the $\hat{M}_0$-geodesic starting at $\omega$ in direction $\nu(\omega)$ never meets $M_0$. (It is enough to embed $(M_0,g_0)$ in some closed manifold and to take $\hat{M}_0$ strictly convex and slightly larger than $M_0$.) Let $(r,\theta)$ be polar normal coordinates in $\hat{M}_0$ with center $\omega = \gamma(0)$, fix $r_0 > 0$ so that the geodesic ball $B(\omega,r_0)$ is contained in $\hat{M}_0^{\rm int}$, and let $\hat{\theta} \in S^{n-2}$ be the direction of $\nu(\omega)$. Choose some orthonormal frame $\eta^3, \ldots, \eta^n$ of the cotangent space of $\partial B(\omega,r_0) \setminus \{(r_0,\hat{\theta})\}$, and extend these as $1$-forms in $M_0^{\rm int}$ by parallel transporting along integral curves of $\partial_r$. We thus obtain a global orthonormal frame $\eta^2, \ldots, \eta^n$ of $T^* M_0^{\rm int}$ with $\eta^2 = dr$ and $\nabla_{\partial_r} \eta^j = 0$ for $3 \leq j \leq n$. Moreover, $\eta^1, \ldots, \eta^n$ will be a global orthonormal frame of $T^* (\mR \times M_0^{\rm int})$ inducing an orthonormal frame $\{ \eta^I \}$ of $\Lambda(\mR \times M_0^{\rm int})$.
\end{proof}

We will now show how the coefficients are uniquely determined by the integrals in Proposition \ref{prop_transform_matrixelements}. This follows by inverting attenuated ray transforms, a topic of considerable independent interest (see the survey \cite{finch} for results in the Euclidean case, and the survey \cite{PSU_survey} and references below for the manifold case). The transform in Proposition \ref{prop_transform_matrixelements} is not exactly the same kind of attenuated ray transform/Fourier transform as in the scalar case for instance in \cite{DKSaU}, since the matrix element of $Q$ that appears in the integral may actually depend on the geodesic $\gamma$ (note that the $1$-forms $\eta$ depend on $\gamma$). To clarify this point, we fix some global orthonormal frame $\{ \eps^1, \ldots, \eps^n \}$ of $T^*(\mR \times M_0)$ with $\eps^1 = dx^1$, and let $\{ \eps^I \}$ be the corresponding orthonormal frame of $\Lambda(\mR \times M_0)$. Define the matrix elements 
$$
q_{I,J} = \langle Q \eps^I, \eps^J \rangle.
$$
Define also 
$$
\hat{q}_{I,J}(\xi_1,x') = \int_{-\infty}^{\infty} e^{-ix_1 \xi_1} q_{I,J}(x_1,x') \,dx_1.
$$
Then the conclusion in Proposition \ref{prop_transform_matrixelements} implies that 
\[
\int_0^L e^{-2\lambda r} \hat{q}_{I',J'}(2\lambda,\gamma(r)) \langle \eta^I, \eps^{I'} \rangle \langle \eta^J, \eps^{J'} \rangle \,dr = 0
\]
for any $\lambda \in \mR$, for any $I, J$, and for any maximal geodesic $\gamma$ in $M_0$. (Note that the inner products $\langle \eta^I, \eps^{I'} \rangle$ do not depend on $x_1$.)

Up until now everything discussed in this paper has held for any dimension $n \geq 3$.  Now, however, we will invoke the assumption that $n = 3$.  Then $q_{I,J}$ is an $8 \times 8$ matrix.  In this case we may choose $\eta^1 = dx^1$, $\eta^2 = dr$, and $\eta^3 = *_{g_0} dr$, where $dr$ is the $1$-form dual to $\dot{\gamma}$ on the geodesic $\gamma$. Let also $\{ e_j \}$ be the orthonormal frame of vector fields dual to $\{ \eps^j \}$ (which is assumed to be positively oriented). It follows that 
\begin{gather*}
\langle \eta^1, \eps^1 \rangle = 1, \quad \langle \eta^1, \eps^2 \rangle = 0, \quad \langle \eta^1, \eps^3 \rangle = 0, \\
\langle \eta^2, \eps^1 \rangle = 0, \quad \langle \eta^2, \eps^2 \rangle = \langle e_2, \dot{\gamma} \rangle, \quad \langle \eta^2, \eps^3 \rangle = \langle e_3, \dot{\gamma} \rangle, \\
\langle \eta^3, \eps^1 \rangle = 0, \quad \langle \eta^3, \eps^2 \rangle = -\langle e_3, \dot{\gamma} \rangle, \quad \langle \eta^3, \eps^3 \rangle = \langle e_2, \dot{\gamma} \rangle.
\end{gather*}
The relations for $\eta^{\{1,2\}} = \eta^1 \wedge \eta^2, \eta^{\{3,1\}},\eta^{\{2,3\}}$ and $\eps^{\{1,2\}}, \eps^{\{3,1\}},\eps^{\{2,3\}}$ can be determined from the above relations by duality. Finally, $\langle \eta^0, \eps^I \rangle = 1$ if $I=0$ and $0$ otherwise, and the other relations for $\eta^0, \eps^0$, $\eta^{\{1,2,3\}}$, and $\eps^{\{1,2,3\}}$ are similar.

Now choosing $I = J = 1$ (here we identify $1$ with $\{1\}$) we obtain 
$$
\int_0^L e^{-2\lambda r} \hat{q}_{1,1}(2\lambda,\gamma(r)) \,dr = 0 \quad \text{for all $\lambda$ and $\gamma$.}
$$
This means that the usual attenuated geodesic ray transform of the function $\hat{q}_{1,1}(2\lambda,\,\cdot\,)$ in $M_0$ vanishes for all $\lambda$. First we have $\hat{q}_{1,1}(2\lambda,\,\cdot\,) \in C^{\infty}(M_0)$ for all $\lambda$ \cite[Proposition 3]{FSU}, and then $\hat{q}_{1,1}(2\lambda,\,\cdot\,) = 0$ for all $\lambda$ by the injectivity of the attenuated ray transform \cite{SaU} and so $q_{1,1} = 0$. The same argument applies for all pairs $(I,J)$ where 
\[
I,J \in \{ 0,1,\{2,3\}, \{1,2,3\}\}.
\]

Now consider the case where $I = 1$ and $J = 2$.  Then 
\[
\int_0^L e^{-2\lambda r} (\hat{q}_{1,2}(2\lambda,\gamma(r)) \langle e_2, \dot{\gamma} \rangle + \hat{q}_{1,3}(2\lambda,\gamma(r)) \langle e_3, \dot{\gamma} \rangle) \,dr = 0.
\]
Then the injectivity result for the attenuated ray transform on 1-tensors \cite{SaU} together with the regularity result \cite[Proposition 1]{HS} says that 
\[
\hat{q}_{1,2}(2\lambda,x) \eps^2 + \hat{q}_{1,3}(2\lambda,x) \eps^3 = 0
\]
for all $\lambda \neq 0$, from which we can conclude that
\[
q_{1,2} = q_{1,3} = 0.
\]
The same argument then applies for all pairs $(I,J)$ where $I \in \{ 0,1,\{2,3\}, \{1,2,3\}\}$ and $J \in \{2,3,\{1,2\},\{3,1\}\}$, or vice versa.

Finally, consider the case when $I = J = 2$.  For brevity, we'll write $\langle e_j, \dot{\gamma} \rangle$ as $\dot{\gamma}_j$.  Then $I=J=2$ gives
\begin{equation}\label{IJ22}
\int_0^L e^{-2\lambda r} (\hat{q}_{2,2} \dot{\gamma}_2^2 + \hat{q}_{2,3}\dot{\gamma}_2 \dot{\gamma}_3 + \hat{q}_{3,2}\dot{\gamma}_3 \dot{\gamma}_2 + \hat{q}_{3,3}\dot{\gamma}_3^2) \,dr = 0.
\end{equation}
The integrand here can be represented as the symmetric 2-tensor
\[
f^{2,2} := \left( \begin{array}{cc}
\hat{q}_{2,2}                        & \half(\hat{q}_{2,3} + \hat{q}_{3,2}) \\
\half(\hat{q}_{2,3} + \hat{q}_{3,2}) & \hat{q}_{3,3}                        \\
\end{array} \right)
\]
(in coordinates provided by $\{\eps^2, \eps^3\}$) applied to $(\dot{\gamma}, \dot{\gamma})$. This shows that the attenuated ray transform of the $2$-tensor $f^{2,2}$ in $(M_0,g_0)$, with constant attenuation $-2\lambda$, vanishes identically.

We will now make use of the methods of \cite{PSaU} in this tensor tomography problem. We only give the details in the case where $Q$ (and hence $f^{2,2}$) is $C^{\infty}$. The result also holds for continuous $Q$ by using an elliptic regularity result for the normal operator, but in the present weighted case for 2-tensors the required result may not be in the literature. We only say that such a result can be proved by adapting the methods of \cite{HS} to the 2-tensor case (in particular one needs a solenoidal decomposition $f = f^s + d\beta$ of a $2$-tensor $f$ and a further solenoidal decomposition $\beta = \beta^s + d\phi$ of the $1$-form $\beta$, and one then shows that the normal operator acting on ''solenoidal triples'' $(f^s, \beta^s, \phi)$ is elliptic because the weight comes from a nonvanishing attenuation).

Since $f^{2,2}$ is $C^{\infty}$, the injectivity result for the attenuated ray transform on symmetric 2-tensors (see \cite{Assylbekov}, following \cite{PSaU}) says that
\[
f^{2,2} = -Xu + 2\lambda u
\]
where $X$ is the geodesic vector field on $(M_0,g_0)$, and $u$ is a smooth function on the unit circle bundle $SM_0$ that corresponds to the sum of a 1-tensor and scalar function, with 
\[
u|_{\partial M_0} = 0.
\]
Here we have identified $f^{2,2}$ and $u$ with functions on $SM_0$ as in ~\cite{PSaU}. 
We can also express $u$ and $f^{2,2}$ in terms of Fourier components as in ~\cite{PSaU},
\begin{equation*}
\begin{split}
u &= u_{-1} + u_0 + u_{1} \\
f^{2,2} &= f^{2,2}_{-2} + f^{2,2}_0 + f^{2,2}_{2}. \\
\end{split}
\end{equation*}
Here $u_0 \in C^{\infty}(M_0)$, $u_1 + u_{-1}$ corresponds to a smooth $1$-tensor in $M_0$, and $u_0, u_1, u_{-1}$ vanish on $\partial M_0$. Then 
\[
-X(u_{-1} + u_0 + u_{1}) + 2\lambda (u_{-1} + u_0 + u_{1}) = f^{2,2}_{-2} + f^{2,2}_0 + f^{2,2}_{2}.
\]
Now parity implies the following two equations: 
\[
2\lambda (u_{-1} + u_{1}) = Xu_0
\]
and
\[
-X(u_{-1} + u_{1}) + 2\lambda (u_0) = f^{2,2}_{-2} + f^{2,2}_0 + f^{2,2}_{2}.
\]
Assume that $\lambda$ is non-zero. 
Using the first equation in the second one implies that 
\begin{equation} \label{f22_equation}
-\frac{X^2(u_0)}{2\lambda} + 2\lambda u_0 = f^{2,2},
\end{equation}
where $X^2u_0$ corresponds to the covariant Hessian $\nabla^2 u_0$ of $u_0$. 
The first equation implies that $u_0$ vanishes to first order on $\partial M_0$.

Unfortunately, this is not enough to conclude that the coefficients of $f^{2,2}$ are $0$.  However, going back and choosing $(I,J) = (2,3), (3,2),$ and $(3,3)$ gives us three additional equations of this type with the same elements $q_{I,J}$.  More specifically, 
\[
f^{2,3} = \left( \begin{array}{cc}
\hat{q}_{2,3}                        & \half(\hat{q}_{3,3} - \hat{q}_{2,2}) \\
\half(\hat{q}_{3,3} - \hat{q}_{2,2}) & -\hat{q}_{3,2}                       \\
\end{array} \right),
\] 
\[
f^{3,2} = \left( \begin{array}{cc}
\hat{q}_{3,2}                        & \half(\hat{q}_{3,3} - \hat{q}_{2,2}) \\
\half(\hat{q}_{3,3} - \hat{q}_{2,2}) & -\hat{q}_{2,3}                        \\
\end{array} \right),
\] 
and
\[
f^{3,3} = \left( \begin{array}{cc}
\hat{q}_{3,3}                         & -\half(\hat{q}_{2,3} + \hat{q}_{3,2})  \\
-\half(\hat{q}_{2,3} + \hat{q}_{3,2}) & \hat{q}_{2,2} \\
\end{array} \right).
\] 
are all of the same form.  Therefore it follows that $f^{2,2} + f^{3,3}$ and $f^{2,3} - f^{3,2}$ are as well.  But these are both scalar matrices, and if 
\[
-\frac{X^2(u_0)}{2\lambda} + 2\lambda u_0 
\]
is a scalar matrix, then also the covariant Hessian $\nabla^2 u_0$ is a scalar matrix in the $\{ \eps^2, \eps^3 \}$ basis.

To make the previous statement more explicit, identify $(M_0,g_0)$ with the unit disk in $\mR^2$ and choose an isothermal coordinate system $(x^1, x^2)$ in which the metric is given by $e^{2\mu} \delta_{jk}$ for some $\mu \in C^{\infty}(M_0)$. Choosing $e_2 = e^{-\mu} \partial_1$ and $e_3 = e^{-\mu} \partial_2$, the condition $\nabla^2 u_0(e_2,e_2) - \nabla^2 u_0(e_3,e_3) = 0$ implies that 
$$
\partial_1^2 u_0 - \partial_2^2 u_0 + b \cdot \nabla u_0 = 0 \quad \text{in } M_0
$$
for some vector field $b \in C^{\infty}(M_0,\mR^2)$ depending on $\mu$. Since $u_0$ vanishes to first order on $\partial M_0$, extending $u_0$ by zero to $\mR^2$ we have 
$$
\partial_1^2 u_0 - \partial_2^2 u_0  + b \cdot \nabla u_0 = 0 \quad \text{in } \mR^2
$$
where $u_0 \in H^2(\mR^2)$ is compactly supported and $b$ is some smooth compactly supported vector field. Uniqueness for hyperbolic equations \cite[Section 2.8]{T1} implies that $u_0 = 0$.

The above argument shows that $f^{2,2} + f^{3,3}$ and $f^{2,3} - f^{3,2}$ are $0$. Thus $\hat{q}_{2,2} + \hat{q}_{3,3} = 0$ and $\hat{q}_{2,3} - \hat{q}_{3,2} = 0$, showing that $f^{2,2}$ and $f^{2,3}$ are trace free. Taking traces in \eqref{f22_equation} and using that $u_0$ vanishes to first order on $\partial M_0$ implies that $u_0 = 0$ by unique continuation for elliptic equations. Thus $f^{2,2} = 0$ and similarly $f^{2,3} = 0$, which shows that $q_{2,2}, q_{2,3}, q_{3,2},$ and $q_{3,3}$ are zero as well.

The same argument now works for the remaining entries of $q$, and this finishes the proof.

\section{Higher dimensions} \label{sec:higher}

In higher dimensions, $n > 3$, as noted above, everything up to and including the proof of Proposition \ref{prop_transform_matrixelements} still holds.  However, this does not reduce easily into a tensor tomography problem, as in the three-dimensional case, because we cannot choose the basis $\{ \eta^i \}$ so that $\eta^3, \ldots, \eta^4$ to depend on $\eta^2 = dr$ in a tensorial manner.  

More precisely, in general we lack tensors $T_i$ for which $\eta^i = T_i(\eta^2, \ldots, \eta^2)$ for $i \geq 3$, as is the case in three dimensions.  Moreover, even if the results of Proposition \ref{prop_transform_matrixelements} can be reduced to a tensor tomography problem, there is no guarantee that it will be one for which there are useful injectivity results, since there are very few such results for $k$-tensors with $k > 2$.  

However, in the Euclidean case we can do better, since we have the extra freedom to vary the Carleman weight $\ph$.  In particular, we can construct CGOs to reduce the problem in Lemmas \ref{RBCtoDensity} and \ref{ABCtoDensity} to a Fourier transform, as has been done for inverse problems for scalar functions, e.g. in ~\cite{BU}.  Therefore we can conclude this paper by a proof for higher dimensions, in the Euclidean case.  

\begin{proof}[Proof of Theorem \ref{EuclideanThm}]

Fix coordinates $x_1, \ldots, x_n$ on $\Rn$.  The corresponding basis for the cotangent space is $dx^1, \ldots, dx^n$, and this gives a corresponding basis $\{ dx^{I} \}$ for $\Lambda M$.  

Now note that if $f$ is a scalar function, $\Lap (f dx^I) = (\Lap f) dx^I$.  Therefore if $\alpha$ and $\beta$ are unit vectors such that $\alpha \cdot \beta = 0$, then
\[
e^{-\frac{\alpha \cdot x}{h}}h^2(-\Lap + Q)(e^{\frac{(\alpha + i\beta)\cdot x}{h}}dx^I) = O(h^2)dx^I.
\]
Therefore Proposition \ref{RBCHahnBanach} implies there exists $r \in L^2(M, \Lambda M)$ such that 
\[
(-\Lap + Q)(e^{\frac{(\alpha + i\beta) \cdot x}{h}}(dx^I + r)) = 0,
\]
with $\|r\|_{L^2(M)} = O(h)$, and $Z = e^{\frac{(\alpha + i\beta)\cdot x}{h}}(dx^I + r)$ has relative boundary conditions which vanish on $\Gamma_+^c$.    

Now if $k$ and $\ell$ are mutually orthogonal unit vectors which are both orthogonal to $\alpha$, then we can set $\beta_1 = \ell + hk$ and $\beta_2 = \ell-hk$, and create $Z_1 = e^{\frac{(-\alpha + i\beta_1)\cdot x}{h}}(dx^I + r_1)$ and $Z_2 = e^{\frac{(\alpha + i\beta_2)\cdot x}{h}}(dx^I + r_2)$ so that $(-\Lap + Q_1)Z_1 = (-\Lap + Q_2)Z_2 = 0$, and $Z_1$ and $Z_2$ have relative boundary conditions that vanish on $\Gamma_{-}^c$ and $\Gamma_{+}^c$ respectively.  

Then Lemma \ref{RBCtoDensity}, together with the hypotheses of Theorem \ref{EuclideanThm} implies that 
\[
(Q_1 - Q_2|e^{-i2k \cdot x}) = 0.
\]
This can be done for any $k$ orthogonal to $\alpha$.  Since $\alpha$ can be varied slightly without preventing the relative boundary conditions of the solutions from vanishing on the correct set, this is in fact true for $k$ in an open set, from which we can conclude that $Q_1 = Q_2$ on $M$.  

The absolute boundary value version works similarly, with the appropriate change in the CGOs.  
\end{proof}

\section{Appendix}

This appendix contains the proofs of Lemmas \ref{lemma_xieta_product}--\ref{ConxnLaplace}. We will write $k$-forms in local coordinates as $u = u_I \eps^I$ where $\eps^1, \ldots, \eps^n$ is some frame of $1$-forms, the sum is over all $k$-combinations $I = \{ i_1, \ldots, i_k \} \subset \{ 1, \ldots, n \}$, and we write $\eps^I = \eps^{i_1} \wedge \ldots \wedge \eps^{i_k}$ if $i_1 < \ldots < i_k$. Often we will choose $\eps^j = dx^j$ and write $u = u_I \,dx^I$.

\begin{proof}[Proof of Lemma \ref{lemma_xieta_product}]
We will prove that 
$$
\xi \wedge i_{\xi} u + i_{\xi} (\xi \wedge u) = \abs{\xi}^2 u.
$$
The lemma will follow immediately from this by polarization. It is enough to prove the identity at a fixed point. Choose a positive orthonormal basis $\eps^1, \ldots, \eps^n$ of covectors such that $\xi = \abs{\xi} \eps^1$, and note that if $I = \{ i_1,\ldots,i_k \}$ with $i_1 < \ldots < i_k$, we have 
$$
\eps^1 \wedge i_{\eps^1} \eps^I = \left\{ \begin{array}{ll} \eps^I, & 1 \in I, \\ 0, & \text{otherwise} \end{array} \right.
$$
and 
$$
i_{\eps^1} (\eps^1 \wedge \eps^I) = \left\{ \begin{array}{ll} \eps^I, & 1 \notin I, \\ 0, & \text{otherwise}. \end{array} \right.
$$
This proves the result.
\end{proof}

\begin{proof}[Proof of Lemma \ref{lemma_conjugated_hodge_expression}]
Note that 
$$
e^{s\rho} d(e^{-s\rho} u) = (d - s \,d\rho \wedge)u
$$
and by duality 
$$
e^{s\rho} \delta (e^{-s\rho} u) = (\delta + s \,i_{d\rho})u.
$$
It follows that 
\begin{align*}
 &e^{s\rho} (-\Delta) (e^{-s\rho} u) = e^{s\rho} (d\delta + \delta d) (e^{-s\rho} u) \\
 &= (d - s \,d\rho \wedge)(\delta u + s \,i_{d\rho} u) + (\delta + s \,i_{d\rho})(du - s \,d\rho \wedge u) \\
 &= -s^2 \left[ d\rho \wedge i_{d\rho} + i_{d\rho} (d\rho \wedge \,\cdot\,) \right] u \\
 & \quad \,+ s \left[ i_{d\rho} \circ d + d \circ i_{d\rho} - d\rho \wedge \delta - \delta(d\rho \wedge \,\cdot\,) \right] u - \Delta u.
\end{align*}
Writing $\rho = \phi + i\psi$ where $\phi$ and $\psi$ are real valued, Lemma \ref{lemma_xieta_product} implies that 
$$
d\rho \wedge i_{d\rho} u + i_{d\rho} (d\rho \wedge u) = \langle d\rho, d\rho \rangle u.
$$

It remains to show that 
$$
i_{d\rho} \circ d + d \circ i_{d\rho} - d\rho \wedge \delta - \delta(d\rho \wedge \,\cdot\,) = 2\nabla_{\text{grad}(\rho)} + \Delta \rho.
$$
To prove this, we compute in Riemannian normal coordinates at a fixed point $p$. Then $u$ has the coordinate expression $u = u_I \,dx^I$. We have 
$$
i_{d\rho}(du)|_p = \partial_j u_I \partial_k \rho \,i_{dx^k}(dx^j \wedge dx^I)|_p
$$
and 
\begin{align*}
d(i_{d\rho} u)|_p &= d( u_I \partial_k \rho \,i_{dx^k} dx^I)|_p \\
 &= \partial_j( u_I \partial_k \rho ) dx^j \wedge i_{dx^k} dx^I + u_I \partial_k \rho \,d(i_{dx^k} dx^I)|_p \\
 &= \partial_j u_I \partial_k \rho \,dx^j \wedge i_{dx^k} dx^I + u_I \partial_{jk} \rho \,dx^j \wedge i_{dx^k} dx^I|_p
\end{align*}
since $d(i_{dx^k} dx^I)|_p = d(i_{\partial_k} dx^I)|_p = 0$. Further, using that 
$$
\delta u|_p = -\sum_{j=1}^n i_{\partial_j} \nabla_{\partial_j} u|_p
$$
we have 
$$
-d\rho \wedge \delta u|_p = \partial_k \rho \,dx^k \wedge i_{\partial_j} \nabla_{\partial_j} (u_I \,dx^I)|_p = \partial_j u_I \partial_k \rho \,dx^k \wedge i_{\partial_j} dx^I|_p
$$
and 
\begin{align*}
-\delta(d\rho \wedge u)|_p &= i_{\partial_j} \nabla_{\partial_j} (u_I \partial_k \rho \,dx^k \wedge dx^I)|_p \\
 &= \partial_j u_I \partial_k \rho \,i_{\partial_j}(dx^k \wedge dx^I) + u_I \partial_{jk} \rho \,i_{\partial_j}(dx^k \wedge dx^I)|_p.
\end{align*}
Combining these computations with Lemma \ref{lemma_xieta_product}, we see that 
$$
 i_{d\rho} \circ du + d \circ i_{d\rho} u - d\rho \wedge \delta u - \delta(d\rho \wedge u)|_p = 2 \partial_j u_I \partial_j \rho \,dx^I + \partial_j^2 \rho \,u_I \,dx^I.
$$
The right hand side is $2\nabla_{{\rm grad}(\rho)} u + (\Delta \rho) u|_p$.
\end{proof}

\begin{proof}[Proof of Lemma \ref{Tdelta}]
First examine $-t\delta u_{\perp}$.  By using the expansion of $\delta$ in coordinates, at a point $p \in \partial M$,
\[
-t\delta u_{\perp} = t i_N \grad_N u_{\perp} + t \sum_{j=1}^{n-1} i_{e_j} \grad_{e_j} u_{\perp},
\]
where $\{e_j\}$ are a basis for $T\partial M$ at $p$.
Since $N$ is a geodesic vector field, $\grad_N N^{\flat} = 0$, and thus
\[
i_N \nabla_N u_\perp = \nabla_N i_N u_\perp.
\]
Therefore
\begin{equation}\label{firstperp}
-t\delta u_{\perp} = t \grad_N i_N u_{\perp} + t \sum_{j=1}^{n-1} i_{e_j} \grad_{e_j} u_{\perp},
\end{equation}
Meanwhile, if $\{X_1, \ldots, X_{k-1}\}$ are local sections of $T\partial M$ near $p$,
\begin{eqnarray*}
& & t\sum_{j=1}^{n-1} i_{e_j}\nabla_{e_j}u_\perp ( X_1,..,X_{k-1}) \\
&=& -t\sum_{j=1}^{n-1} \left( u_\perp(\nabla_{e_j}e_j, X_1,..,X_{k-1}) +\sum_{l=1}^{k-1} u_\perp(e_j,  X_1,..,\nabla_{e_j}X_l,..,X_{k-1}) \right)\\
&=& -t\sum_{j=1}^{n-1} \big( u_\perp(\nabla'_{e_j}e_j + II(e_j, e_j), X_1,..,X_{k-1}) \\
& & +\sum_{l=1}^{k-1} u_\perp(e_j,  X_1,..,\nabla'_{e_j}X_l + II(e_j, X_l),..,X_{k-1}) \big),
\end{eqnarray*}
where $II : T\partial M \oplus T\partial M \to N\partial M$ is the second fundamental form of $\partial M$ relative to its embedding in $M$.  

Since all of  $e_j$ and $X_k$ in the above sum belong to $T\partial M$ and $u_\perp$ needs an element of $N\partial M$ in its argument to survive, we have that 
\[
u_\perp(e_j,  X_1,..,\nabla'_{e_j}X_l + II(e_j, X_l),..,X_{k-1})) =  u_\perp(e_j,  X_1,.., II(e_j, X_l),..,X_{k-1})).
\]
Therefore, $t\sum_{j=1}^{n-1} i_{e_j}\nabla_{e_j}u_\perp ( X_1,..,X_{k-1})$ is equal to 
\[
-t\sum_{j = 1}^{n-1} \left( u_\perp(II(e_j, e_j), X_1,..,X_{k-1}) +\sum_{l = 1}^{k-1} u_\perp(e_j,  X_1,..,II(e_j, X_l),..,X_{k-1}) \right).
\]
Then
\begin{eqnarray*}
& & -\sum\limits_{j = 1}^{n-1} \left( u_\perp(II(e_j, e_j), X_1,..,X_{k-1}) +\sum\limits_{l = 1}^{k-1} u_\perp(e_j,  X_1,..,II(e_j, X_l),..,X_{k-1}) \right) \\
&=& -\sum\limits_{j = 1}^{n-1} u_\perp(II(e_j, e_j), X_1,..,X_{k-1})+ \sum\limits_{l = 1}^{k-1}\sum\limits_{j = 1}^{n-1}u_\perp(N,  X_1,..,e_j (N|II(e_j, X_l)),..,X_{k-1})) \\
&=& -(n-1)\kappa u_\perp(N, X_1,..,X_{k-1})+ \sum\limits_{l = 1}^{k-1}\sum\limits_{j = 1}^{n-1}u_\perp(N,  X_1,..,e_j (sX_l | e_j),..,X_{k-1})) \\
&=&-(n-1)\kappa u_\perp(N, X_1,..,X_{k-1}) + \sum\limits_{l =1}^{k-1} u_\perp(N,  X_1,..,s X_l,..,X_{k-1}))\\ 
\end{eqnarray*}
Therefore
\[
t\sum_{j=1}^{n-1} i_{e_j}\nabla_{e_j}u_\perp = (S-(n-1)\kappa) t i_N u_\perp.
\]
Using this together with \eqref{firstperp} gives
\begin{equation}\label{secondperp}
-t\delta u_{\perp} = t \grad_N i_N u_{\perp} + (S-(n-1)\kappa)  t i_N u_\perp.
\end{equation}

Now consider $t\delta u_{\|}$.  Again, using the expansion of $\delta$ in coordinates,
\[
-t\delta u_{||} = t\sum\limits_{j= 1}^n i_{e_j} \nabla_{e_j} u_{||}  = ti_{N} \nabla_{N} u_{||} + t\sum\limits_{j= 1}^{n-1} i_{e_j} \nabla_{e_j} u_{||}.
\]
Since $N$ is a geodesic vector field, $\grad_N N = 0$.  Moreover, $i_N u_{\|} =0$ by the definition of $u_{\|}$.  Therefore $i_N \grad_N u_{\|} = 0$.  Then for the remaining sums, we let $\{X_1,..X_{k-1}\}$ be sections of $T\partial M$ and compute
\begin{eqnarray*}
& & \sum\limits_{j= 1}^{n-1} (i_{e_j} \nabla_{e_j} u_{||}) (X_1,..,X_{k-1}) \\
&=& \sum\limits_{j= 1}^{n-1} \big(e_j  u_{||}(e_j, X_1,..,X_{k-1}) - u_{||}(\nabla_{e_j}e_j, X_1,..,X_{k-1}) \\
& & -\sum\limits_{l = 1}^{k-1} u_{||}(e_j, X_1, .., \nabla_{e_j} X_l,.., X_{k-1})\big). 
\end{eqnarray*}

We have that  $\nabla_{e_j} X = \underbrace{\nabla'_{e_j} X}_{\in T\partial M} + \underbrace{II(e_j, X)}_{\in N\partial M}$ for all $X\in T\partial M$. Using this formula and the fact that $i_N u_{||} =0$ we have that
\begin{eqnarray*}
& & \sum\limits_{j=1}^{n-1} (i_{e_j} \nabla_{e_j} u_{||}) (X_1,..,X_{k-1}) \\
&=& \sum\limits_{j=1}^{n-1} \big(e_j  (u_{||}(e_j, X_1,..,X_{k-1})) - u_{||}(\nabla'_{e_j}e_j, X_1,..,X_{k-1}) - \\
& & -\sum\limits_{l=1}^{k-1} u_{||}(e_j, X_1, .., \nabla'_{e_j} X_l,.., X_{k-1})\big) \\
&=& \sum\limits_{j=1}^{n-1} (i_{e_j} \nabla'_{e_j} u_{||}) (X_1,..,X_{k-1}) \\
&=& (-\delta' t u_{||})(X_1,..,X_{k-1}). 
\end{eqnarray*}

Therefore
\[
t\delta u_{||} = \delta' tu_{||}.
\]
Combining this with \eqref{secondperp} finishes the proof of the lemma.
\end{proof}

\begin{proof}[Proof of Lemma \ref{iNdu}]
As in the previous proof, we'll begin by examining the $u_{\perp}$ part.  Using the expansion of $d$ in terms of sections, if $\{X_1, \ldots, X_k\}$ are local sections of $T \partial M$, 
\[
i_{N}d u_{\perp} (X_1, \ldots, X_k) = \sum_{j=1}^k (-1)^j \nabla_{X_j} u_\perp (N, X_1,..,\hat X_j,..,X_k) + \nabla_N u_\perp (X_1,..,X_k)
\]
Here we have identified $X_j$ with its extension to a neighbourhood of $\partial M$ in the interior of $M$ by parallel transport along normal geodesics, so that $(N|X_j)$ vanishes identically.  Then the last term becomes
\[
\nabla_N u_\perp (X_1,..,X_k) = \underbrace{N (u_\perp (X_1,..,X_k))}_{=0} - \sum\limits_{l=1}^k u_\perp(X_1,..,\underbrace{\nabla_{N}X_l}_{=0},..X_k) = 0.
\]

For the first term, we can use the Leibnitz rule to get
\begin{eqnarray*}
& & \sum\limits_{j =1}^k (-1)^j \nabla_{X_j} u_\perp (N, X_1,..,\hat X_j,..,X_k) \\
&=& \sum\limits_{j =1}^k (-1)^j\big( X_j (u_\perp (N, X_1,..,\hat X_j,..,X_k)) - u_\perp(\nabla_{X_j} N, X_1,..,\hat X_j,.., X_k) \\
& & -\sum\limits_{l =1}^k u_\perp(N,X_1,..,\nabla_{X_j} X_l,..,\hat X_j,..,X_k)\big)
\end{eqnarray*}
For all $j$, $u_\perp(\nabla_{X_j} N, X_1,..,\hat X_j,.., X_k) = u_\perp((\nabla_{X_j} N)_\perp, X_1,..,\hat X_j,.., X_k)$ since one needs a normal component in the argument for the expression to not vanish. However, 
\[
0 = X_j \underbrace{(N|N)}_{=1} = 2(\nabla_{X_j} N, N)
\]
so
\[
u_\perp(\nabla_{X_j} N, X_1,..,\hat X_j,.., X_k)= 0
\]
 for all $j$. On the other hand, for all $l$ and $j$,
\begin{eqnarray*}
& & u_\perp(N,X_1,..,\nabla_{X_j} X_l,..,\hat X_j,..,X_k) \\
&=& u_\perp(N,X_1,..,(\nabla_{X_j} X_l)_{||},..,\hat X_j,..,X_k) \\
&=& u_\perp(N,X_1,..,\nabla'_{X_j} X_l,..,\hat X_j,..,X_k)
\end{eqnarray*}
since the expression only allows one normal component in its argument. Therefore,
\begin{eqnarray*}
& & \sum\limits_{j=1}^k (-1)^j \nabla_{X_j} u_\perp (N, X_1,..,\hat X_j,..,X_k) \\
&=& -\sum\limits_{j=1}^k (-1)^{j+1}\nabla'_{X_j} i_N u_\perp(X_1,..,\hat X_j,..,X_k) \\
&=& -d' t i_N u_\perp
\end{eqnarray*}

Therefore
\[
t i_{N}d u_{\perp} = -d' t i_N u_\perp.
\]

Now examine the parallel part.  Suppose $\{X_1, \ldots, X_k\}$ are local sections of $T \partial M$.  By definition of $u_{\|}$,
\[
u_{||}(N,X_1,..,\nabla_{X_j} X_l, ..,X_k) = 0,
\]
so we can write
\begin{eqnarray*}
i_N du_{||} (X_1,..,X_k) &=& \sum_{j=1}^k (-1)^j \nabla_{X_j}u_{||}(N, X_1,..,\hat X_j,..,X_k) + \nabla_N u_{||}(X_1,..,X_k)\\
                         &=&-\sum_{j=1}^k (-1)^j u_{||}(\nabla_{X_j}N,..,\hat X_j, ..,X_k) + \nabla_N u_{||}(X_1,..,X_k)\\
                         &=&-\sum_{j=1}^k (-1)^j u_{||}((\nabla_{X_j}N)_{||},..,\hat X_j, ..,X_k) + \nabla_N u_{||}(X_1,..,X_k) \\
                         &=& \sum_{j=1}^k u_{||}(X_1,..,s X_j, ..,X_k) + \nabla_{N} u_{||}(X_1,..,X_k)\\
\end{eqnarray*}

Therefore
\[
t i_N du_{||} = S tu_{||} + t \nabla_{N}u_{||}.
\]
Putting this together with the calculation for the normal part gives
\[
t i_N du = -d' t i_N u_\perp + Stu_{||} + t\nabla_{N}u_{||}
\]
Since $i_N u_{\perp} = i_N u$, 
\[
t i_N du = -d' t i_N u + Stu_{||} + t\nabla_{N}u_{||}
\]
as desired.
\end{proof}

To prove Lemma \ref{deltaBu}, we will first require the following intermediate lemma.  

\begin{lemma}\label{tnablaNiNnablagradf}
For all smooth functions $f$, and $u \in \Om^k(M)$ such that $tu =0$, we have that
\begin{eqnarray*}
t \nabla_N i_N\nabla_{\grad f}u &=& \nabla_{\grad f_{||}}'t \nabla_N i_N u - t i_{s \grad f_{||}} \nabla_N u_{||} +t i_N R(N, \nabla f_{||}) u_\perp \\
                                & &+ t \nabla_{[\grad f_{||}, N]}i_Nu-\partial_\nu f t\nabla_N \nabla_N i_N u + \partial^2_\nu f t\nabla_N i_N u. \\
\end{eqnarray*}
\end{lemma}

\begin{proof}
We first split $\nabla_{\grad f} $ into its normal and tangential parts: $\nabla_{\grad f} = -\partial_\nu f \nabla_{N} + \nabla_{\grad f_{||}}$ and 
\[
t \nabla_N i_N \nabla_{\grad f} u = - t \nabla_N \partial_\nu f i_N \nabla_{N} u + t \nabla_N i_N \nabla_{\grad f_{||}} u
\]
where $\grad f_{||}$ is defined by $\grad f = -(\partial _\nu f) N + \grad f_{||}$. We first compute the tangential derivative
\begin{eqnarray*} 
t\nabla_N i_N \nabla_{\grad f_{||}} u &=& ti_N\nabla_N \nabla_{\grad f_{||}} u\\
                                      &=&  ti_N\nabla_{\grad f_{||}} \nabla_{N} u +t i_N R(N, \nabla f_{||}) u+  ti_N\nabla_{[\grad f_{||}, N]} u \\
                                      &=& t \nabla_{\grad f_{||}} \nabla_N i_N u - t i_{\nabla_{\grad f_{||}} N} \nabla_N u +t i_N R(N, \nabla f_{||}) u+ t i_N \nabla_{[\grad f_{||}, N]}u
\end{eqnarray*}
We now observe that $t i_N \nabla_N u_{||} = 0$ and $t\nabla_N u_\perp =0$ so that we have 
\begin{equation*}
\begin{split}
t i_{\nabla_{\grad f_{||}}N} \nabla_N u_{||} &=  t i_{(\nabla_{\grad f_{||}}N)_{||}} \nabla_N u_{||}\\
                                             &= t i_{s\grad f_{||}} \nabla_N u_{||} \\
\end{split}
\end{equation*}
and 
\[
t i_{\nabla_{\grad f_{||}}N} \nabla_N u_\perp = t i_{(\nabla_{\grad f_{||}}N)_{\perp}} \nabla_N u_\perp = 0.
\] 
Plugging all this into the equality above we have
\begin{equation} \label{tangential part of double nabla u}
\begin{split}
t\nabla_N i_N \nabla_{\grad f_{||}} u = &t \nabla_{\grad f_{||}} \nabla_N i_N u - t i_{s \grad f_{||}} \nabla_N u_{||} +t i_N R(N, \nabla f_{||}) u+ t i_N \nabla_{[\grad f_{||}, N]}u \\
                                      = &t \nabla_{\grad f_{||}} \nabla_N i_N u - t i_{s \grad f_{||}} \nabla_N u_{||} +t i_N R(N, \nabla f_{||}) u \\
                                        &+ t \nabla_{[\grad f_{||}, N]}i_Nu - t i_{\nabla_{[\grad f_{||}, N]}N} u.
\end{split}
\end{equation}
If $tu =0$, then the last term $t i_{\nabla_{[\grad f_{||}, N]}N} u = t i_{(\nabla_{[\grad f_{||}, N]}N)_{\perp}} u = 0$ since $[\grad f_{||}, N]\in T\partial M$ and $N$ is constant length. Furthermore, since $(\nabla_N i_N u)_\perp = 0$ for all $u$ we have that $ t \nabla_{\grad f_{||}} \nabla_N i_N u =\nabla_{\grad f_{||}}' t \nabla_N i_N u$. Therefore we get
\begin{equation*}
t\nabla_N i_N \nabla_{\grad f_{||}} u = \nabla_{\grad f_{||}}' t \nabla_N i_N u - t i_{s \grad f_{||}} \nabla_N u_{||}+t i_N R(N, \nabla f_{||}) u_\perp + t \nabla_{[\grad f_{||}, N]}i_Nu.
\end{equation*}
For the derivative in the normal direction we write 
\[
-\nabla_N \partial_\nu f i_N\nabla_N u = -\partial_\nu f \nabla_N \nabla_N i_N u + \partial^2_\nu f \nabla_N i_N u.
\]
Combining the last two equations finishes the proof.
\end{proof}

\begin{proof}[Proof of Lemma \ref{deltaBu}]
By Lemma \ref{Tdelta},
\[
t \delta Bu = \delta ' tBu - (S - (n-1)\kappa)ti_{N}(Bu) - t \grad_{N} i_N Bu.
\]
Now
\[
t\grad_N i_{N}(Bu) = -ih t \left( 2 \grad_N i_N\nabla_{\text{grad}(\varphi_c)} u + \grad_N \Delta \varphi_c i_N u \right) 
\]
and
\begin{equation*}
\begin{split}
ti_{N}(Bu) &= -ih ti_N \left( 2 \nabla_{\text{grad}(\varphi_c)} + \Delta \varphi_c \right)u \\
           &= -ih ti_N \left( 2 \nabla_{\text{grad}(\varphi_c)_{\|}}u - 2\partial_{\nu}\ph_c \nabla_{N}u + \Delta \varphi_c u \right) \\           
\end{split}
\end{equation*}
If $tu = 0$, then $t i_{\text{grad}(\varphi_c)_{\|}N}u = 0$, so
\[
ti_{N}(Bu) = -ih\left( 2 t \nabla_{\text{grad}(\varphi_c)_{\|}}i_N u - 2\partial_{\nu}\ph_c t \nabla_{N}i_N u + \Delta \varphi_c ti_N u \right).
\] 
Therefore
\begin{equation*}
\begin{split}
t \delta Bu = &\delta ' tBu + ih t \left( 2 \grad_N i_N\nabla_{\text{grad}(\varphi_c)} u + \grad_N \Delta \varphi_c i_N u \right)  \\
              &+ih(S -(n-1) \kappa)\left( 2 t \nabla_{\text{grad}(\varphi_c)_{\|}}i_N u - 2\partial_{\nu}\ph_c t \nabla_{N}i_N u + \Delta \varphi_c ti_N u \right) \\
\end{split}
\end{equation*}
Now by Lemma \ref{tnablaNiNnablagradf},
\begin{eqnarray*}
& & t \grad_N i_N\nabla_{\text{grad}(\varphi_c)} u \\
&=& \nabla_{\grad(\ph_c)_{||}}'t \nabla_N i_N u - t i_{s \grad(\ph_c)_{||}} \nabla_N u_{||}\\
& & + ti_N R(N,\grad(\ph_c)_{||} ) u_\perp+t \nabla_{[\grad(\ph_c)_{||}, N]}i_Nu-\partial_\nu \ph_c t\nabla_N \nabla_N i_N u +\partial^2_\nu \ph_c t\nabla_N i_N u.\\
\end{eqnarray*}
Substituting this into the previous equation finishes the proof.
\end{proof}

\begin{proof}[Proof of Lemma \ref{ConxnLaplace}]
First we see that for all $v\in \Omega^k(M)$ and $e_j, X_l \in T\partial M$ we have 
\[
(ti_N \nabla_{e_j} v)(X_1,..,X_{k-1})  = (\nabla_{e_j}' i_N v) (X_1,..,X_{k-1}) - v (\nabla_{e_j} N, X_1,..,X_{k-1})
\]
Since $(\nabla_{e_j} N)_\perp = 0$ and $(\nabla_{e_j} N)_{||} = se_j$ by Weingarten's identity, we have that 

\[ti_N \nabla_{e_j} v = \nabla_{e_j}' t i_N v_\perp - t i_{se_j}v_{||}.\]
Setting $v =  \nabla_{e_j} u$ we have for $e_j\in T\partial M$, $j = 1,..,n-1$
\begin{eqnarray*}
ti_N \nabla_{e_j} \nabla_{e_j} u &=&t i_N \nabla_{e_j}\big( (\nabla_{e_j} u)_\perp + (\nabla_{e_j} u)_{||}\big)\\
                                 &=& \nabla'_{e_j}\big( ti_N (\nabla_{e_j} u)_\perp) -i_{s e_j}(\nabla_{e_j} u)_{||}\big.
\end{eqnarray*}
Since $(\nabla_{e_j}N)_\perp =0$ we have that $ \nabla'_{e_j} ti_N (\nabla_{e_j} u)_\perp =\nabla'_{e_j} \nabla'_{e_j}(ti_N u)$ and therefore
\begin{eqnarray}
\label{laplace on boundary}
t i_N\sum\limits_{j =1}^{n-1} \nabla_{e_j} \nabla_{e_j} u = \tilde\Delta' t i_N u - \sum\limits_{j=1}^{n-1}i_{se_j}(\nabla_{e_j} u)_{||}
\end{eqnarray}
For the term involving $e_n$ we use $\nabla_N N = 0$ to obtain $i_N \nabla_{e_n}\nabla_{e_n} u = \nabla_{e_n} \nabla_{e_n} i_N u$. Combining this and (\ref{laplace on boundary}) we obtain 
\[t i_N \tilde\Delta u = \tilde \Delta' t i_Nu + t \nabla_{N}\nabla_N i_N u- \sum\limits_{j=1}^{n-1}ti_{se_j}(\nabla_{e_j} u)_{||}.\]
Let $p \in \partial M$  and since the above identity is true for any local orthonormal frame, we will choose $\{e_1,.., e_{n-1}\}$ so that they are eigenvectors of the shape operator $s$ at the point $p$ with $se_j = \lambda_j e_j$. For $X_{1,},..,X_{l} \in T\partial M_p$ we compute
\begin{eqnarray*}
t i_N \tilde\Delta u(X_1,..,X_{k-1}) &=& \tilde \Delta' t i_Nu(X_1,..,X_{k-1}) + t \nabla_{N}\nabla_N i_N u(X_1,..,X_{k-1})\\&-& \sum\limits_{j=1}^{n-1}i_{se_j}(\nabla_{e_j} u)_{||}(X_1,..,X_{k-1})\\
&=&\tilde \Delta' t i_Nu(X_1,..,X_{k-1}) + t \nabla_{N}\nabla_N i_N u(X_1,..,X_{k-1})\\&+& \sum\limits_{j=1}^{n-1} \big(u( \nabla_{e_j} se_j, X_1,..,X_{k-1}) +\sum\limits_{l=1}^{k-1}u(se_j,..,\nabla_{e_j}X_l,..)\big)\\
&=& \tilde \Delta' t i_Nu(X_1,..,X_{k-1}) + t \nabla_{N}\nabla_N i_N u(X_1,..,X_{k-1})\\&+& \sum\limits_{j=1}^{n-1} \big(u( \lambda_{j} \langle II(e_j,e_j),N\rangle N, X_1,..,X_{k-1}) \\ &+& \lambda_j\sum\limits_{l=1}^{k-1}u(e_j,..,\nabla_{e_j}X_l,..)\big) \\
\end{eqnarray*}
Since $tu = 0$, only the second fundamental form appears in $\nabla_{e_j}$ terms:
\begin{eqnarray*}
t i_N \tilde\Delta u(X_1,..,X_{k-1}) &=& \tilde \Delta' t i_Nu(X_1,..,X_{k-1}) + t \nabla_{N}\nabla_N i_N u(X_1,..,X_{k-1})\\&+& \sum\limits_{j=1}^{n-1} \big(\lambda_{j} \langle II(e_j,e_j),N\rangle u(  N, X_1,..,X_{k-1}) \\ &+& \lambda_j\sum\limits_{l=1}^{k-1}u(e_j,..,II(e_j, X_l),..)\big).\\
\end{eqnarray*}
This becomes
\begin{eqnarray*}
t i_N \tilde\Delta u(X_1,..,X_{k-1})
&=& \tilde \Delta' t i_Nu(X_1,..,X_{k-1}) + t \nabla_{N}\nabla_N i_N u(X_1,..,X_{k-1})\\&+& \sum\limits_{j=1}^{n-1} \big(\lambda_{j}^2u(  N, X_1,..,X_{k-1}) -\lambda_j\sum\limits_{l=1}^{k-1}u(N,..,\langle II(e_j, X_l),N\rangle e_j,..)\big)\\
&=& \tilde \Delta' t i_Nu(X_1,..,X_{k-1}) + t \nabla_{N}\nabla_N i_N u(X_1,..,X_{k-1})\\&+& \sum\limits_{j=1}^{n-1} \big(\lambda_{j}^2u(  N, X_1,..,X_{k-1}) -\lambda_j^2\sum\limits_{l=1}^{k-1}u(N,..,\langle e_j, X_l\rangle e_j,..)\big)\\
&=& \tilde \Delta' t i_Nu(X_1,..,X_{k-1}) + t \nabla_{N}\nabla_N i_N u(X_1,..,X_{k-1})\\&+& \sum\limits_{j=1}^{n-1} \big(\lambda_{j}^2u(  N, X_1,..,X_{k-1}) -\sum\limits_{l=1}^{k-1}u(N,..,\langle e_j, s^2X_l\rangle e_j,..)\big)\\
&=& \tilde \Delta' t i_Nu(X_1,..,X_{k-1}) + t \nabla_{N}\nabla_N i_N u(X_1,..,X_{k-1})\\&+&\big(tr(s^2)i_Nu( X_1,..,X_{k-1}) -\sum\limits_{l=1}^{k-1}i_Nu(..,s^2X_l,..)\big)
\end{eqnarray*}
Therefore we conclude that 
\[t i_N \tilde\Delta u = \tilde \Delta' t i_Nu + t \nabla_{N}\nabla_N i_N u+tr(s^2)i_Nu -S_2i_Nu\]
where $S_2 \omega (X_1,..,X_{k-1}) := \sum\limits_{l=1}^{k-1}\omega(.., s^2 X_l, ..)$.
\end{proof}
 \label{sec:appendix}

\bibliographystyle{alpha}

\end{document}